\documentclass[11pt]{amsart}
\usepackage[utf8]{inputenc}

\newcommand{\ma}{\mathcal{A}}
\usepackage{amsmath,amssymb}
\usepackage{comment}
\usepackage{amscd,eucal,amsthm}
\usepackage{pb-diagram}
\usepackage{tikz-cd}
\newcommand{\bc}{\mathbb{C}}
\newcommand{\br}{\mathbb{R}}
\newcommand{\bz}{\mathbb{Z}}
\newcommand{\bn}{\mathbb{N}}
\newcommand{\fg}{\mathfrak{g}}
\newcommand{\fn}{\mathfrak{n}}
\newcommand{\fh}{\mathfrak{h}}
\newcommand{\fk}{\mathfrak{k}}
\newcommand{\fb}{\mathfrak{b}}
\newcommand{\fz}{\mathfrak{z}}
\newcommand{\ft}{\mathfrak{t}}
\newcommand{\bp}{\mathbb{P}}
\newcommand{\ba}{\mathbb{A}}
\DeclareMathOperator{\diag}{\Delta}

\DeclareMathOperator{\rk}{rk}
\DeclareMathOperator{\spann}{span}
\DeclareMathOperator{\gr}{gr}
\DeclareMathOperator{\Id}{id}
\DeclareMathOperator{\fsl}{\mathfrak{sl}}
\DeclareMathOperator{\ev}{ev}
\DeclareMathOperator{\Hom}{Hom}
\DeclareMathOperator{\Res}{Res}
\newcommand{\ul}{{\underline{\lambda}}}
\newcommand{\uz}{{\underline{z}}}
\newcommand{\fgl}{\mathfrak{gl}}

\newcommand{\CB}{\mathcal{B}}
\newcommand{\CA}{\mathcal{A}}
\newcommand{\CO}{\mathcal{O}}
\newcommand{\C}{\mathbb{C}}
\newcommand{\CP}{\mathbb{CP}}
\newcommand{\Rees}{\operatorname{Rees}}
\newcommand{\Cx}{\bc^\times}
\newcommand{\PP}{{\mathbb P}}
\newcommand{\CS}{{\mathcal S}}
\newcommand{\tU}{{\widetilde U}}
\newcommand{\BG}{{\mathbb G}}
\newcommand{\CF}{{\mathcal F}}
\newcommand{\CU}{{\mathcal U}}
\newcommand{\tCU}{\widetilde \CU}
\newcommand{\BA}{{\mathbb A}}
\newcommand{\BC}{{\mathbb C}}
\newcommand{\fr}{{\mathfrak{r}}}
\newcommand{\frr}{{\widetilde \fr}}
\newcommand{\fs}{{\mathfrak{s}}}
\newcommand{\CW}{{\mathcal W}}

\newcommand{\Cf}{{\mathcal T}}
\newcommand{\CQ}{{\mathcal Q}}
\newcommand{\CG}{{\mathcal G}}
\newcommand{\CV}{{\mathcal V}}
\newcommand{\Gr}{{\mathrm{Gr}}}
\newcommand{\bM}{{\mathbf M}}
\newcommand{\End}{{\operatorname{End}}}

\newcommand{\sslash}{\mathbin{/\mkern-6mu/}}

\DeclareMathAlphabet\mathbfcal{OMS}{cmsy}{b}{n}

\newtheorem{lemma}{Lemma}[section]
\newtheorem{defn}[lemma]{Definition}
\newtheorem{prop}[lemma]{Proposition}
\newtheorem{thm}[lemma]{Theorem}
\newtheorem{lem}[lemma]{Lemma}
\newtheorem{corol}[lemma]{Corollary}
\newtheorem{ex}[lemma]{Example}
\newtheorem{conj}[lemma]{Conjecture}
\newtheorem{rem}[lemma]{Remark}

\newcommand*{\dupcntr}[2]{%
  \expandafter\let\csname c@#1\expandafter\endcsname\csname c@#2\endcsname
}
\makeatother
\author{Aleksei Ilin, Joel Kamnitzer and Leonid Rybnikov}

\title{Gaudin models and moduli space of flower curves}

\begin{document}
\begin{abstract} We introduce and study the family of trigonometric Gaudin subalgebras in $U\fg^{\otimes n}$ for arbitrary simple Lie algebra $\mathfrak{g}$. This is the family of commutative subalgebras of maximal possible transcendence degree that serve as a universal source for higher integrals of the trigonometric Gaudin quantum spin chain attached to $\mathfrak{g}$. We study the parameter space that indexes all possible degenerations of subalgebras from this family. In particular, we show that (rational) inhomogeneous Gaudin subalgebras of $U\fg^{\otimes n}$ previously studied in \cite{ffry} arise as certain limits of trigonometric Gaudin subalgebras. Moreover, we show that both families of commutative subalgebras glue together into the one parameterized by the space $\overline{\mathcal F}_n$, which is the total space of degeneration of the Deligne-Mumford space of stable rational curves $ \overline M_{n+2} $ to the moduli space of cactus flower curves $ \overline F_n $ recently introduced in \cite{iklpr}.
As an application, we show that trigonometric Gaudin subalgebras act on tensor products of irreducible finite-dimensional $\fg$-modules without multiplicities, under some explicit assumptions on the parameters in terms of two different real forms of $\overline M_{n+2}$. This gives rise to a monodromy action of the affine cactus group on the set of eigenstates for the trigonometric Gaudin model. We also explain the relation
between the trigonometric Gaudin model and the quantum cohomology of minuscule resolutions of affine Grassmannian slices.


\end{abstract}
\maketitle

\section{Introduction.}

\subsection{Gaudin models} Let $\fg$ be a complex simple Lie algebra and $U\fg$ be its universal enveloping algebra. We consider some remarkable families of large commutative subalgebras in $U\fg^{\otimes n}$ called \emph{Gaudin subalgebras}. Originally, in \cite{g1}, Gaudin introduced the generators of such subalgebras for $\fg=\fsl_2$ as a complete set of integrals of a useful degeneration of the Heisenberg quantum spin chain. In \cite{g2}, the Gaudin model was generalized to an arbitrary semisimple Lie algebra $\fg$. In \cite{ffre}, Feigin, Frenkel and Reshetikhin constructed higher integrals of this generalized model with the help of the critical level phenomenon for the affine Lie algebra $\hat{\fg}$. 

In this paper, we consider three versions of the Gaudin models: rational homogeneous, rational inhomogeneous, and trigonometric.  Our main goal is to study the compactified parameter space for each of these families of subalgebras.  For the homogeneous Gaudin model, this has been thoroughly studied in \cite{r}, so we will focus on the inhomogeneous and trigonometric models.

In order to explain these subalgebras, it will be necessary to introduce some notation.  Fix a triangular decomposition $\fg = \fn_+ \oplus \fh \oplus \fn_-$, where $\fh$ is a Cartan subalgebra.  Fix a non-degenerate invariant symmetric bilinear form $ (\cdot,\cdot) $ on $ \fg$ and let $ \Omega \in \fg^{\otimes 2}$ be the corresponding Casimir element.  Write $ \Omega = \Omega_+ + \Omega_0 + \Omega_-$ where $ \Omega_+ \in \fn_+ \otimes \fn_-$, $\Omega_0 \in \fh \otimes \fh$ and $ \Omega_- \in \fn_- \otimes \fn_+$.  


\begin{itemize}
    \item Let $ z_1, \dots, z_n \in \bc$ be distinct.  The \emph{homogeneous Gaudin subalgebra} is a maximal free commutative subalgebra $$\ma(z_1, \ldots, z_n) \subset (U\fg^{\otimes n})^{\fg} \subset U\fg^{\otimes n}$$ 
    which contains the Gaudin Hamiltonians $H_i$ for $ i = 1, \ldots, n$
$$H_i = \sum_{j \not  =i} 
\frac{\Omega^{(ij)}}{z_i - z_j}.
$$
where $ \Omega^{(ij)} $ denotes the result of inserting $ \Omega $ in the $ i, j$ tensor factors.

\item Let $z_1, \dots, z_n $ be as above and also let $ \chi \in \fg $.  The \emph{inhomogeneous Gaudin subalgebra} is a maximal free commutative subalgebra $$\ma_{\chi}(z_1, \ldots, z_n) \subset (U\fg^{\otimes n})^{\fz_{\fg}(\chi)} \subset U\fg^{\otimes n}$$ 
which contains the inhomogeneous Gaudin Hamiltonians $H_{i,\chi}$ for $ i = 1, \ldots, n$
$$ H_{i,\chi} = \chi^{(i)} + \sum_{j\ne i} \frac{\Omega^{(ij)}}{z_i - z_j}$$
(where $\chi^{(i)}$ denotes $\chi$ inserted to the $i$-th tensor factor) as well as the dynamical Gaudin Hamiltonians, see Section \ref{inhom}.

\item Let $ z_1, \dots, z_n \in \bc^\times$ be distinct and let $ \theta \in \fh $.   The \emph{trigonometric Gaudin subalgebra} is a maximal free commutative subalgebra $$\ma_{\theta}^{trig}(z_1, \ldots, z_n) \subset (U\fg^{\otimes n})^{\fh} \subset U\fg^{\otimes n}$$ which 
contains the trigonometric Gaudin Hamiltonians $H_{i,\theta}^{trig}$ for $ i = 1, \ldots, n$
$$ H_{i, \theta}^{trig} = \frac{\theta^{(i)}}{z_i} + \sum_{j\ne i } \frac{\Omega^{(ij)}}{z_i - z_j} - \sum_j \frac{\Omega_-^{(ij)}}{z_i}$$

\end{itemize}

\subsection{Trigonometric Gaudin model}
In the present paper, we define the trigonometric Gaudin model for any semisimple Lie algebra. In type A, it was considered in \cite{mtv} and \cite{mr}.\footnote{The formulas for quadratic Hamiltonians in \cite{mtv} and \cite{mr} are slightly different from ours but become the same up to a change of the parameter $\theta$ after restriction to a weight space (with respect to the diagonal Cartan subalgebra) in a representation of $U\fg^{\otimes n}$.} We give three equivalent descriptions of the trigonometric Gaudin subalgebras, each of which can be used to explain the above complicated formulas for the quadratic elements.  As above, we fix $ z_1, \dots, z_n \in \bc^\times$ distinct and $ \theta \in \fh $. 

\subsubsection{Image of a Gaudin algebra}
We begin with the Gaudin subalgebra $\ma(0, z_1, \ldots, z_n) \subset U\fg^{\otimes {n+1}} $.  Then we define $ \ma^{trig}_\theta(z_1, \dots, z_n) $ to be its image under a certain homomorphism $\psi_{\theta}: (U\fg^{\otimes n+1})^{\fg} \to U\fg^{\otimes n}$ coming from the quantum Hamiltonian reduction with respect to the diagonally embedded copy of $ \fn_+$, see section \ref{se:deftrig}.

\subsubsection{Verma tensor product multiplicity} \label{se:introVerma}
As $ \ma^{trig}_\theta(z_1, \dots, z_n) $ is a subalgebra of $ (U\fg^{\otimes n})^\fh $, it acts on tensor products of representations, preserving weight spaces.  Fix $ \lambda_1, \dots, \lambda_n $ dominant weights and consider the tensor product of irreducible representations $ V(\ul) := V(\lambda_1) \otimes \cdots \otimes V(\lambda_n) $.  We will now characterize its action on this representation.

For any $ \tau \in \fh^*$, let $ M(\tau)$ denote the Verma module with highest weight $ \tau$. Then for generic $ \theta \in \fh \cong \fh^*$ we have an identification
$$
\Hom_{\fg}(M(\theta + \mu), M(\theta)\otimes V(\ul)) \cong V(\ul)_\mu
$$
 
With respect to this identification, in section \ref{4.3} we establish an equality in the endomorphism algebra of this vector space
$$
\text{image of } \ma(0, z_1 \dots, z_n) = \text{ image of } \ma_\theta^{trig}(z_1 \dots, z_n) 
$$

\subsubsection{Image of a universal trigonometric Gaudin algebra}
For this construction, we begin with the universal Gaudin algebra $ \ma \subset U(\fg[t^{-1}]) $, which is defined using the centre of the affine Lie algebra at critical level.  Then using quantum Hamiltonian reduction, we produce a universal trigonometric Gaudin algebra $ \ma_\theta^{trig} \subset U(\fn_+ \oplus t^{-1} \fg[t^{-1}])$.  The trigonometric Gaudin algebra $ \ma_\theta^{trig}(z_1, \dots, z_n)$ is the image of this algebra under the evaluation homomorphisms.

\subsection{Parameter spaces of Gaudin subalgebras}
Each class of Gaudin subalgebra depends on a choice of parameters.  However, they exhibit the following invariance with respect to scaling and/or translation.
\begin{itemize}
    \item $\ma(z_1, \ldots, z_n) = \ma(a z_1 +b, \ldots, a z_n +b)$ for any $a \in \bc^{\times}$ and $b \in \bc$;
    \item  $\ma_{\theta}^{trig}(z_1, \ldots, z_n) = \ma_{\theta}^{trig}(a z_1, \ldots, a  z_n)$ for any $a \in \bc^{\times}$; 
    \item $\ma_{\chi}(z_1, \ldots, z_n) = \ma_{\chi}(z_1 + b, \ldots, z_n +b)$ for any $ b \in \bc $; \\ 
    $\ma_{\chi}(z_1, \ldots, z_n) = \ma_{a^{-1} \chi} (a z_1, \ldots, a z_n)$ for any $a \in \bc^{\times}$.
\end{itemize}

In the present paper, we consider families of trigonometric/inhomogeneous Gaudin subalgebra with fixed $\theta \in \fh$ and $\chi \in \fg$ and varying $ z_1, \dots, z_n $. 

Using the above invariance properties, we assign the following parameter spaces for families of Gaudin subalgebras.  Here $ B = \bc^\times \ltimes \bc $ is the group of affine linear transformations and $ \Delta $ denotes the fat diagonal.
\begin{itemize}
    \item For the homogeneous Gaudin model, we have the parameter space $M_{n+1} := (\bc^n \setminus \Delta) / B $, which we identify with the moduli space of rational curves with $n+1$  distinct marked points $z_1, \ldots, z_n, \infty$.
    \item For the trigonometric Gaudin model, we have the parameter space $M_{n+2} := ((\bc^\times)^n \setminus \Delta) / \bc^\times$, which we identify the moduli space of rational curves with $n+2$ distinct marked points $0, z_1, \ldots, z_n, \infty$.
    \item For the inhomogeneous Gaudin model, we have the parameter space $F_n := (\bc^n \setminus \Delta) /  \bc$, which we identify with the moduli space of rational curves with $n$ distinct marked points $z_1, \ldots, z_n$ and one distinguished point $ \infty$ which caries a non-zero tangent vector.
\end{itemize}

The main goal of the present paper is to describe all possible limits of these Gaudin subalgebras and to determine the compactified parameter space for these algebras.  These limit subalgebras arise as some of the parameters $z_i$ approach each other or tend to $ \infty $ (or $0$ in the trigonometric case).  More precisely, we have an injective map 
$$ M_{n+1} \rightarrow \{ \text{subalgebras of $ (U\fg)^{\otimes n} $} \}$$ 
(this is for the homogeneous Gaudin model; for the other models, the source will be different) and we wish to determine the closure of the image.  The rigorous definition of this closure depends on a choice of filtration by finite-dimensional subalgebras, see section \ref{se:Family}.

For homogeneous Gaudin subalgebras, the first author \cite{r} gave an explicit description of all limit subalgebras and the corresponding compactification of the parameter space. These results are summarized in section \ref{homgaudin} and more tersely as follows.  Let $ \overline M_{n+1}$ be the Deligne-Mumford compactification of $M_{n+1}$; it parametrizes genus 0 stable nodal curves with $n+1$ distinct marked points.
\begin{thm} \label{th:intromain0}
The compatified parameter space of homogeneous Gaudin subalgebras is $ \overline M_{n+1}$.  Moreover, the limit subalgebras are products of Gaudin subalgebras glued together according the operadic structure of $\overline M_{n+1}$. 
\end{thm}

The first main result of present paper is to generalize this result to inhomogeneous and trigonometric Gaudin models. 

\begin{thm} \label{th:intromain1}
\begin{enumerate}
\item For any $\theta \in \fh$ the compactified parameter space of trigonometric Gaudin subalgebras is  $\overline M_{n+2}$.  The limit subalgebras are certain products of trigonometric Gaudin subalgebras, homogeneous Gaudin subalgebras and certain images of homogeneous Gaudin subalgebras, see Theorem \ref{trigcomp}.
\item For any $\chi \in \fh^{reg}$ the compactified parameter space of inhomogeneous Gaudin subalgebras is the space of cactus flower curves $\overline{F}_n$. 
 The limit subalgebras are certain products of inhomogeneous and homogeneous Gaudin subalgebras, see Theorem \ref{compinhom}.
\end{enumerate}
\end{thm}

The above space $ \overline F_n $ was defined in \cite{iklpr} as a blowup of the Ardila-Boocher matroid Schubert variety associated to the type A root hyperplane arrangement.  A point of $ \overline F_n $ can be viewed a ``cactus flower curve'';  a curve $ C = C_1 \cup \dots \cup C_m$ where each $ C_j$ is a genus 0, nodal, marked stable curve, all $C_j$ meet at a distinguished point $ z_{n+1}$, and we have the data of a non-zero tangent vector to each $ C_j$ at $ z_{n+1}$.

The proof of Theorem \ref{th:intromain0} in \cite{r} used a theorem of Aguirre-Felder-Veselov \cite{afv} who showed that $\overline M_{n+1}$ is a space of maximal commutative subspaces of the Drinfeld-Kohno (or holonomy) Lie algebra (see Section \ref{se:afv}).  We use their theorem to help prove part (1) of Theorem \ref{th:intromain1}.  In order to prove part (2) of Theorem \ref{th:intromain1}, we introduce a new Lie algebra which we call the inhomogeneous holonomy Lie algebra $ \fr_n$ and prove an isomorphism (Theorem \ref{th:FnGnIso}) between $ \overline F_n $ and a space of commutative subalgebras of $\fr_n $ (an analog of the Aguirre-Felder-Veselov result).

\subsection{Relation between the models}
It is possible to include these integrable models into one picture describing different degenerations of the XXZ Heisenberg spin chain\footnote{We thank Pavel Etingof for explaining this.}:
$$
\begin{tikzcd}
& {\text{XXZ Heisenberg}} \arrow[dr] \arrow[dl] \\
{\text{XXX Heisenberg}} \arrow[dr] &&
  {\text{Trigonometric Gaudin}}  \arrow [dl,swap] \\
& {\text{(In)homogeneous Gaudin}}
\end{tikzcd}
$$
According to Mukhin, Tarasov and Varchenko \cite{mtv}, 
in type A, the intermediate degenerations are related by a duality which can be seen as Howe duality on the algebras of observables and as bispectral duality on the solutions of the Bethe ansatz.  From this perspective, the compactified parameter space for trigonometric Gaudin algebras from Theorem \ref{th:intromain1} can be compared with the compactification result for the XXX spin chain in type A, see \cite{ir2}. For a concrete result about the degeneration of the XXX model to inhomogeneous Gaudin subalgebras, see \cite[Theorem 8.12]{kmr}. For degeneration of XXZ to Trigonometric Gaudin in type A see \cite{mr}. To the best of our knowledge there are no results in the literature about degeneration of Bethe subalgebras in $U_q(
\hat \fg)$ (quantum affine universal enveloping algebra)  to Bethe subalgebras in Yangian, however see \cite{kz}.

In the present paper, we prove that the trigonometric Gaudin algebras can degenerate to the inhomogeneous ones.  In particular, we have the following result for $\chi\in\fh^{reg}$ (Theorem \ref{th:familyCFn}):
$$\lim_{\varepsilon \to 0} \ma^{trig}_{\varepsilon^{-1}\chi}(1-\varepsilon  z_1, \ldots, 1-\varepsilon z_n) = \ma_{\chi}(z_1, \ldots, z_n).$$

In order to study this degeneration further, we introduce some schemes over $ \BA^1$.  Following Zahariuc \cite{Z}, we consider
$$
\BG = \{(z, \varepsilon) : 1 - z \varepsilon \ne 0 \}
$$
a group scheme over $ \BA^1 $ (via the map $(z,\varepsilon) \mapsto \varepsilon$) which degenerates $ \Cx$ to $ \C $.  Let $ \CF_n = (\BG^n \setminus \Delta) / \BG $.  This is the total space of a degeneration of $ M_{n+2} $ to $ F_n $.  We prove that for fixed $ \chi \in \fh$, there is injective map 
$$ \CF_n \rightarrow \{ \text{subalgebras of $ (U\fg)^{\otimes n} $} \}$$ 
which at $ \varepsilon \ne 0 $ gives the trigonometric Gaudin subalgebras (for $ \theta = \varepsilon^{-1} \chi $) and at $ \varepsilon = 0 $ gives the inhomogeneous Gaudin subalgebras.

We prove that the fibrewise (with respect to $ \varepsilon \in \BA^1$) compactification of this family is $ \overline \CF_n$, another space we introduced in \cite{iklpr}.  
This is the total space of a degeneration of $ \overline M_{n+2}$ to $ \overline F_n $.  Geometrically, $\overline{\mathcal F}_n $ parametrizes families of marked curves where two marked points come together to form a distinguished point with a tangent vector.  

As an application, we use the degeneration of trigonometric Gaudin subalgebras to inhomogeneous Gaudin subalgebras to show the existence of a cyclic vector for trigonometric Gaudin subalgebras (for generic parameters). Potentially, this is important for the completeness of the Bethe ansatz in the trigonometric Gaudin model.




\subsection{Bethe ansatz for Gaudin models.}
One of the main problems in Quantum Integrable systems is to explicitly find eigenvalues and eigenspaces for a given model. The Bethe ansatz method outputs a system of algebraic equations on some auxiliary parameters, such that the desired eigenvectors and eigenvalues are some explicit rational expressions of these parameters. However, almost always it is impossible to solve these equations explicitly; in particular, the solutions of the Bethe ansatz equations are highly multivalued functions of the parameters.
The Bethe ansatz conjecture states that 
the number of solutions of a corresponding system is equal to the dimension of a representation (at least for generic values of the parameters of the model), i.e. the Bethe ansatz method is complete. In \cite{ffre}, the Bethe Ansatz equations were interpreted as a ``no-monodromy'' condition on certain space of {\em opers} on the projective line $\mathbb{P}^1$. Namely, it was shown that $\ma(\uz)$ is isomorphic to
the algebra of functions on the space of $G^\vee$-opers on $\mathbb{P}^1$ with regular singularities at the points $z_1,\dots,z_n$ and
$\infty$, where $G^\vee$ is the Langlands dual group of $G$, and $G^\vee$-opers are connections on a principal $G^\vee$-bundle over $\mathbb{P}^1$ satisfying a certain transversality condition. This led to the new form of the completeness conjecture in terms of monodromy-free opers, which is known both for homogeneous and inhomogeneous Gaudin models: see \cite{r4} for the homogeneous and \cite{ffry} for the inhomogeneous model. In \cite[Conjecture 5.8]{efk}, Etingof, Frenkel and Kazhdan generalized this form of the Bethe ansatz conjecture so that it covers the trigonometric Gaudin model as well.

In the present paper, we work towards this conjecture for the trigonometric Gaudin model. We show that under certain reality conditions on the parameters, the spectrum of the trigonometric Gaudin model is simple. In particular we consider two real forms of $\overline{M}_{n+2}$: split $\overline{M}_{n+2}^{split} $ and compact $\overline{M}_{n+2}^{comp}$.  (Both spaces are in fact compact; the names come from the fact that they are related to the split (resp. compact) real forms of the torus $ (\bc^\times)^n / \bc^\times $.)  On the other hand, we consider only one real form of $ \overline F_n$, which we denote by $ \overline F_n(\br)$; it appears as the degeneration of both $\overline{M}_{n+2}^{split} $ and $\overline{M}_{n+2}^{comp}$. More precisely we consider two real forms of $\overline \CF_n$,  $ \overline{\mathcal F}^{split}_n$ and $ \overline{\mathcal F}^{comp}_n$, which have general fibres $\overline{M}_{n+2}^{split} $ and $\overline{M}_{n+2}^{comp}$, respectively, and each have $ \overline F_n(\br) $ as central fibre.

As above, for any $ C \in \overline M_{n+2}$ and $ \theta \in \fh$, the algebra $\ma_\theta^{trig}(C) $ acts on $V(\ul)_\mu$, a weight space in a $n$-fold tensor product.  Similarly, for $ C \in \overline F_n$ and $ \chi \in \fg$, the algebra $ \ma_\chi(C) $ acts on $ V(\ul)$.  An important question is to determine when this action splits $ V(\ul)_\mu$, or $ V(\ul)$, into a direct sum of distinct one-dimensional representations (or joint eigenvectors) for these algebras; in this case we say that the algebra acts with \textit{simple spectrum}.

The second main result of the present paper is summarized in the following theorem.  Here $\fh^{split}$ (resp. $\fh^{comp}$) denotes the Cartan subalgebra of the split (resp. compact) real form of $ \fg $.

\begin{thm} \label{thmA}
\begin{enumerate}
\item Let $ \theta \in \fh^{split} $ be generic and sufficiently dominant (i.e. $(\alpha,\theta) \gg 0$ for all positive roots $\alpha$).  For every $ C\in \overline M_{n+2}^{split}$, $\ma_\theta^{trig}(C)$ acts on $V(\ul)$ with 
 simple spectrum.
\item
Let $ \theta \in \fh $ be generic and such that $\theta-\frac{1}{2}\mu\in \fh^{comp}$.
For every $C \in \overline M_{n+2}^{comp}$, $\ma_\theta^{trig}(C)$ acts on $V(\ul)_\mu$ with simple spectrum.
\item Let $ \chi \in \fh^{split}$ be regular.  For every $ C\in \overline F_n(\br) $, $ \ma_\chi(C)$ acts on $ V(\ul)$ with simple spectrum.
\end{enumerate}
\end{thm}

One can think about the above statements as necessary conditions for the completeness of the Bethe ansatz.

\subsection{Monodromy of eigenvectors}
From Theorem \ref{thmA}, we can form  $S_n$-equivariant coverings of $ \overline M_{n+2}^{split}$, $ \overline M_{n+2}^{comp} $, and $ \overline F_n(\br) $ whose fibres over a point $C$ is the set of joint eigenvectors for the action of $\ma^{trig}_\theta(C) $ (or $\ma_\chi(C)$) on $ V(\ul)_\mu$ (here we have fixed $ \theta$ or $\chi$ satisfying the corresponding conditions of Theorem \ref{thmA}).  In each situation, this leads to an action of the equivariant fundamental group of the base on the fibre of the covering.

In \cite{hkrw}, we studied the corresponding covering of joint eigenvectors for the action of homogeneous Gaudin algebras and related it to the action of the cactus group $ C_n = \pi_1^{S_n}(\overline M_{n+1}^{split})$ on tensor products of crystals. 

In our setting, the fundamental groups were described in \cite[Theorems 11.11, 11.12]{iklpr}.
\begin{thm}
    \begin{enumerate}
        \item $\pi_1^{S_n}(\overline M_{n+2}^{comp}) \cong \widetilde{AC}_n$, \text{ the extended affine cactus group}
        \item $\pi_1^{S_n}(\overline F_n(\br)) \cong vC_n$, \text{ the virtual cactus group}
    \end{enumerate}
\end{thm} 
See \cite[\S 10]{iklpr} for the exact definitions of these groups.  For our purposes, the most important thing to know is that $ vC_n$ is generated by a copy of the cactus group $ C_n $ and of the symmetric group $ S_n$.

Moreover, as $ \pi_1^{S_{n+1}}(\overline M_{n+2}^{split}) \cong C_{n+1}$, it is easy to see that $ \pi_1^{S_n}(\overline M_{n+2}^{split})$ is the preimage of $ S_n$ in $ C_{n+1}$.  

In \cite[Theorems 9.23, 9.24]{iklpr}, we also proved that $ \overline{\mathcal F}^{split}_n$ and $ \overline{\mathcal F}^{comp}_n$ both deformation retract onto $ \overline F_n(\br) $ and thus the monodromy action of the fundamental groups $ \pi_1^{S_{n}}(\overline M_{n+2}^{split})$, $\pi_1^{S_n}(\overline M_{n+2}^{comp})$ must factor through homomorphisms from these groups to $ vC_n$.

 We expect this monodromy action to have a combinatorial meaning in terms of crystals.  Recall, from \cite{hk}, that a \emph{coboundary category} is a monoidal category $ \mathcal{C} $ along with a natural isomorphism called \emph{commutor} $$ \sigma_{A, B} : A \otimes B \rightarrow B \otimes A$$
	satisfying the following two axioms.
	\begin{enumerate}
		\item For all $A, B \in \mathcal{C}$, we have $\sigma_{B, A} \circ \sigma_{A,B} = id_{A \otimes B} $
		\item For all $ A, B, C \in \mathcal{C} $ we have $\sigma_{A, B \otimes C}\circ (id \otimes\sigma_{B, C}) = \sigma_{A \otimes B, C}\circ (\sigma_{A, B} \otimes id)$.
	\end{enumerate}

All possible commutors on the $n$-fold tensor product	$X_1\otimes\ldots\otimes X_n$ in a coboundary category $\mathcal{C}$ generate an action of the cactus group $C_n$.

Similarly, the \emph{virtual cactus group} $vC_n=\pi_1^{S_n}(\overline F_n)$ arises in \emph{concrete} coboundary monoidal categories \cite{KR}. Namely, suppose we are given a coboundary monoidal category $\mathcal{C}$ and a faithful monoidal functor $F:\mathcal{C}\to Sets$. Then for any $n$-tuple of objects $X_1,\ldots,X_n\in \mathcal{C}$, the set $F(X_1\otimes\ldots\otimes X_n)=F(X_1)\times\ldots\times F(X_n)$ is naturally acted on by $vC_n$, where the symmetric group $S_n=\pi_1^{S_n}(pt)$ acts by naive permutations of factors (i.e. by commuters in the category of sets) and $C_n$ acts by commutors in the category $\mathcal{C}$. 

The main example of a coboundary monoidal category is the category of Kashiwara $\fg$-crystals (a combinatorial version of the category of finite-dimensional $\fg$-modules). More precisely, to any irreducible representation $V(\lambda)$ one assigns an oriented graph $B(\lambda)$ (called the normal crystal of highest weight $\lambda$) whose vertices correspond to basis vectors of $V(\lambda)$ and are labeled by the weights of $V(\lambda)$, while the edges correspond to the action of the Chevalley generators and are labeled by the simple roots of $\fg$.  There is a purely combinatorial rule of tensor multiplication that provides a structure of a crystal on any Cartesian product $B(\lambda_1)\times B(\lambda_2)$.  This is a concrete coboundary monoidal category with $F$ being the forgetful functor.  

In \cite{hkrw}, we proved that there is a bijection between the set of eigenvectors for the Gaudin Hamiltonians in $\Hom_\fg(V(\mu),V(\ul))$ and the multiplicity set of $ B(\mu)$ in the tensor product of  $\fg$-crystals $B(\lambda_1)\otimes\ldots\otimes B(\lambda_n)$, compatible with the $C_n$-action.
 
In a future work \cite{KR}, we will use these ideas to prove the following conjectures. 

\begin{conj} \label{conj:monodromy}
\begin{enumerate} \item The monodromy action of $ vC_n = \pi_1^{S_n}(\overline F_n(\br))$ on the inhomogeneous Gaudin eigenlines in $ V(\ul)_\mu$ matches the action of $ vC_n$ on $ B(\lambda_1) \otimes \cdots \otimes B(\lambda_n)$ given by crystal commutors and naive permutation of tensor factors.

\item The monodromy action of $\widetilde{C_n}=\pi_1^{S_n}(\overline M_{n+2}^{split}(\br))$ on the trigonometric Gaudin eigenlines in $V(\ul)_\mu$ factors surjectively through the action of $ vC_n$.
\item
The monodromy action of $ \widetilde{AC}_n = \pi_1^{S_n}(\overline M_{n+2}^{comp}(\br))$ on the trigonometric Gaudin eigenlines in $ V(\ul)_\mu$ factors through the action of $ vC_n $ and is given by crystal commutors and cyclic rotation of tensor factors.
\end{enumerate}
\end{conj}

\subsection{Howe duality and relation to the XXX model.} Let $\fg=\fsl_N$ and $\lambda_i=l_i\omega_1$, i.e. $V(\lambda_i)=S^{l_i}\bc^N$ is the $l_i$-th symmetric power of the tautological representation $\bc^N$. Let $\mu=(m_1,\ldots,m_N)$ be an integral weight of $\fgl_N$ that occurs in the weight decomposition of $V(\ul)=\bigotimes\limits_{i=1}^n V(\lambda_i)$ (in particular, we have $\mu_j\in\bz_{\ge0}$). On the other hand, consider the irreducible representations $V(\mu_j)$ of $\fgl_n$ with $\mu_j=m_j\omega_1$, i.e. $V(\mu_j)=S^{m_j}\bc^n$. Then the weight $\lambda:=(l_1,\ldots,l_n)$ occurs in the weight decomposition of $V(\underline{\mu}):=\bigotimes\limits_{j=1}^N V(\mu_j)$. 

Howe's duality identifies the weight space $V(\ul)_\mu$ with the weight space $V(\underline{\mu})_\lambda$. Namely, both are the $(\mu,\lambda)$-weight space in the decomposition of the $\fgl_N\oplus\fgl_n$-module $S^\bullet(\bc^N\otimes\bc^n)$.

The trigonometric Gaudin algebra gets switched with the XXX Heisenberg system under this identification. More precisely, let $\theta=(u_1,\ldots,u_N)\in \fh\subset\fgl_N$. Then we can regard the space $V(\ul)_\mu$ as the space of states for the trigonometric Gaudin model, i.e. study the action of the commutative subalgebra $\ma_\theta^{trig}(z_1,\ldots,z_n)\subset U(\fgl_N)^{\otimes n}$ on this space. On the other hand, we can regard the same space as $V(\underline{\mu})_\lambda$, which is the $\lambda$-weight space in the tensor product module $V(\mu_1,u_1)\otimes\ldots\otimes V(\mu_N,u_N)$ over the Yangian $Y(\fgl_n)$ with the highest weight $\mu_j$ and the evaluation parameters $u_j$. This module is acted on by the commutative \emph{Bethe subalgebra} $B(\operatorname{diag}(\uz))\subset Y(\fgl_n)$ with parameter the diagonal $n\times n$-matrix with entries $z_1,\ldots,z_n$. According to \cite{mtv} the image of $\ma_\theta^{trig}(\uz)$ in the operators on this vector space coincides with the image of $B(\operatorname{diag}(\uz))$ under Howe duality.  Thus Theorem \ref{th:intromain1}(1) matches with the description of the compactification of the parameter space for Bethe subalgebras in the Yangian from \cite{ir2}.   A similar statement holds with the symmetric powers replaced by exterior powers as well, see \cite{tu,u}. 

This means that in two particular cases (of symmetric or exterior powers of the standard representation of $\fgl_N$), we can describe the fibers of the corresponding covering $\mathcal{E}_\theta^{trig}$ of $\overline{M}_{n+2}^{comp}(\br)$ in two different ways: as the sets of eigenlines for the trigonometric Gaudin subalgebras in $U\fgl_N^{\otimes n}$ and as those for Bethe subalgebras in the Yangian $Y(\fgl_n)$. On the latter one, we have a combinatorial structure of a Kirillov-Reshetikhin (KR) crystal described in \cite{kmr}. We expect the following to be true:

\begin{conj}
    The $\widetilde{AC_n}$ action by the monodromy of the covering $\mathcal{E}_\theta^{trig}$ over $\overline{M}_{n+2}^{comp}(\br)$ comes from the inner $\widetilde{AC_n}$ action on the corresponding KR crystal by partial Sch\"utzenberger involutions.
\end{conj}

We plan to address this Conjecture in a subsequent joint paper with Vasily Krylov.

\subsection{Relation to quantum cohomology of slices in the affine Grassmannian}

One motivation for studying trigonometric Gaudin algebras is that we expect them to be related to the quantum cohomology of affine Grassmannian slices.  To explain this relation, we begin by recalling some definitions (see \cite{KWWY} for more details).  

We work in the affine Grassmannian $ \Gr := G^\vee((t))/ G^\vee[[t]]$, where $ G^\vee $ is the Langlands dual group.  To each dominant weight of $G$, we have the finite dimensional spherical Schubert variety $ \Gr^\lambda \subset \Gr$ and its closure $ \overline \Gr^\lambda$.  Given a sequence $ \ul = (\lambda_1, \dots, \lambda_n) $, then we consider the convolution of these spherical Schubert varieties $ \Gr^\ul := \Gr^{\lambda_1} \tilde \times \cdots \tilde \times \Gr^{\lambda_n}$ and its closure $ \overline \Gr^\ul$.  We have the convolution morphism $ m_\ul : \overline \Gr^\ul \rightarrow \Gr$ whose image is $ \overline \Gr^{\lambda_1 + \dots + \lambda_n}$.  

Furthermore, we also consider the transversal slices $\Gr_\mu \subset \Gr$, for $ \mu$ dominant.  We will be interested in the intersections $ \overline \Gr^\lambda_\mu := \overline \Gr^\lambda \cap \Gr_\mu$ and the preimages $ \overline \Gr^\ul_\mu := m_\ul^{-1}(\Gr_\mu) $.  

For any $ \lambda $ and $ \mu$, the affine Grassmannian slice $ \overline \Gr^\lambda_\mu$ is an affine Poisson variety which is usually singular.  If $ \lambda = \lambda_1 + \cdots + \lambda_n$ is a sum of minuscule dominant weights, then $ \Gr^\ul = \overline \Gr^\ul $ is smooth and $ \Gr^\ul_\mu $ is a symplectic resolution of $ \overline \Gr^\lambda_\mu$.

 These affine Grassmannian slices and their resolutions have been studied extensively in recent years, see the survey paper \cite{KamBLMS}, because they are connected to the geometric Satake correspondence and because they are symplectic dual to Nakajima quiver varieties.

For now, let us assume that  $\lambda_1, \dots, \lambda_n$ are minuscule.  
The resolved affine Grassmannian slice $ \mathrm{Gr}^\ul_\mu $ has an action of $ T^\vee \times \bc^\times$, where $T^\vee \subset G^\vee $ is the maximal torus and $\C^\times $ acts by loop rotation.  We will now give a Langlands dual description of $ H^*_{T^\vee \times \C^\times}(\Gr^\ul_\mu) $, following the work of Ginzburg-Riche \cite{gr}.  

Consider the Rees algebra $ U_\hbar \fg $ and the algebra $ S_\hbar(\fh) := S(\fh) \otimes \C[\hbar] $.  Let $ \bM[\mu] := U_\hbar \fg \otimes_{U_\hbar \fb_+} S_\hbar \fh$ be the shifted universal Verma module, where $ [\mu] $ indicates that $ x \in \fh \subset \fb_+ $ acts by $ x + \hbar \mu(x) $ on $S_\hbar(\fh)$.  This is a $ U_\hbar \fg, S_\hbar (\fh) $-bimodule.  We will be interested in tensor products
$$
\bM \otimes V(\ul) := \bM \otimes V(\lambda_1) \otimes \cdots \otimes V(\lambda_n)
$$
which we regard as a $ U_\hbar \fg $ module using the coproduct $ U_\hbar \fg \rightarrow U_\hbar \fg \otimes U \fg^{\otimes n} $ defined $ x \mapsto x^{(0)} + \sum_{i=1}^n \hbar x^{(i)}$ for $ x \in \fg$.  

\begin{thm} \label{th:GR}
    There is a canonical isomorphism of $ S_\hbar(\fh)$-modules
    \begin{equation} \label{eq:GR}
H^*_{T^\vee \times \C^\times}(\Gr^\ul_\mu) \cong \Hom_{U_\hbar\fg\otimes S(\fh)}(\bM[\mu], \bM \otimes V(\ul)),
    \end{equation}
\end{thm}
In this theorem, the left hand side of (\ref{eq:GR}) acquires the $S_\hbar(\fh)$-module structure using $ S_\hbar(\fh) \cong H^*_{T^\vee\times\bc^\times}(pt)$.  

\begin{proof}
Let $ A := T^\vee \times \C^\times$.
    
    Let $ \CF^\ul := (m_{\ul})_* \underline{\C}_{\Gr^\ul}[\dim \Gr^\ul]$ be the push-forward of the shifted constant sheaf.  This is a perverse sheaf on $ \Gr$ which corresponds to $ V(\ul) $ under the geometric Satake correspondence.  Let $ i_\mu $ be the inclusion of the point $ t^\mu$ in $ \Gr$.  By \cite[Theorem 2.2.4]{gr}, we have a canonical isomorphism 
$$ H^*_A(i_\mu^! \CF^\ul) \cong  \Hom_{U_\hbar\fg\otimes S(\fh)}(\bM[\mu], \bM \otimes V(\ul))$$
On the other hand, let $ j_\mu : \Gr_\mu \rightarrow \Gr $ and $ j^\ul_\mu : \Gr^\ul_\mu \rightarrow \Gr^\ul $ be the inclusions.  Using the smoothness of $ \Gr^\ul$ and base change, we find that (ignoring cohomological shifts)
$$
H^*_A(\Gr^\ul_\mu) \cong H^*_A(\Gr^\ul_\mu, (j^\ul_\mu)^! \C_{\Gr^\ul}) \cong H^*_A(\Gr_\mu, j_\mu^! \CF^\ul) 
$$
Since $ \Gr_\mu$ is $ A$-equivariantly contractible to the point $ t^\mu$, we find that $ H^*_A(\Gr_\mu, j_\mu^! \CF^\ul)  \cong H^*_A(i_\mu^! \CF^\ul) $ which completes the proof.
\end{proof}

The Gaudin algebra $ \ma(0, z_1, \dots, z_n) \subset U_\hbar \fg \otimes U \fg^{\otimes n} $ acts on $ \bM \otimes V(\ul) $.  This action is closely related to the trigonometric Gaudin model.  More precisely, if we specialize at a generic point $ \theta \in \fh \cong \fh^*$ and at $ \hbar = 1$, then there is an isomorphism (as in Proposition \ref{thm1}) 
\begin{equation} \label{eq:HomtoV}
\Hom_{U_\hbar\fg\otimes S(\fh)}(\bM[\mu], \bM \otimes V(\ul)) \otimes_{S_\hbar(\fh)} \C_{\theta, 1} \cong V(\ul)_\mu
\end{equation}
and the action of $ \ma(0,z_1, \dots, z_n)$ on the left hand side becomes the action of $ \ma^{trig}_\theta(z_1, \dots, z_n) $ on $ V(\ul)_\mu$ (this is a slight variant of section \ref{se:introVerma}).  

 The left hand side of (\ref{eq:GR}) admits an action of the quantum cohomology ring $QH^*_{T^\vee\times\bc^\times}(\mathrm{Gr}^\ul_\mu)$ which depends on a quantum parameter in $ H^2(\mathrm{Gr}^\ul_\mu, \bc^\times)_{reg}$, which is the complement to the divisors defined by the K\"ahler roots, see \cite[Section 3.1.4]{Ok}.  For resolved affine Grassmannian slices, this space of quantum parameters is very easy to understand, since there is a map $ ((\bc^\times)^n \setminus \Delta)/ \Cx \rightarrow H^2(\mathrm{Gr}^\ul_\mu, \bc^\times)_{reg} $ which is usually an isomorphism (it is always surjective by \cite[Prop 3.4.7]{Dan}), so we interpret the quantum parameter as a point $ \uz \in ((\bc^\times)^n \setminus \Delta)/ \Cx = M_{n+2}$. 

Danilenko \cite{Dan} studied the quantum cohomology of affine Grassmanians, under the assumption that $ G $ is simply-laced.  For generic $ \theta$, he identified $ H^*_{\theta, 1}(\mathrm{Gr}^\ul_\mu) $ with $ V(\ul)_\mu$ by sending the stable basis to the tensor product basis.  Under this identification (which is compatible with the isomorphisms (\ref{eq:GR}), (\ref{eq:HomtoV})), he proved that quantum multiplication by divisors is given by the trigonometric Gaudin Hamiltonians. 

This leads us to the following conjecture. 

\begin{conj}\label{conj:danilenko}
Let $G$ be simply-laced.  Let $ \ul $ be a sequence of minuscule weight and let $ \mu $ be another dominant weight.  For any $ \uz \in M_{n+2}$, the equivariant quantum cohomology ring $ QH^*_{T^\vee\times\bc^*}(\mathrm{Gr}^\ul_\mu)$ at the quantum parameter $ \uz $ is equal, under the identification (\ref{eq:GR}), to the image of $ \ma(0, z_1, \dots, z_n) $ acting on $\Hom_{U_\hbar\fg\otimes S(\fh)}(\bM[\mu], \bM \otimes V(\ul)) $.  In particular, for generic $ \theta$, $QH^*_{\theta, 1}(\mathrm{Gr}^\ul_\mu) $ is equal to the image of $ \ma^{trig}_\theta(\uz)$ acting on $ V(\ul)_\mu$.
\end{conj}
  Suppose that $ G = SL_n$, $ \lambda_i = \omega_{k_i}$ with $ \sum k_i = n  $ and $ \mu = (1, \dots, 1)$, then $ \mathrm{Gr}^\ul_\mu$ is a partial flag variety for $ SL_n$.  In this case, the conjecture follows from the work of Gorbounov-Rimanyi-Tarasov-Varchenko \cite{GRTV} along with the above-mentioned bispectral duality.  
More generally, for the specialization at generic $\theta$ and $\hbar=1$, the conjecture should follow from Danilenko's work.

From Theorem \ref{thmA} and Conjecture~\ref{conj:danilenko}, we see that for generic $\theta$ and for any $C\in\overline M_{n+2} ^{comp}$ the corresponding subalgebra  $\mathcal{A}_{\theta}^{trig}(C)$ acts with simple spectrum on any weight space, so the corresponding quantum cohomology ring is semisimple.

\subsection{Cohomology of resolved affine Grassmannian slices}
Following the philosophy of taking limit subalgebras, Conjecture~\ref{conj:danilenko} also gives the following description of the \emph{classical} cohomology ring of $\mathrm{Gr}^\ul_\mu$. Note that $\uz=(z_1,\ldots,z_n)$ is a collection of nonzero complex numbers up to dilation, so the actual quantum parameters are the ratios $\frac{z_{i+1}}{z_{i}}$. So to obtain the classical cohomology ring we consider a limit $ \lim_{t \to 0} \ma(\uz(t))$ where $\lim_{t\to 0} \frac{z_{i+1}(t)}{z_{i}(t)}=0$ for all $i=1,\ldots,n-1$. By Theorem \ref{th:intromain0}, the compactified parameter space for these Gaudin algebras is $ \overline M_{n+2}$.  Taking this limit corresponds to the \emph{caterpillar} curve $C\in \overline{M}_{n+2}$ which has $ n$ components $C_1, \dots, C_n$ arranged linearly with $ z_0, z_1 \in C_1, z_2 \in C_2, \dots, z_{n-1} \in C_{n-1} $ and $ z_n, z_{n+1} \in C_n$. 

By Proposition \ref{pr:limitGaudin}, the corresponding limit Gaudin subalgebra $ \ma(C)$ is generated by certain diagonally embedded two point Gaudin subalgebras.  However, when this algebra acts on $ \bM \otimes V(\ul)$, since all $V(\lambda_i)$ are minuscule (and so, all the weight multiplicities in $V(\lambda_i)$ are at most $1$), the tensor product of $V(\lambda_i)$ with a generic Verma module is a multiplicity-free sum of Verma modules that are the joint eigenspaces for the diagonally embedded centre of the universal enveloping algebra. This means that the action of each of the above two point Gaudin subalgebras is generated by the diagonally embedded centre $ Z = Z(U_\hbar \fg) $.  This leads to the following conjecture.

\begin{conj}
    Let $ \ul $ be a sequence of minuscule weight and let $ \mu $ be another dominant weight.  The equivariant cohomology ring $ H^*_{T^\vee\times\bc^\times}(\mathrm{Gr}^\ul_\mu)$ equals, under the identification (\ref{eq:GR}), to the image of
    $$ \Delta^{\{0\}, \{1\}}(Z) \Delta^{\{0,1\}, \{2\}}(Z) \cdots \Delta^{\{0,\dots, n-1\}, \{n\}}(Z) $$ acting on $\Hom_{U_\hbar\fg\otimes S(\fh)}(\bM[\mu], \bM \otimes V(\ul)) $.  In particular, for generic $ \theta$, $H^*_{\theta, 1}(\mathrm{Gr}^\ul_\mu) $ is equal to the image of $ \ma^{trig}_\theta(C)$ acting on $ V(\ul)_\mu$.
\end{conj}
Here we use a limit trigonometric algebra $ \ma^{trig}(C)$ which does not admit a very simple description, as explained in section \ref{comp}.

We believe that this conjecture holds even in the case where $ \lambda_i$ are not minuscule and $ \overline \Gr^\ul_\mu$ is not smooth.  More precisely, now let $ \ul$ be any sequence of dominant weights.  As in Theorem \ref{th:GR}\footnote{If the reader examines this proof carefully, they will notice that we need to know that $ (j^\ul_\mu)^! IC_{\overline \Gr^\lambda} \cong IC_{\overline \Gr^\ul_\mu}$, up to cohomological shift.  This follows from the proof of Lemma 8.4.1 in \cite{gr}.}, we have an identification
        \begin{equation} \label{eq:GR2}
IH^*_{T^\vee \times \C^\times}(\Gr^\ul_\mu) \cong \Hom_{U_\hbar\fg\otimes S(\fh)}(\bM[\mu], \bM \otimes V(\ul)),
    \end{equation}
The left hand side of (\ref{eq:GR2}) has an action of $ H^*_{T^\vee\times\bc^\times}(\mathrm{Gr}^\ul_\mu)$ and so we can consider it as a subalgebra of the endomorphism ring of the left hand side. 

\begin{conj} \label{co:cohom2}
    Let $ \ul $ be any sequence of dominant weight and let $ \mu $ be another dominant weight.  The equivariant cohomology ring $ H^*_{T^\vee\times\bc^\times}(\mathrm{Gr}^\ul_\mu)$ (as a subalgebra of the endomorphism ring of the equivariant intersection homology) equals, under the identification (\ref{eq:GR2}), to the image 
    $$ \Delta^{\{0\}, \{1\}}(Z) \Delta^{\{0,1\}, \{2\}}(Z) \cdots \Delta^{\{0,\dots, n-1\}, \{n\}}(Z) $$ in $\Hom_{U_\hbar\fg\otimes S(\fh)}(\bM[\mu], \bM \otimes V(\ul)) $.  
\end{conj}

We expect that the above conjecture follows from the results of \cite{BF} describing the $G^\vee\times\bc^\times$-equivariant intersection cohomology of spherical Schubert varieties; we plan to return to this in a subsequent work.

A similar relation between inhomogeneous Gaudin subalgebras and spherical Schubert varieties was recently observed by Tamas Hausel in \cite{ha}, following \cite{ffry}.  Hausel considers the Rees algebra with respect to the first tensor factor of the two point Gaudin algebra $ \Rees \ma(0, 1) \subset (U_\hbar \fg \otimes U \fg)^\fg$.  This algebra contains $ (U_\hbar\fg)^\fg \otimes 1 \cong S_\hbar(\fh)^W \cong H_{G^\vee \times \C^\times}(pt)$ and also $ \Delta(Z(U \fg))$.  For any dominant weight $ \lambda$, we can consider $ (U_\hbar \fg \otimes U \fg)^\fg \rightarrow (U_\hbar \fg \otimes \End V(\lambda))^\fg $ (this latter space is natural from the viewpoint of the derived Satake correspondence \cite{BF}).

Hausel states the following result.

\begin{thm} \label{th:Hausel}
\begin{enumerate}
    \item 
The equivariant cohomology ring $ H^*_{G^\vee \times \Cx}(\overline \Gr^\lambda)$ is isomorphic to the image of $ \Delta(Z(U \fg)) $ acting on $ (U_\hbar \fg \otimes \End V(\lambda))^\fg $ as a $S_\hbar(\fh)^W$-algebra.
\item The equivariant intersection homology $ IH^*_{G^\vee \times \Cx}(\overline \Gr^\lambda)$ is isomorphic to the image of $ \Rees \ma(0,1) $ acting on $ (U_\hbar \fg \otimes \End V(\lambda))^\fg $ as a $(U_\hbar\fg)^\fg$-module.
\end{enumerate}
\end{thm}

\subsection{Ring structure on intersection cohomology of slices in the affine Grassmannian}
In \cite{mp}, McBreen and Proudfoot conjectured (and proved in the hypertoric case) that there is a natural ring structure on the equivariant intersection homology of conical symplectic singularities which admit symplectic resolutions.  More precisely, given a conical symplectic resolution $ Y \rightarrow X$ admitting a Hamiltonian $ T $ action, McBreen-Proudfoot define a procedure for specializing the equivariant quantum cohomology ring $QH^*_{T \times \Cx}(Y) $ at the quantum parameter $ q = 1$.  After quotienting the result by the annihilator of $ \hbar$, they reach an algebra which they conjecture provides a ring structure to the equivariant intersection homology of $ X $.

We consider their conjecture with $ Y = \Gr^{\ul}_\mu $ and $ X = \overline \Gr^\lambda_\mu$, where $ \lambda = \lambda_1 + \dots + \lambda_n $ and all $ \lambda_i $ are minuscule.  To simplify notation, we fix $ \theta \in \fh$ generic and $ \hbar = 1$.  Then conjecture~\ref{conj:danilenko} gives a description of the equivariant quantum cohomology ring $ QH^*_{\theta, 1}(\mathrm{Gr}^\ul_\mu)$ as the image of $\ma^{trig}_\theta(\uz)$ in $ \operatorname{End} V(\ul)_\mu$, with $\uz$ being the quantum parameter.  The McBreen-Proudfoot specialization of the quantum parameter at $1$ corresponds to $z_1=\ldots=z_n$.  

 This specialization corresponds to the divisor in the parameter space $\overline M_{n+2}$ consisting of two-component curves with the marked points $0,\infty$ on one component and all the other points $z_1,\ldots, z_n$ on the other component attached to the first one at $1$. The corresponding subalgebras are the products of $\ma_\theta^{trig}(1)\subset U\fg$ diagonally embedded into $U\fg^{\otimes n}$ and the homogeneous Gaudin subalgebra $\ma(w_1,\ldots,w_n)\subset (U\fg^{\otimes n})^\fg$. The piece of this subalgebra independent of the $w$'s, i.e. the diagonal $\ma_\theta^{trig}(1)$, is the desired McBreen-Proudfoot specialization. Under the geometric Satake correspondence, we may identify (for generic $ \theta$) the intersection homology with the weight space $ IH_{\theta,1}(\overline \Gr^\lambda_\mu) \cong V(\lambda)_\mu $, regarded as a subspace of $ V(\ul)_\mu$. This subspace is stable under $\ma_\theta^{trig}(1)$ and is cyclic as a $\ma_\theta^{trig}(1)$-module for generic $\theta$ (see Theorem \ref{thmA}). So it is natural to expect the following.

\begin{conj}
    The equivariant intersection homology ring $IH_{\theta,1}(\overline \Gr^\lambda_\mu)$ in the sense of McBreen and Proudfoot is the image of $\ma_\theta^{trig}(1)$ acting on $V(\lambda)_\mu$.
\end{conj}

Note the similarity with Theorem \ref{th:Hausel}(2).  Also, we note that in \cite[Theorem 3.1(4)]{ha}, Hausel presents a similar conjecture for the non-equivariant intersection homology of $ \overline \Gr^\lambda_\mu$.

\subsection{Wall-crossing for symplectic singularities}
Let $ Y $ any conical symplectic resolution with an action of a Hamiltonian torus $ T$.  According to the general philosophy of Bezrukavnikov and Okounkov, we expect that the monodromy of the joint eigenlines of the operators of quantum multiplication in $H_{T \times \bc^\times}^*(Y)$ is given by the combinatorial counterpart of \emph{wall-crossing functors} between the derived categories $\mathcal{O}$ of modular representations of the quantization of the underlined symplectic singularity. Namely, Losev \cite{Losev2} has studied wall-crossing functors that correspond to moving the parameter of the quantum deformation of the coordinate ring of a symplectic singularity $Y$ across the walls of a periodic hyperplane arrangement in $ H^2(Y, \mathbb{R}) $.  He proved that these wall-crossing functors are  perverse equivalences, and thus give some distinguished bijections between the sets of simple objects of the corresponding categories (see \cite[section 10.1]{Losev2}).  When $ Y = \mathrm{Gr}^\ul_\mu $, we expect that these bijections generate the action of the extended affine cactus group $ \widetilde{AC}_n $ (see \cite{HLLY} for a related result). Moreover, we expect a bijection between the set of simple objects in category $\mathcal{O}$ and the set of eigenlines for the quantum cohomology algebra, which intertwines this affine cactus group action and the monodromy of eigenvectors for trigonometric Gaudin algebras (see Conjectures \ref{conj:monodromy} and \ref{conj:danilenko}).

\subsection{Organization} The paper is organized as follows. Section 2 contains some general background which will be useful.  In section 3, we discuss the different moduli spaces of curves and their relation to commutative subspaces of holonomy Lie algebra. In section 4 we recall the results about the homogeneous Gaudin model. In section 5 we define and study the trigonometric Gaudin model. In section 6 we discuss the inhomogeneous Gaudin model and describe its parameter space. In section 7 we include the trigonometric and inhomogeneous Gaudin models into one family. In section 8, we discuss the spectra of Gaudin models and give the sufficient conditions for the simple spectrum in terms of real forms of the moduli spaces of curves.

\subsection{Acknowledgements.} 
We thank Alexander Braverman, Pavel Etingof, Iva Halacheva, Tamas Hausel, Leo Jiang, Yu Li, Vasily Krylov, and Alexei Oblomkov for many comments and suggestions. We also thank the referee for their thoughtful feedback and constructive remarks, which helped improve the manuscript. The work of L.R. was supported by the Fondation Courtois. The work was accomplished during L.R.'s stay at Harvard University and MIT. L.R. is grateful to Harvard University (especially Dennis Gaitsgory) and MIT (especially Roman Bezrukavnikov and Pavel Etingof) for their hospitality. The research of A.I. is supported by the MSHE project No. FSMG-2024-0048.

\section{Notation and background}

\subsection{Simple Lie algebras}
Let  $\fg$ be a simple Lie algebra of rank $ r $.  Fix a triangular decomposition $ \fg = \fn_+ \oplus \fh \oplus \fn_-$ and let $ p = \dim \fn_+$ be the number of positive roots.  (The results of this paper also apply when $ \fg $ is semisimple, but some statements would need modification, specifically those concerning the quadratic parts of the subalgebras.)

Let $\Phi_1, \dots, \Phi_r$ be free generators of $S(\fg)^{\fg}$, where $ r = \rk \fg$. We assume that all $ \Phi_l$ are homogeneous and set $ d_l = \deg \Phi_l$. Fix a $\fg$-invariant non-degenerate symmetric bilinear form $(\cdot,\cdot)$ on $\fg$ and let $\{x_a\}, \{x^a\},\ a=1,\ldots,\dim\fg$ be dual bases of $\fg$ with respect to this bilinear form.
 Moreover, we assume that $ d_1 = 2 $, and that $ \Phi_1 $ is the quadratic Casimir corresponding to this bilinear form.   
 
There is a (non-canonical) isomorphism $ Z(U \fg) \cong S(\fg)^{\fg}$.  For each $ l = 1, \dots, r$, we will later choose $ \widetilde \Phi_l \in Z(U \fg)$ such that $ \widetilde \Phi_l \in U \fg^{(d_l)}$ and under the isomorphism $ \gr U \fg \cong S(\fg)$, $ \widetilde \Phi_l $ is mapped to $ \Phi_l $.  Here and throughout the paper, $ U \fg^{(k)} $ denotes the $k$th filtered part of $ U\fg $ with respect to the PBW filtration.

We have 
$$
\Phi_1 = \sum_a x_a x^a \in S^2(\fg) \quad \widetilde \Phi_1 = \sum_a x_a x^a \in U \fg
$$

We also represent the bilinear form by $ \Omega \in \fg \otimes \fg $.  As in the introduction, we decompose $ \Omega = \Omega_+ + \Omega_0 + \Omega_-$ so that 
$$ \Omega = \sum_a x_a \otimes x^a \quad \Omega_- = \sum_c e^c \otimes e_c $$
where $ \{e_c \} \subset \{x_a \} $ denotes a basis for $ \fn_+ $ with corresponding dual basis $ \{e^c \} $ for $ \fn_-$.

We will use $ \omega $ for the element $ \widetilde \Phi_1 \in U \fg$, and we have a similar decomposition $ \omega = \omega_- + \omega_0 + \omega_+ $ where $ \omega_- = \sum e^c e_c \in \fn_- \fn_+$.

 We also note the identity $ \sum_{l = 1}^r d_l = r + p $, which can be deduced from the Chevalley-Solomon formula.
 
\subsection{Set partitions and embeddings of tensor products} \label{se:Delta}
Throughout this paper, $ n $ will be a natural number and we write $ [n] := \{1, \dots, n\}$.

For each finite set $ S $, let $ p(S) $ denote the set of pairs $ (i,j)$ of distinct elements of $ S $.  We will abuse notation by abbreviating $ (i,j) $ to $ ij$.  Similarly, we write $ t(S) $ for the set of triples $ (i,j,k) $ (abbreviated to $ ijk$) of distinct elements of $ S $.

An \emph{ordered set partition} of $ [n] $ is a sequence $ \CB = (B_1, \dots, B_r)$ of subsets of $ [n] $ such that $ B_1 \sqcup \cdots \sqcup B_r = [n] $. 

Given an ordered set partition $ \CB = (B_1, \dots, B_r)$ of $ [n]$, we define a diagonal embedding of Lie algebras
\begin{gather*} \Delta^{\CB}: \fg^{\oplus r} \rightarrow  \fg^{\oplus n} \\
\Delta^{\CB}(x_1, \dots, x_r) = (y_1, \dots, y_n) \text{ where } y_k = x_j \text{ if } k \in B_j 
\end{gather*}
This extends to a map of the universal enveloping algebras $$ \Delta^{\CB} : U \fg ^{\otimes r} = U (\fg^{\oplus r}) \rightarrow U(\fg^{\oplus n}) = U \fg^{\otimes n}$$  Note that when $ \CB = \{ \{1, 2\} \} $, then $ \Delta^\CB : U\fg \rightarrow U \fg^{\otimes 2}$ is the trivial partition of $ [2]$, then this is the usual coproduct.  More generally, $ \Delta^{\{[n]\}} : U \fg \rightarrow U \fg^{\otimes n}$ is the usual iterated coproduct, which we write simply as $ \Delta$.

For any algebra $ A$, $ a \in A$, and $ 1 \le i \le n$, let $a^{(i)} = 1 \otimes \cdots \otimes a \otimes \cdots \otimes 1 \in A^{\otimes n} $, where we insert $ a $ in the $i$th tensor factor.  The map $ \Delta^{\CB} : U \fg^{\otimes r} \rightarrow U \fg^{\otimes n}$ is characterized by
$$
\Delta^\CB(x^{(j)}) = \sum_{k \in B_j} x^{(k)}, \quad \text{ for } x \in \fg.
$$

Given distinct indices $ I = ( i_1, \dots, i_k) $, we define an injective algebra morphism $$ j_I^n : A^{\otimes k} \rightarrow A^{\otimes n } \quad a_1 \otimes \cdots \otimes a_k \mapsto a_1^{(i_1)} \cdots a_k^{(i_k)}$$

\subsection{Forests}
In this paper, we will use some combinatorics of planar binary forests.  A \emph{tree} is a connected graph without cycles.  A \emph{forest} is a disjoint union of trees.  Given two vertices $ v, w $ of a tree, there is a unique embedded edge path connecting them: we call this \emph{the path} between $v$ and $w$.  A \emph{rooted tree} is a tree with a distinguished vertex, called the \emph{root}, contained in exactly one edge, and with no vertices contained in exactly two edges. Given a non-root vertex $v$ of a rooted tree $\tau$,
the unique edge containing $v$ that lies on the path between $v$ and the root is \emph{descending at $v$}. The remaining edges containing $v$ are \emph{ascending at $v$}. A vertex with no ascending edges is a \emph{leaf}. A vertex $v$  is \emph{above} a vertex $ w$, if the path between $v$ and passes through $w$.  This defines a partial order on the set of vertices.  The \emph{meet} of two leaves $i, j$ is the maximal vertex $ v $ below $ i $ and $ j $.  

A vertex of a rooted forest is \emph{internal} if it is neither a leaf, nor a root.  A \emph{binary tree} is a rooted tree in which every internal vertex is contained in exactly three edges (one descending and two ascending).  A \emph{planar binary tree} is a binary tree such that for each internal vertex, one of the ascending edges is designated as the ``left'' edge and the other is designated as the ``right'' edge.  

A \emph{planar binary forest} is an ordered list $ \tau = (\tau_1, \dots, \tau_m) $ of planar binary trees.  Finally a planar binary forest is \emph{labelled} by a set $ S $ if we are given a bijection from $ S $ to the set of leaves of $ \tau$.  Note that if $ \tau $ consists of $ m $ trees, then $\tau$ determines an ordered set partition of $ S$ with $ m $ parts.  From now on, every planar binary forest will be labelled by $[n]$.

A planar binary forest $ \tau $ also determines a total order on $ [n]$ where $ i < j $ if either $ i,j $ lie in different trees $ \tau_a, \tau_b $ with $ a < b $, or $i,j $ lie in the same tree and $ i $ is on the left branch of the meet of $ i,j$.  Given any internal vertex $ v $ of $ \tau$, there is a unique pair of leaves $ p, q $ with $ p < q$ consecutive in the total order whose meet is $ v $; more specifically $ p $ is the rightmost (greatest) leaf on the left branch above $ v $ and $ q $ is the leftmost (least) leaf on the right branch above $ v $.  The following observation will be very useful to us.

\begin{lem} \label{le:biject}
    Let $ \tau $ be a planar binary tree.  There is a bijection from non-leaf vertices of $ \tau$ to leaves of $ \tau$ defined by $ v \mapsto p $ where
    $$ p := \begin{cases} \text{the rightmost leaf on the left branch above $ v $, } \quad \text{ if $ v $ is internal} \\
    \text{the rightmost leaf of $ \tau$, } \quad \text{ if $ v $ is the root}     
    \end{cases}
    $$
    This extends in the obvious way to planar binary forests.  
\end{lem}

\subsection{Families of subalgebras} \label{se:Family}
In this paper, we will study schemes which parameterize families of subalgebras of a given algebra.  We will now make this notion precise.  We begin with subspaces of a vector space.

Let $k $ be a commutative ring and let $ W $ be a free finite-rank $ k $ module. Let $ X $ be a scheme over $ k $.  A \emph{family of $r$-dimensional subspaces} of $ W $ parametrized by $X $ is a rank $ r$ locally free sheaf $ \mathcal V \subset \CO_X \otimes_k W $ such that $ (\CO_X \otimes_k W) / \mathcal V $ is locally free.  Equivalently, we have a morphism of $k$-schemes $ X \rightarrow\Gr(r, W)$.  In particular, for each closed point $ p \in X$, we have a dimension $ r $ subspace $$ \mathcal V(p) := \CV_p / \mathfrak m_p \CV_p  \subset k(p) \otimes_k W, \text{ where }  k(p) = \CO_{X, p}/ \mathfrak m_p. $$ 
If $ k = \C$, then $k(p) = \C$ as well, and $ \CV(p)$ is a subspace of $ W $.  On the other hand, we will later consider an example where $ k = \C[\hbar]$ is a polynomial ring, in which case $ X \rightarrow \mathbb A^1$ and $ \CV(p) $ is a subspace of the specialization $ \C[\hbar]/(\hbar - a) \otimes_{\C[\hbar]} W$ (where $ a $ is the image of $ p $ in $ \mathbb A^1$). The family is called \emph{faithful} if the morphism $ X \rightarrow \Gr(r, W)$ is an embedding.

Now, let $ R $ be a (not necessarily commutative, not necessarily finite) $k$-algebra.  A \emph{family of subalgebras of $ R $ parameterized by $ X $} is a locally free sheaf of subalgebras $ \CA \subset \CO_X \otimes_k R $.  
In this case, for each closed point $ p \in X$, we get a subalgebra $ \CA(p) := \CA_p / \mathfrak m_p \CA_p \subset k(p) \otimes_k R $, where $ k(p) = \CO_{X,p}/ \mathfrak m_p$.  

In order to bring finite-dimensional geometry into play, we now assume that $ R $ is equipped with an increasing filtration $ R^{(1)} \subset R^{(2)} \subset \cdots$, where each $ R^{(n)} $ is a free $k$-module of finite rank.  We assume that this filtration is exhaustive, i.e. $ \cup R^{(n)} = R $.   We say that $ \CA $ is \textit{compatible with the filtration}, if we are given locally free sheaves $ \CA^{(n)} \subset \CO_X \otimes_k R^{(n)} $, such that $ \cup \CA^{(n)} = \CA $.  We do not, however, demand that $ \CA^{(n)} = \CA \cap \CO_X \otimes_k R^{(n)}$.  We say that $\CA $ is a \emph{flat family} if the  quotients $ \bigl( \CO_X \otimes_k R^{(n)} \bigr) / \CA^{(n)}$ are locally free for all $ n $ (note that the definition of flatness depends on the filtration).    Let $r_n$ denote the rank of $ \CA^{(n)}$.  In this case, we obtain a morphism $X \rightarrow\Gr(r_n, R^{(n)}) $ to a Grassmannian.

Note that $ \CA^{(n)} $ is flat if and only if for each closed point $ p \in X$, the map $$ \CA^{(n)}(p) \rightarrow k(p) \otimes_k R^{(n)}$$
is injective.  As the left hand side is a $ k(p)$ vector space of dimension $ r_n$, we can test for flatness by checking if the dimension of the image of this map is $ r_n$.

Assume that we have a flat family of subalgebras parametrized by a scheme $ X$.  For each $ n$, we have the morphism $ X \rightarrow\Gr(r_n, R^{(n)})$, which we use to construct a morphism $ X \rightarrow \prod_{k \le n}\Gr(r_k, R^{(k)})$.  We define $ X_n $ to be the image of $ X $ in this product.  Note that we have surjective morphisms $ X_{n+1} \rightarrow X_n$ for each $ n $.   We say that this parametrization is \textit{faithful} if $X $ is the projective limit of this system of schemes.  Note that if $ X \rightarrow X_n $ is an isomorphism for some $ n $, then $ X \rightarrow X_k $ is automatically an isomorphism for all $k \ge n $ and $ X $ is the projective limit.  Note also that if we have a faithful parametrization, then the subalgebras $ \CA(p), \CA(q) $ indexed by distinct points of $p, q \in X $ are different.

Suppose that $ X, \CA_X $ is a flat family of subalgebras with compatible flat filtrations.  We say that another such family of subalgebras $ Y, \CA_Y$ is a \textit{compactification} of the family $ \CA_X$ if
\begin{enumerate}
    \item $Y $ is a compactification of $ X $.  This means that $ Y $ is proper over $k $ and we are given an embedding of $X $ into $ Y $ as an open dense subscheme.
    \item The restriction of $ \CA_Y$ to $ X $ equals $\CA_X$ and for each $ n $, the restriction of $ \CA_Y^{(n)} $ to $ X $ equals $ \CA_X^{(n)}$.
\end{enumerate}

In particular this means that the map $ X \rightarrow \mathrm{Gr}(r_n, R^{(n)}) $ factors through the map $ Y \rightarrow \mathrm{Gr}(r_n, R^{(n)})$.  Since $ Y $ is proper and $ X $ is dense in $ Y$, the image of $ Y $ in $ \mathrm{Gr}(r_n, R^{(n)}) $ is the closure of the image of $ X$.  

Suppose that $  X, \CA_X $ and $ Y, \CA_Y $ are as above and they are both faithful parametrizations.  Then $ Y, \CA_Y $ is uniquely determined by $ X, \CA_X $ (both the scheme and the family of filtered subalgebras), so it is natural to call $ Y, \CA_Y$ \textit{the compactification} of $ X, \CA_X$.  We note that the compactification can depend of the choice of filtration.

\subsection{Algebras over $ \C[\hbar]$}
In this paper, we will often work with algebras over $ \C[\hbar]$.  Given an algebra $ R $ over $ \C[\hbar]$, and $ \varepsilon \in \C$, we let $ R(\varepsilon) := R \otimes_{\C[\hbar]} \C[\hbar]/(\hbar - \varepsilon)$ be the specialization at $ \hbar = \varepsilon$.  Similarly, if $ f : R \rightarrow S $ is a map of $ \C[\hbar]$-algebras, we let $ f(\varepsilon) : R(\varepsilon) \rightarrow S(\varepsilon) $ be the map between their specializations.   Similarly, for any scheme $ X $ defined over $ \mathbb A^1$, we write $ X(\varepsilon) $ for the fibre over $ \varepsilon \in \C$.

We will also need the Rees construction.  Suppose that $ R $ is a $ \mathbb N$-filtered algebra with filtered pieces $ R^{(k)}$, we let $ \Rees(R) = \oplus_{k \in \mathbb N} \hbar^k R^{(k)} \subset \C[\hbar]\otimes R $.  We have $ \Rees(R)(\varepsilon) \cong R$ for $ \varepsilon\ne 0$, while $ \Rees(R)(0) = \gr R = \oplus_{k \in \mathbb N} R^{(k)}/ R^{(k-1)}$.

\subsection{Quantum Hamiltonian reduction}
\label{se:QHR}

Let $A$ be an algebra and $J \subset A$ be a left ideal. Consider the normalizer of $J$ inside $A$, the biggest subspace of $A$ where $J$ is a two-sided ideal:
$$N_A(J) = \{ a \in A \, | \, Ja \subset J \}.$$
We call the algebra $N_A(J) / J$ the {\em quantum Hamiltonian reduction} of $A$ with respect to $J$.  Note that there is a natural algebra isomorphism $ N_A(J)/J \cong \operatorname{End}_A(A/J)^{op} $.  Thus for any $A$-module $M$, both the subspace of \textit{$J$-invariants} $M^J:=\{m\in M\ |\ Jm=0\}$ and \textit{$J$-coinvariants}, $M_J:=M/JM $ are naturally $N_A(J) / J$-modules.

The following special case will be particularly useful to us.  Suppose that $\fg $ is a Lie algebra and that $ \Phi : \fg \rightarrow A $ satisfies $ \Phi([x,y]) = [\Phi(x), \Phi(y)]$ for all $x,y \in \fg $ ($\Phi$ is called a {\em quantum comoment map}).  Suppose also that we are given a character $ \theta : \fg \rightarrow \bc $.  Then we define $ J \subset A $ to be the left ideal generated by $ \{ \Phi(x) - \theta(x) : x \in \fg \}$.

\begin{prop}
In the above setting, we have an isomorphism 
$$N_A(J) /J  \cong (A / J)^\fg = \{ [a] \in A/J : \Phi(x)a - a \Phi(x) \in J  \text{ for all } x \in \fg \}$$
\end{prop}

\begin{proof}
There is a natural injective map $ N_A(J)/J \rightarrow A/J$.  In fact, the image of this map is contained in $ (A/J)^\fg$ since if $ a \in N_A(J)$ and $ x \in \fg$, then $a (\Phi(x) - \theta(x)) , (\Phi(x) - \theta(x))a \in J $.   Hence $ \Phi(x)a - a \Phi(x) \in J $ as desired.  

So it suffices to show that its image is exactly $(A/J)^\fg$.  Let $ [a] \in (A/J)^\fg$, let $ b \in A $ and let $ x \in \fg $.  Then $ \Phi(x) a - a \Phi(x) \in J $ which implies that $ (\Phi(x) - \theta(x))a - a (\Phi(x) - \theta(x)) \in J$. Multiplying by $ b $ and using that $ba( \Phi(x) -\theta(x))\in J $, we conclude that $ b(\Phi(x) - \theta(x)) a \in J$.  
  Since every element of $J $ is a sum of terms of the form $ b (\Phi(x) - \theta(x))$, we conclude that $ a \in N_A(J)$ as desired.
\end{proof}

We will be mostly concerned with the following situation.  Let $ \widetilde \fg $ be another Lie algebra with $ \fg \subset \widetilde \fg$.  Set $ A = U \widetilde \fg $ and let $ \Phi : \fg \rightarrow A $ be the obvious embedding.  Suppose we are given a complementary subalgebra $ \fk \subset \widetilde \fg$, so that $ \widetilde \fg = \fg \oplus \fk $ as a vector space (but we do not assume that $ \fk$ nor $ \fg $ are ideals).    
\begin{prop} \label{pr:HamReductionToU}
The natural map $ U \fk \rightarrow A $ gives a vector space isomorphism $ U \fk \rightarrow A/J$.  Inverting this isomorphism and restricting to $ N_A(J)$ gives an injective algebra morphism $ N_A(J)/J \rightarrow U \fk $.  The associated graded of this morphism is the morphism $ (S(\widetilde \fg)/ \langle \fg \rangle)^\fg \rightarrow S(\fk) $ which comes from the projection $ \widetilde \fg \rightarrow \fk$.
\end{prop}
\begin{proof}
    By the PBW theorem, we have a vector space decomposition $ U \fk \oplus J = A$ and so the result follows. 
\end{proof}

Assume now that $ \fk$ is an ideal.  In particular, this gives us an action of $ \fg$ on $ \fk$ and on $ U\fk$.  Then we can upgrade the previous Proposition as follows.

\begin{prop} \label{pr:HamReductiontoU2}
The morphism $ N_A(J)/J \rightarrow U\fk$ from Proposition \ref{pr:HamReductionToU} is an algebra isomorphism $ N_A(J)/J \cong (U\fk)^\fg$.
\end{prop}

\begin{proof}
    We simply observe that the isomorphism $ U\fk \rightarrow A/J $ is $ \fg $-equivariant and thus restricts an isomorphism between the invariants.
\end{proof}

\section{Moduli spaces of curves and commutative subspaces of Lie algebras}
\subsection{The Deligne-Mumford space and the cactus flower space}
We review the construction of the moduli space of cactus flower curves, following our paper \cite{iklpr}.
Let $$ M_{n+1} = (\C^n \setminus \Delta) / \Cx \ltimes \C = ((\Cx)^{n-1} \setminus \Delta) / \Cx $$
 be the space of $ n $ distinct points on $ \C $ up to affine linear transformations.  This space has a natural compactification  $ \overline M_{n+1} $, the Deligne-Mumford space of genus 0 stable curves with $ n+ 1$ marked points.  A point in $ \overline M_{n+1} $ is given by a \emph{cactus curve} $(C, \uz) $, a possibly reducible curve, carrying $n+1 $ marked points $z_1, \dots, z_{n+1} $, each of whose components is a projective line with at least 3 nodal or marked points, and whose components form a tree.  The embedding $ M_{n+1} \rightarrow \overline M_{n+1} $ is given by $ (z_1, \dots, z_n) \mapsto (\PP^1, (z_1, \dots, z_n, z_{n+1} = \infty))$

Given a triple of distinct indices $ i,j,k \in \{1, \dots, n\}$, there is a map $ \mu_{ijk} : M_{n+1} \rightarrow \PP^1 $ given by
$$
\mu_{ijk}(z_1, \dots, z_n) = \frac{z_i - z_k}{z_i - z_j}
$$

This ratio $\mu_{ijk} $ extends to a map $ \overline M_{n+1} \rightarrow \PP^1$ and we can use these maps to embed $ \overline M_{n+1} $ inside a product of projective lines.  

 The following result is due to Aguirre-Felder-Veselov \cite[Theorem A.2]{afv}, building on earlier work of Gerritzen-Herrlich-van der Put \cite{GHP}.
\begin{thm} \label{th:embedP1}
	The maps $ \mu_{ijk} $ embed $ \overline M_{n+1} $ as the subscheme of $ (\PP^1)^{t([n])} $  defined by
	\begin{align*}
		\mu_{ijk} \mu_{ikj} = 1 \quad \mu_{ijk} + \mu_{jik} = 1  \quad \mu_{ijk}\mu_{ilj} = \mu_{ilk}
	\end{align*}
for distinct $ i, j, k, l$.  (Here $ t([n]) $ is the set of triples, see \S \ref{se:Delta}.)
\end{thm}

There is a natural stratification of $\overline M_{n+1}$, indexed by $[n]$-labelled, rooted trees.  Each stratum is a moduli space of nodal curves of the prescribed combinatorial type, which means we record the tree formed by the irreducible components and the labels of the marked points on each component. All these strata are locally closed subsets of $\overline M_{n+1}$ (they can be defined by imposing some $\mu_{ijk} = 0$ and some $\mu_{ijk} \not = 0$). The closure of each stratum is a union of strata of the same form. The dimension of a stratum is $n - k$, where $k$ is the number of internal vertices of the tree.

\subsection{Moduli space of cactus flower curves} \label{se:cactusflower}
\begin{figure}
	\includegraphics[trim=0 80 60 40, clip,width=\textwidth]{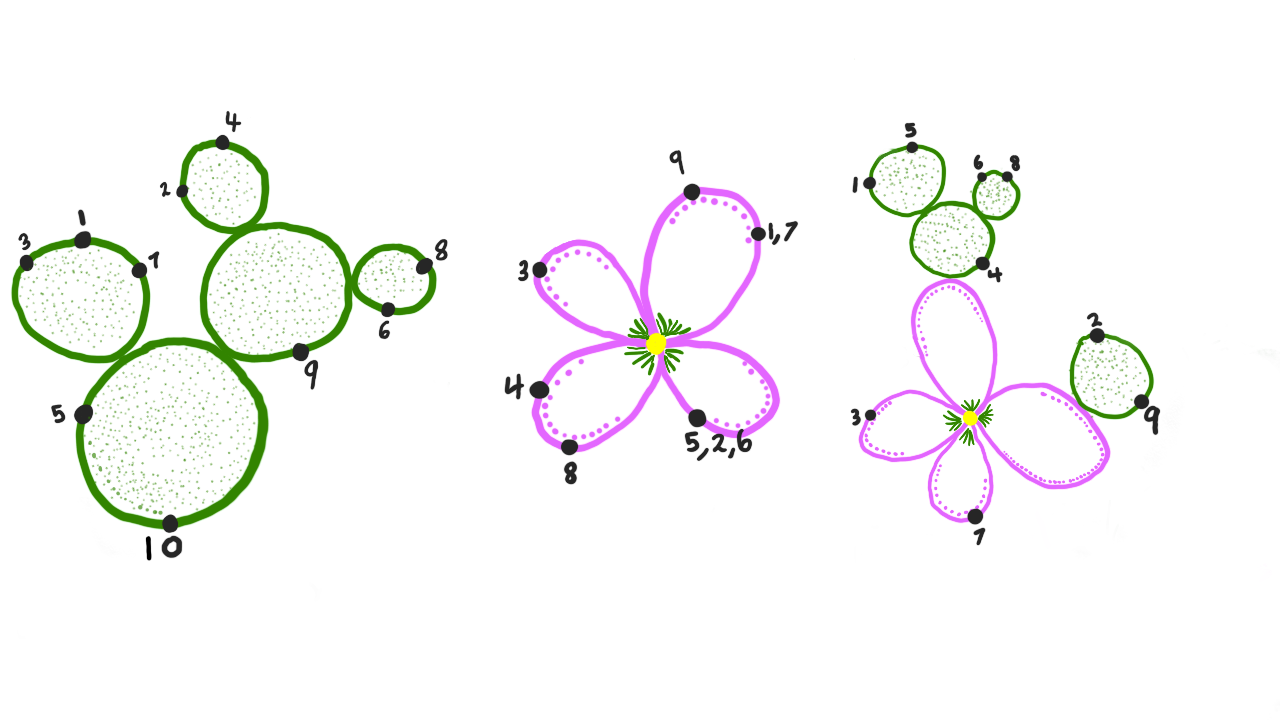}
	\caption{A point of $ \overline M_{9+1}$ (a cactus curve), a point of $ \overline \ft_9$ (a flower curve), and a point of $ \overline F_9$ (a cactus flower curve).}
\end{figure}

The map $ M_{n+1} = (\C^n \setminus \Delta) / \Cx \ltimes \C \hookrightarrow \PP(\ft_n) $, where $ \ft_n = \C^n / \C$, extends to a map $ \overline M_{n+1} \rightarrow \PP(\ft_n)$ given by collapsing all components not containing the marked point $n+1$.  Thus, the moduli space $ \overline M_{n+1} $ carries a line bundle $ \widetilde M_{n+1} \rightarrow \overline M_{n+1} $ given by pulling back the tautological line bundle from $ \PP(\ft_n) $.  A point of $ \widetilde M_{n+1} $ is given by a \emph{framed cactus curve}, which is a cactus curve $ (C, \uz) $, along with a non-zero tangent vector $ a \in T_{z_{n+1}} C $ at the ``infinity'' marked point, which we call the \emph{distinguished} point (see \cite[Remark 4.8]{iklpr}). Note that this non-zero tangent vector has the effect of reducing the automorphism group of the component of $ C $ containing $ z_{n+1} $. In this way we get an embedding
$$ F_n := (\C^n \setminus \Delta) / \C \rightarrow \widetilde M_{n+1}$$
given by $ (z_1, \dots, z_n) \mapsto (\PP^1, (z_1, \dots, z_n, z_{n+1} = \infty), a = 1) $.   In our recent paper \cite{iklpr}, we defined a compactification $ \overline F_n $ of $ F_n $ which contains $ \widetilde M_{n+1} $ as an open subscheme.  A point of $ \overline F_n $ is a \emph{cactus flower curve}, which is a curve $ C = C_1 \cup \dots \cup C_m $, where each $ C_j $ is a framed cactus curve; these cactus curves are glued together at their distinguished point, and altogether carry $n $ distinct marked points.

More precisely, we construct $ \overline F_n $ by starting with the matroid Schubert variety $ \overline \ft_n $, \cite[\S 3.2]{iklpr}.  On $ \ft_n = \C^n/\C $, we have the natural functions $ \delta_{ij} = z_i - z_j $, which collectively give a linear embedding $ \ft_n \hookrightarrow \C^{p([n])} $,  where $ p([n]) $ is the set of pairs, see \S \ref{se:Delta}.  Then  $ \overline \ft_n $ is the closure of $ \ft_n $ inside $ (\PP^1)^{p([n])} $.  

More explicitly, $ \overline \ft_n $ is the subscheme of $(\PP^1)^{p([n])} $ defined by the equations 
$$
\delta_{ij} + \delta_{jk} = \delta_{ik} \qquad \delta_{ij} + \delta_{ji} = 0
$$
for distinct $ i,j,k$.  A point of $ \overline \ft_n $ can be visualized as a \emph{flower curve}, $ C = C_1 \cup \dots \cup C_m $, where each $ C_j $ is a framed projective line, these are glued together at their distinguished point, and altogether carry $n $ (not necessarily distinct) marked points.  

Equivalently, we can set $ \nu_{ij} = \delta_{ij}^{-1} $.  In these coordinates, the defining equations of $ \overline \ft_n $ become
\begin{equation} \label{eq:ftnd}
\nu_{ij} \nu_{jk} = \nu_{ik} \nu_{jk} + \nu_{ij} \nu_{ik} \qquad \nu_{ij} + \nu_{ji} = 0
\end{equation}

The matroid Schubert variety is stratified in two ways by set partitions $ \CS $ of $ [n] $.  Given set partitions $ \CS, \CB $, we have the strata
\begin{gather*}
	V_\CS = \{ \delta \in \overline \ft_n: \delta_{ij} = \infty \text{ if and only if $i \nsim_\CS j $} \} \\
	V^\CB = \{ \delta \in \overline \ft_n: \delta_{ij} = 0 \text{ if and only if $i \sim_\CB j $} \} \\
\end{gather*}

For any set partition $ \CS $, we define an open subscheme $ U_\CS $ containing $ V_\CS $ by
$$
U_\CS = \{ \delta \in \overline \ft_n : \delta_{ij} \ne 0 \text{ if $ i \nsim_\CS j$}, \delta_{ij} \ne \infty \text{ if $ i \sim_\CS j$} \}
$$ 
There is a morphism $ U_\CS \rightarrow \prod_{k=1}^m \ft_{S_k} $ and we define $ \tU_\CS := U_\CS \times_{\prod_k \ft_{S_k}} \prod_k \widetilde M_{S_k +1}$.  Here we use $ \ft_S = \C^S/\C $ and $ \widetilde M_{S +1} $ for the versions of these spaces where the marked points are labelled by the finite set $ S$.

More explicitly, $ \widetilde U_\CS $ is the subscheme of $ (\nu, \mu^1, \dots, \mu^m) \in (\PP^1)^{p([n])} \times \prod_k (\PP^1)^{t(S_k)}$ such that
\begin{itemize}
	\item $ \nu_{ij} $ satisfy the ``non-vanishing'' conditions given in the definition of $ U_\CS $
	\item all the equations from (7) of \cite{iklpr} hold, whenever they make sense.
\end{itemize}

In \cite[\S 6]{iklpr}, we define $ \overline F_n $ by gluing together the schemes $ \tU_\CS $ and we proved the following.
\begin{thm}
    $\overline F_n $ is proper and reduced.  It contains $ F_n$ as a dense subset and is covered by the open subsets $ \{ \widetilde U_\CS \} $.
\end{thm}

The space $ \overline F_n $  comes with a map $ \overline F_n \rightarrow \overline \ft_n $, such that the fibre over a point $ \delta \in V^\CB $ is the product $ \overline M_{B_1 + 1} \times \cdots \times \overline M_{B_m+1} $ of Deligne-Mumford spaces; this corresponds to collapsing a cactus flower curve to a flower curve.


\subsection{Degeneration} \label{se:deg}
In \cite{iklpr}, we also defined a degeneration of $ \overline M_{n+2} $ to $ \overline F_n $.  This degeneration compactifies a degeneration of $ M_{n+2} $ to $ F_n $, which in turn is built on a degeneration of the multiplicative group $ \Cx $ to the additive group $ \C $.  

We  recall the group scheme $ \BG \rightarrow \mathbb A^1 $ studied in \cite{iklpr}. We let 
$$ \BG= \{(z; \varepsilon) \in \C^2 : 1 -  \varepsilon z \ne 0 \} $$
with multiplication given by $ (z_1; \varepsilon) (z_2; \varepsilon) = (z_1 + z_2 - \varepsilon z_1 z_2; \varepsilon) $.  It is a group scheme over $ \mathbb A^1 $ with $ \BG(\varepsilon) \cong \Cx$ for $ \varepsilon \ne 0 $ (given by $ z \mapsto 1 - \varepsilon z$) and $ \BG(0) \cong \C$. 

We will work with $$ \BG^n \setminus \Delta  = \{(z_1, \dots, z_n; \varepsilon) : 1 - \varepsilon z_i \ne 0, z_i \ne z_j \} $$ 
Here the products of $ \BG$ are taken over $ \mathbb A^1$; that is why there is only one $ \varepsilon$ coordinate in $ \BG^n$.  We let $ \mathcal F_n  = (\BG^n \setminus \Delta) / \BG$ and $ \Cf_n = \BG^n/\BG$.

In particular, the coordinate ring of $ \CF_n $ is given by
$$
\C[\mathcal F_n] = \C[\hbar, u_1, \dots, u_n, (1 - \hbar u_i)^{-1}, (u_i - u_j)^{-1}]^\BG $$
and it contains the invertible elements
$$
\nu_{ij} = \frac{ 1- \hbar u_j}{u_i - u_j}
$$

We have isomorphisms $ \mathcal F_n(\varepsilon) \cong M_{n+2} = ((\C^\times)^n \setminus \Delta) / \C^\times$ for $ \varepsilon \ne 0 $ and $ \mathcal F_n(0) \cong F_n = (\C^n \setminus \Delta) / \C$, defined by
$$
(z_1, \dots, z_n; \varepsilon) \mapsto (1 - \varepsilon z_1, \dots, 1 - \varepsilon  z_n) \quad (z_1, \dots, z_n; 0) \mapsto (z_1, \dots, z_n)
$$

In \cite{iklpr}, we defined fibrewise compactifications $ \overline \Cf_n $ and $ \overline \CF_n $  which fit into the diagram 
 \begin{equation} \label{eq:maindiagram}
	\begin{tikzcd}
		\overline M_{n+2} \ar[r] \ar[d] & \overline \CF_n \ar[d]  & \ar[l] \ar[d] \overline F_n \\
		\overline T_n \ar[r] \ar[d] & \overline \Cf_n \ar[d] & \ar[l] \overline \ft_n \ar[d] \\
		\{\varepsilon \ne 0 \} \ar[r] & \BA^1 & \ar[l] \{0 \}
	\end{tikzcd}
\end{equation}
  Here $ \overline T_n $ is the Losev-Manin space, which is the same as the toric variety associated to the permutahedron.

The variety $ \overline \Cf_n $ is defined as the subscheme of $ (\nu,\varepsilon) \in (\PP^1)^{p([n])} \times \BA^1 $  defined by the equations
\begin{equation} \label{eq:ftn}
\varepsilon \nu_{ik} + \nu_{ij} \nu_{jk} = \nu_{ik} \nu_{jk} + \nu_{ij} \nu_{ik} \qquad \nu_{ij} + \nu_{ji} = \hbar
\end{equation}
for all distinct $ i,j,k$.

\subsection{An open affine cover} \label{se:Wtau}

We constructed  $ \overline \CF_n $ in \cite{iklpr} as follows.  First, for each set partition $ \CS$, deforming $ U_S$, we defined the open affine subscheme $ \CU_\CS \subset \overline \Cf_n$,
$$ 
\CU_\CS =  \{ (\nu, \varepsilon) \in \overline \Cf_n: \nu_{ij} \ne \infty \text{ if $ i \nsim_\CS j $}, \ \nu_{ij} \ne 0, \varepsilon \text{ if $i \sim_\CS j) $}  \}
$$

Then we constructed $ \tCU_\CS \subset (\bp^1)^{p([n])} \times \prod_k (\bp^1)^{t(S_k)} \times \BA^1 $ a deformed version of $ \tU_\CS$.  Finally, we defined $ \overline \CF_n $ by glueing these together.  These schemes $ \tCU_\CS$ are not in general affine.

Given any planar binary forest $ \tau$, in \cite[\S 6.5]{iklpr} we defined open affine subschemes $ W_\tau \subset  \overline F_n $ and $ \CW_\tau \subset \overline \CF_n$ (in fact, these do not depend on the planar structure on the forest, but for later use in this paper the planar structure will be important).  The forest $ \tau $ indexes a dimension 0 stratum of $ \overline F_n$, corresponding to the cactus flower curve whose components are labelled by the non-leaf vertices of $ \tau$, and, roughly speaking, $ \CW_\tau$ is the set of curves which can be degenerated to this curve.  More precisely, we let $ \CS $ be the set partition of $ [n] $ given by $\tau$ and then we define
$$
\CW_\tau = \{(\nu, \mu) \in \tCU_\CS : \mu_{ijk} \ne \infty \text{ if the meet of $ i,k $ is above the meet of $ i, j$ in $\tau $} \}
$$

The following functions on $ \CF_n$ extend to regular functions on $ \CW_\tau$
\begin{itemize}
\item $\nu_{ij} = \frac{1 - \hbar z_j}{z_i - z_j} $ for $ i, j $ in different trees,
\item $ \delta_{ij} = \frac{z_i - z_j}{1 - \hbar z_j} $ for $ i , j $ in the same tree,
\item $ \delta_{ij} \nu_{kl} = \frac{(z_i - z_j)(1- \hbar z_l)}{(z_k - z_l)(1-\hbar z_j)}$ for $ i,j, k, l $ in the same tree, such that the meet of $ i,j $ is weakly above the meet of $ k,l$ (see \cite[Lemma 6.16]{iklpr}). 
\end{itemize}

\begin{figure}
	\includegraphics[trim=50 280 100 40, clip,width=0.7\textwidth]{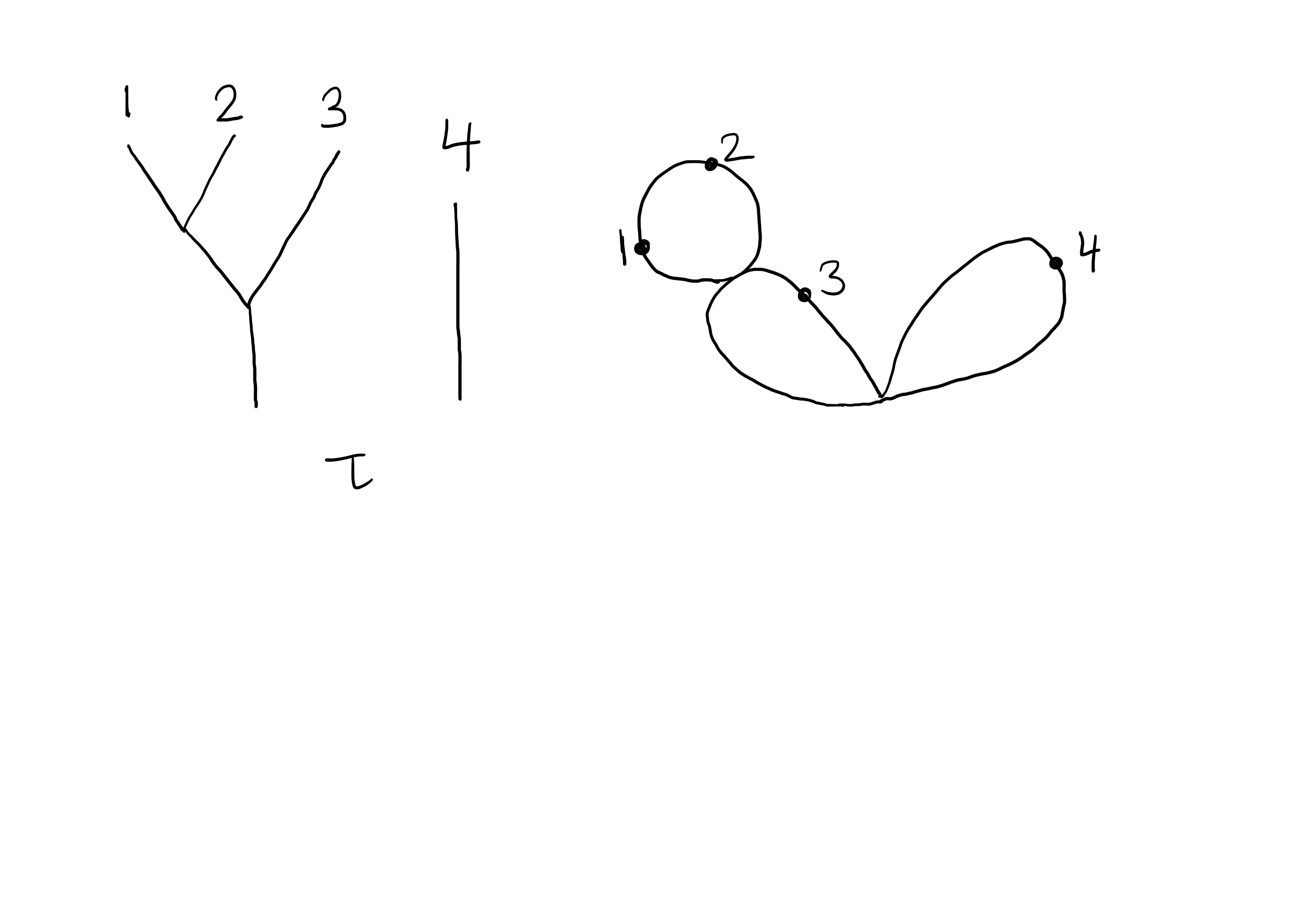}
	\caption{A planar binary forest $ \tau$ and the corresponding cactus flower curve.  On $ \CW_\tau$, we have the following regular functions: $ \nu_{14}, \nu_{24}, \nu_{34},$ $ \delta_{12}, \delta_{13}, \delta_{23},$ $ \mu_{142}, \mu_{143}, \mu_{243}, \mu_{132}$.} \label{fig:tau}
\end{figure}

Now suppose that $i,j,k \in t([n]) $ such that either
\begin{itemize}
    \item $i,j$ are in the same tree and $k $ is from a different tree, or
    \item $i,j,k $ are in the same tree and the meet of $ i,j $ is weakly above the meet of $ i,k$
\end{itemize}
By a slight abuse of notation, we will write $ \mu_{ikj}$ for the regular function on $ \CW_\tau$ which extends $ \frac{z_i - z_j}{z_i - z_k} = \delta_{ji} \nu_{ki}$.  We note the relation $ \mu_{ijk} \nu_{ki} = \nu_{ji} $.  See Figure \ref{fig:tau} for an example.

\subsection{Commutative subalgebras of the Drinfeld-Kohno Lie algebra} \label{se:afv}

Let $ \fs_n $ denote the Drinfeld-Kohno Lie algebra (also known as the holonomy Lie algebra associated to the type $A_{n-1}$ root hyperplane arrangement).   It is the Lie algebra generated by $t_{ij} = t_{ji}$, for $ i \ne j \in \{1, \ldots n\}$ with relations 
\begin{equation} 
[t_{ij}, t_{kl}] = 0 \text{ for $i,j,k,l $ distinct} \quad [t_{ij}, t_{ik}] = -[t_{ij}, t_{jk}]  \text{ for $i,j,k$ distinct}
\end{equation}
This Lie algebra is graded, where all generators have degree 1 and we will be concerned with the first graded piece $ \fs_n^1 \subset \fs_n$, which has basis $ \{ t_{ij} \}$.  A subspace $ V \subset \fs_n^1$ is called \emph{commutative} if $ [x,y] = 0 $ for all $ x,y \in V$. 
Given $ z_1, \dots, z_n \in \bc $ distinct, we can define
$$h_i(\uz) = \sum_{j \ne i} \frac{t_{ij}}{z_i - z_j} \in \fs_n^1$$
These vectors satisfy $[h_i(\uz), h_j(\uz)] = 0 $ and $ \sum_i h_i(\uz) = 0 $.  Thus, they span an $n-1 $-dimensional commutative subspace $  Q(\uz) \subset \fs_n^1$.  It is easy to see that $ Q(\uz) $ only depends on the image of $ \uz $ in $ M_{n+1} $.  We let $ N_n = \{ Q(\uz) : \uz \in M_{n+1} \} \subset\Gr(n-1, \fs_n^1) $ be the set of all subspaces constructed in this manner.  Let $ \overline N_n $ be the closure of this locus.

\begin{rem}
Since commutativity is a closed condition, every $ V \in \overline N_n $ is commutative.  In fact, Aguirre-Felder-Veselov \cite{afv} showed that every maximal commutative subspace of $ \fs_n^1 $ has dimension $ n -1 $ and is contained in $ \overline N_n$.  Thus, we can regard $ \overline N_n $ as the space of all maximal commutative subspaces of $ \fs_n^1$.
\end{rem}

The following is the main result of \cite{afv}.
\begin{thm} \label{th:AFV}
    The map $ M_{n+1} \rightarrow N_n $ defined by $ \uz \mapsto Q(\uz)$
    extends to an isomorphism $ \overline M_{n+1} \cong \overline N_n$.
\end{thm}
In particular, for any $ C \in \overline M_{n+1}$, we may consider the subspace $ Q(C) \subset \fs_n^1$.

Using the coordinates $ \mu_{ijk} $ on $ \overline M_{n+1}$, it is possible to describe $ Q(C) $.  Given any $ V \in \overline N_n $, the projection onto the subspace spanned by $ t_{ij}, t_{jk}, t_{ik} $ is always 2-dimensional and contains $ t_{ij} + t_{jk} + t_{ik} $.  Therefore it defines a point in $ \PP^1 $ which is our coordinate $ \mu_{ijk}$.  

More precisely, define an isomorphism 
\begin{gather*} J : \PP^1 \rightarrow \{W \in\Gr(2, \BC^3) : (1,1,1) \in W \}  \\
J(\tfrac{a}{b}) := \{ (x,y,z) \in \BC^3 : bx - ay + (a-b)z = 0 \}
\end{gather*}
for $\tfrac{a}{b} = [a,b] \in \PP^1$.  Then given any triple $ ijk \in t([n]) $, we define $q_{ijk}: \fs_n^1 \rightarrow \BC^3 $ by 
\begin{equation}
\begin{gathered} \label{eq:qijk}
 q_{ijk}(t_{ij}) = e_1,\ q_{ijk}(t_{ik}) = e_2,\ q_{ijk}(t_{jk}) = e_3, \\ q_{ijk}(t_{ab}) = 0 \text{ if $ \{a,b\} \not\subset \{i,j,k\}$ }
 \end{gathered}
 \end{equation}
  From the proof of \cite[Theorem 2.5]{afv}, for any $ C \in \overline M_{n+1}$ with coordinates $ \mu \in (\PP^1)^{t([n])}$, we have
$$ 
Q(C) = \bigcap_{ijk \in t([n])} q_{ijk}^{-1} \left( J(\mu_{ijk}) \right)
$$

\subsection{Commutative subalgebras of the inhomogeneous Drinfeld-Kohno Lie algebra} \label{se:Fnfrn}
We will now prove an analog of Theorem \ref{th:AFV} for $ \overline F_n $.

Let $ \fr_n $ to be the Lie algebra with generators $ \{ t_{ij} : \{i,j\} \in \binom{[n]}{ 2} \} \cup \{ u_i : i \in [n] \} $ and relations
\begin{gather*}
[t_{ij}, t_{kl} ] = 0 \text{ for $ i,j,k,l $ distinct } \quad [t_{ij}, t_{ik}] = -[t_{ij} , t_{jk} ]  \text{ for $ i,j,k $ distinct } \\
 [t_{ij}, u_k] = 0 \text{ for $ i,j,k$ distinct } \quad [t_{ij}, u_i ] = -[t_{ij}, u_j]  \quad  [u_i, u_j] = 0 \text{ for $i,j$ distinct }
\end{gather*}
As before, this is graded and we will be interested in commutative subspaces of the first graded piece $ \fr_n^1$, which has a basis given by the generators.

\begin{rem}
    The definition of this Lie algebra looks a bit unnatural, but it has a generalization to any hyperplane arrangement which we will study in a future work.
\end{rem}

For $ z_1, \dots, z_n \in \bc $ distinct, we define 
\begin{equation} \label{eq:Hifr}
h_i^0(\uz) = u_i + \sum_{j \ne i} \frac{1}{z_i-z_j} t_{ij}
\end{equation}
Let $ Q(\uz) $ be the span of these vectors.  This subspace depends only on the image of $ \uz$ in $ F_n $ and thus we have  a map $ F_n \rightarrow \Gr(n, \fr_n^1) $ defined by $ \uz \mapsto Q(\uz) $.  Note that $ Q(\uz) $ always contains $ \sum_i u_i = \sum_i h_i^0(\uz) $ which is central in $ \fr_n $.

\begin{lem}
$Q(\uz) $ is a commutative subspace.
\end{lem}
\begin{proof}
We must check that $ [h_i^0(\uz), h_j^0(\uz)] = 0 $.  In particular, for every basis element of $ \fs_{n}^2 $ (which is $ \binom{n}{3} + \binom{n}{2}$ dimensional), we check that the corresponding coefficient is 0.  For example, we can consider the coefficient of $ [t_{ij}, u_i] = -[t_{ij}, u_j]  $ which is $\frac{1}{z_i - z_j} + \frac{1}{z_j - z_i}  = 0 $. The remaining coefficients are checked in a similar manner.
\end{proof}

Let $ G_n \subset \Gr(n, \fr_n^1)  $ be the image of $ F_n $ and let $ \overline G_n $ be its closure in $ \Gr(n, \fr_n^1) $.  Note that every $ V \in \overline G_n $ is commutative since being commutative is a closed condition.
\begin{rem}
It is not clear if all maximal commutative subspaces of $ \fr_n^1 $ lie in $ \overline G_n $.  Note that the analogous statement for holonomy Lie algebras outside type A is false, as shown in \cite{afv2}.
\end{rem}

We will now extend $ F_n \rightarrow G_n $  to a map $ \overline F_n \rightarrow \overline G_n$.  

Let $ \tau $ be a planar rooted binary forest and recall the open affine subscheme $ W_\tau \subset \overline F_n $ from section \ref{se:Wtau}.  For $ C \in W_\tau$, we will define a collection of vectors $h_v^0(C) $ depending on the internal vertices of $ \tau $ and then we will define $$ Q(C) = \spann (h_v^0(C) : v \text{ is an internal vertex of } \tau ) $$

If $ v $ is not a root, then let $ p, q $ be the consecutive leaves whose meet is $ v $ (i.e. $ p $ is the rightmost leaf in the left branch above $ v $ and $ q $ is the leftmost leaf in the right branch above $ v $) and we define 
\begin{equation} \label{eq:Hv1}
h_v^0(C) = \sum_i \delta_{p q} u_i + \sum_{i,j} \delta_{p q} \nu_{ij} t_{ij}
\end{equation} 
where $ i $ ranges over all leaves in the left branch above $ v $ and $ j $ ranges over all other leaves in the forest.

If $ v $ is a root, then we define
\begin{equation} \label{eq:Hv2}
h_v^0(C) = \sum_i u_i + \sum_{i,j} \nu_{ij} t_{ij} 
\end{equation}
where $i $ ranges over the leaves in the tree of $ v $ and $ j $ ranges over the leaves in other trees.

\begin{lem} \label{le:VtauVz}
If $ C = (\PP^1, z_1, \dots, z_n, z_{n+1} = \infty) \in F_n $, then $ Q(C) = Q(\uz) $.
\end{lem}
\begin{proof}
Let $ v $ be an internal vertex.  Then $ h_v^0(C) = \delta_{pq} \sum_i h_i^0(\uz) $ where $ i $ ranges over  all leaves in the left branch above $ v $, and $ p, q $ are the consecutive leaves meeting at $ v $.

Similarly, if $ v $ is a root, then $ h_v^0(C) = \sum_i h_i^0(\uz) $ where $ i $ ranges over all the leaves in the tree of $ v$.

Since $ \delta_{pq} = z_p - z_q $ are all non-zero, this is an upper triangular change of basis with respect to the partial order on the vertices and the bijection between vertices and leaves, which takes a vertex $ v $ to the rightmost leaf in the left branch above $ v $ (see Lemma \ref{le:biject}).
\end{proof}

\begin{lem} \label{le:linindepVC}
For any $ C \in W_\tau  $, $ \{ h_v^0(C) \}$ is linearly independent and thus $ \dim Q(C) = n $.
\end{lem}
\begin{proof}
Again we consider the natural partial order on vertices and the bijection between vertices and leaves.  If $ p,q $ are consecutive leaves whose meet is $ v $, then $ t_{p q} $ appears in $ h_v^0 $ with coefficient 1 and does not appear in any $ h_{v'}^0 $ for $ v' $ below $ v $.  This implies the linear independence.
\end{proof}


Combining together these lemmas and using that  $ \overline F_n $ is a compactification of $ F_n$, we conclude the following.
\begin{prop} \label{pr:FnGnsurject}
For any $ C \in \overline F_n $, $ Q(C) $ is an $ n$-dimensional subspace of $ \fr_n^1 $ containing $ \sum u_i$.  The map $ C \mapsto Q(C) $ defines a surjective morphism $ \overline F_n \rightarrow \overline G_n $.
\end{prop}
\begin{proof}
From Lemma \ref{le:linindepVC}, we get a morphism $ W_\tau \rightarrow\Gr(n, \fr_n^1) $.  By Lemma \ref{le:VtauVz}, this morphism extends our previous map $ F_n \rightarrow\Gr(n, \fr_n^1)$. Because $ F_n $ is dense in $ \overline F_n $, it is dense in each $W_\tau \cap W_{\tau'} $.  Hence any extension of $ F_n \rightarrow\Gr(n, \fr_n^1) $ to $ W_\tau \cap W_{\tau'} $ is unique.  Thus, these morphisms $ W_\tau \rightarrow \Gr(n, \fr_n^1) $ glue to a morphism $ \overline F_n \rightarrow \Gr(n, \fr_n^1) $.  

Now the statement follows from the fact that $ \overline F_n $ is a compactification of $ F_n $ and hence its image in $ \Gr(n, \fr_n^1) $ is closed and contains $ G_n $.
\end{proof}

\begin{ex} \label{eg:QC}
The space $ \overline F_n $ contains the ``maximal flower point'', where $ \delta_{ij} = \infty $ for all $i,j$.  This point lies in  $W_\tau $ for any planar binary forest with $ n$ trees.  In this case, every internal vertex is a root $ v $ and we have $ h_v = u_i $ where $ i $ is the label on the leaf.  The corresponding subspace $ Q(C) $ is $ \spann(u_1, \dots, u_n) $. 

On the other extreme, we can consider the stratum of $ F_n$, where $ \delta_{ij} = 0 $ for all $i,j$.  This corresponds to cactus flower curves with one petal and where that petal contains no marked point.  This stratum is isomorphic to $ \overline M_{n+1} $.  Given $ C \in V^{[n]} $, let $ C' $ be its image in $ \overline M_{n+1}$.  The point $ C$ lies in $ W_\tau $ for some tree $ \tau $.   So (\ref{eq:Hv2}) implies that if $ v $ is the root of this tree, then $ h_v^0(C) = u_1 + \cdots + u_n$.  Comparing (\ref{eq:Hv1}) with the definition of $ Q(C') \subset \fs_n^1 $ from section \ref{se:afv}, shows that 
 $$ Q(C) = \C( u_1 + \dots + u_n) \oplus Q(C') $$
 \end{ex}

\begin{thm} \label{th:FnGnIso}
The morphism $ \overline F_n \rightarrow \overline G_n $ is an isomorphism.
\end{thm}

In order to prove this result, we will now construct a morphism in the reverse direction.  As in section \ref{se:afv}, we will consider projections onto the three dimensional subspaces spanned by the basis vectors.  However, unlike in section \ref{se:afv}, these projections will not necessarily produce 2-dimensional subspaces.  For example, as discussed above in Example \ref{eg:QC}, the maximal flower point gives $ \spann(u_1, \dots, u_n) $ and this projects to the 0 subspace in $ \spann(t_{ij}, t_{ik}, t_{jk}) $.  This is closely related to the fact that the functions $ \mu_{ijk} $ are not defined on the whole of $ \overline F_n $.  For this reason, we will begin by concentrating on the following projections.  For each $ ij \in p([n]) $, define $ p_{ij} : \fr_n^1 \rightarrow \BC^3 $ by $$ p_{ij}(u_i) = e_1, \ p_{ij}(u_j) = e_2,\ p_{ij}(t_{ij}) = e_3 $$ and $ p_{ij}$ maps all other basis vectors to 0.

Define an isomorphism $ K : \PP^1 \rightarrow \{ W \in \Gr(2,\C^3) : (1,1,0) \in W \} $ by 
$$
K(\tfrac{a}{b}) := \{ (x,y,z) \in \BC^3 : ax - ay - bz = 0 \} $$


\begin{lem} \label{le:pijV}  Fix $ ij \in p([n])$.
For any $ Q \in \overline G_n$, we have the following
\begin{enumerate}
\item $ (1,1,0) \in p_{ij}(Q) $
\item $ \dim p_{ij}(Q) = 2 $ 
\item If $ Q = Q(C) $ for $ C \in \overline F_n $, then $ p_{ij}(Q) = K(\nu_{ij}(C))$
\end{enumerate}
\end{lem}
\begin{proof}
First, we know that $ \sum u_k \in Q $, so $ (1,1,0) \in p_{ij}(Q) $.

Next, we know that $ p_{ij}(Q) \subset \BC^3 $ is a commutative subspace (where we identify $\BC^3 = \fr_2^1$, so  $ [e_1,e_3] = -[e_2,e_3], \ [e_1,e_2] = 0 $), which implies that $ \dim p_{ij}(Q) \le 2 $.  

Finally, we will now find another vector in $ p_{ij}(Q) $ by an explicit computation.  Since $ Q = Q(C) $ for some $ C \in \overline F_n $, we need to find a vector $ v \in Q(C)$ such that $ p_{ij}(v) $ does not lie in the span of $ u_i + u_j $.  Suppose that $ C \in W_\tau $ for some planar binary forest $ \tau $.  We have two cases depending on whether $ i, j $ lie in the same tree in $ \tau $.  

If they are in the same tree, then let $ v $ be their meet. Assume without loss of generality that $i$ lies in the left branch above $ v $.  Then using (\ref{eq:Hv1}), we have $ p_{ij}( h_v^0(C)) = (\delta_{pq}, 0, \delta_{pq} \nu_{ij} )$ where $p,q $ are the consecutive leaves meeting at $ v$.  On $ W_\tau $, $\delta_{pq} \nu_{ij} $ is non-zero (by \cite[Lemma 6.16]{iklpr}) and so $ p_{ij}( h_v^0(C)) $ does not lie in the span of $ (1,1,0)$, but does lie in $ K(\nu_{ij})$.

On the other hand, if $ i, j$ do not lie in the same tree, then let $ v $ be the root of the tree containing $ i $.  Then using (\ref{eq:Hv2}), we have $ p_{ij}(h_v^0(C)) = (1 ,0, \nu_{ij} ) $ which again does not lie in the span of $(1,1,0) $, but does lie in $ K(\nu_{ij}) $.
\end{proof}

In this way, we get a map $ \nu_{ij} : \overline G_n \rightarrow \PP^1 $ by $ Q \mapsto K^{-1}(p_{ij}(Q)) $.
\begin{lem} \label{le:deltaGn}
These maps $ \nu_{ij} $ combine together to give a map $ p :  \overline G_n \rightarrow \overline \ft_n$, such that the usual map $ \overline F_n \rightarrow \overline \ft_n $ factors as $ \overline F_n \rightarrow \overline G_n \rightarrow \overline \ft_n$.
\end{lem}
\begin{proof}
This follows immediately from Lemma \ref{le:pijV}(3).
\end{proof}

Fix a set partition $ \CS $ of $ [n] $.  We have the open subset $ U_\CS \subset \overline \ft_n $ and its preimage $ \tU_\CS \subset \overline F_n $.  Fix a triple $ ijk \in t(S_r) $ for some $ r$ and consider the projection $ q_{ijk} : \fr^1_n \rightarrow \C^3 $ defined as in (\ref{eq:qijk}), but also with $ q_{ijk}(u_l) =  0 $ for all $ l$.

\begin{lem} \label{le:qJ}
Let $ C \in \tU_\CS $ and let  $ Q = Q(C) $. Then $ q_{ijk}(Q) = J(\mu_{ijk}(C))$. 
\end{lem}
\begin{proof}
As in the proof of Lemma 6.17 from \cite{iklpr}, we can find a planar binary forest $ \tau $ such that $ C \in W_\tau $ and the partition of $ [n] $ corresponding to the leaves of $ \tau $ is exactly $ \CS $.  So $ \tau $ has trees $ \tau_1, \dots, \tau_m $ such that the leaves of $ \tau_r $ are labelled $ S_r $.  In particular, all $ i,j,k$ lie on the same tree.

Without loss of generality, assume that $ i < j < k $ in the partial order defined by $ \tau $.  Also, assume that the meet of $i,j $ lies above the meet of $ i,k$.  Let $ v_1 $ be the meet of $i,j$ and $ v_2 $ be the meet of $ i,k$ (which coincides with the meet of $ j,k$). We know that $ v_1 \ne v_2 $, since $ \tau $ is a binary forest.  Then $ i$ is in the left branch above $v_1$ and $ i,j $ are in the left branch above $ v_2 $.  Let $p_1, q_1 $ be the consecutive leaves meeting at $ v_1 $ and $ p_2, q_2 $ be the consecutive leaves meeting at $ v_2 $.  Let
\begin{gather*}
w_1 := q_{ijk}(h_{v_1}^0(C)) = (\delta_{p_1 q_1} \nu_{ij}, \delta_{p_1 q_1} \nu_{ik}, 0) \\ w_2 := q_{ijk}(h_{v_2}^0(C)) = (0, \delta_{p_2 q_2} \nu_{ik}, \delta_{p_2 q_2} \nu_{jk})
\end{gather*}
Recall that $$ J(\mu_{ijk}) = \bigl\{(x,y,z) : x - \mu_{ijk} y + (\mu_{ijk} - 1)z = 0 \bigr\}$$
Since $ \mu_{ijk} \nu_{ik} = \nu_{ij}$, we see that $ w_1 \in J(\mu_{ijk})$.  Also, $$ -\mu_{ijk} \nu_{ik} + (\mu_{ijk} -1) \nu_{jk} = -\nu_{ij} - \mu_{jik}\nu_{jk} = -\nu_{ij} - \nu_{ji} = 0 $$
and so $w_2 \in  J(\mu_{ijk})$.  On the other hand,  $ \delta_{p_1 q_1} \nu_{ij} \ne 0 $ and $ \delta_{p_2 q_2} \nu_{ik} \ne 0 $ by \cite[Lemma 6.16]{iklpr}, so $ w_1, w_2 $ are linearly independent.  

Thus, we have found two linearly independent elements of $ q_{ijk}(Q) $ which lie in $ J(\mu_{ijk}(C))$, and so $ J(\mu_{ijk}(C)) \subseteq q_{ijk}(Q)$.  On the other hand, since $ Q $ is commutative, so is $ q_{ijk}(Q) $ and so $ q_{ijk}(Q) $ is at most 2-dimensional.
\end{proof}

\begin{proof}[Proof of Theorem \ref{th:FnGnIso}]
The surjective map $ \overline F_n \rightarrow \overline G_n $ restricts to a surjective map $ \tU_\CS \rightarrow p^{-1}(U_\CS) \subset \overline G_n $ by Lemma \ref{le:deltaGn}. By construction, we have an embedding $ \tU_\CS \rightarrow (\PP^1)^{p([n])} \times \prod_r (\PP^1)^{t(S_r)} $. By Lemma \ref{le:qJ}, we see that there is a morphism $ p^{-1}(U_\CS) \rightarrow (\PP^1)^{p([n])} \times \prod_r (\PP^1)^{t(S_r)}$ such that the composition $ \tU_\CS \rightarrow p^{-1}(U_\CS) \rightarrow (\PP^1)^{p([n])} \times \prod_r (\PP^1)^{t(S_r)} $ is this embedding.  Thus, we conclude that $ \tU_\CS \rightarrow p^{-1}(U_\CS) $ is an isomorphism and thus $ \overline F_n \rightarrow \overline G_n $ is an isomorphism. 
    
\end{proof}

\subsection{Degeneration} \label{se:CFnfrrn}
There is a one-parameter family of Lie algebras degenerating the Drinfeld-Kohno Lie algebra $ \fs_{n+1} $ into $\fr_n$.  We define $ \frr_n $ to be the $ \C[\hbar] $-Lie algebra with the same set of generators as $ \fr_n $ but relations
\begin{gather*}
[t_{ij}, t_{kl} ] = 0 \text{ for $ i,j,k,l $ distinct } \quad [t_{ij}, t_{ik}] = -[t_{ij} , t_{jk} ]  \text{ for $ i,j,k $ distinct } \\
 [t_{ij}, u_k] = 0 \text{ for $ i,j,k$ distinct } \quad [u_i, t_{ij} ] = -[u_j, t_{ij}]  \quad  [u_i, u_j] = \hbar [u_i, t_{ij}] \text{ for $i,j$ distinct }
\end{gather*}
For $ \varepsilon \ne 0 $, we have an isomorphism $ \frr_n(\varepsilon) \cong \fs_{n+1} $ given $ t_{ij} \mapsto t_{ij} $ and $ u_i \mapsto -\varepsilon t_{i \, 0} $.  On the other hand, there is an obvious isomorphism $ \frr_n(0) \cong \fr_n$.

 We consider a grading on $ \frr_n $ where all generators have degree 1 (and $ \hbar $ has degree 0) and as before we study the first graded piece $ \frr_n^1 $.  We will regard $ \frr_n^1 $ as a $ \C[\hbar] $-module, and thus as a vector bundle of $ \BA^1 $.  We consider the relative Grassmannian $ \Gr_{\BA^1}(n, \frr_n)$, a proper scheme over $ \BA^1 $.  A $ \C $-point of $ \Gr_{\BA^1}(n, \frr_n) $ is a pair $ (V, \varepsilon) $ where $ \varepsilon \in \C $ and $ V \in \Gr(n, \frr_n(\varepsilon)) $.  

Define $ h^\hbar_i \in \CO(\CF_n) \otimes_{\C[\hbar]} \frr_n $ for $ i =1, \dots, n $ by
$$ h^\hbar_i := u_i - \sum_{j \ne i} \nu_{ji} t_{ij} $$

For any $ (\uz; \varepsilon) \in \CF_n $, we can consider the evaluation map $ \CO(\CF_n) \rightarrow \C $ defined by $ (\uz; \varepsilon) $ and extend to a linear map $  \CO(\CF_n) \otimes_{\C[\hbar]} \frr_n \rightarrow \frr_n(\varepsilon) $.  We write $ h^\varepsilon_i(\uz)$ for the image of $ h^\hbar_i $ under this latter map.
\begin{lem} \label{le:Hvareps}
Suppose that $ \varepsilon \ne 0 $.  Then for $(\uz; \varepsilon) \in \CF_n $, $ h^\varepsilon_i(\uz) = (1-\varepsilon z_i) h_i(\varepsilon^{-1}, z_1, \dots, z_n) $ under the isomorphism $ \frr_n(\varepsilon) \cong \fs_{n+1} $.
\end{lem} 

\begin{proof}
We have
\begin{equation*}
    \begin{gathered}
        h^\varepsilon_i(\uz) = u_i - \sum_{j\ne i} \frac{1 - \varepsilon z_i}{z_j - z_i} t_{ij} = (1- \varepsilon z_i) \Bigl( \frac{-\varepsilon}{1- \varepsilon z_i } t_{i \, 0} - \sum_{j\ne i} \frac{1}{z_j -z_i} t_{ij} \Bigr) \\ = (1- \varepsilon z_i) \Bigl(\frac{1}{z_i - \varepsilon^{-1}} t_{i \, 0} + \sum_{j \ne i} \frac{1}{z_i - z_j} t_{ij}\Bigr)
    \end{gathered}
\end{equation*} 
and so the result follows.
\end{proof}

On the other hand, it is clear that if we set $ \hbar = 0 $, we recover $ h_i^0(\uz) $ and so we deduce the following result.

\begin{corol} \label{co:hepsind}
For any $ ( \uz; \varepsilon) \in \CF_n $, $ \{ h^\varepsilon_i(\uz) : i =1, \dots, n \} $ are linearly independent in $ \frr^1_n(\varepsilon) $. 
\end{corol}

For $ (\uz, \varepsilon) \in \CF_n$, we let $ Q^\varepsilon(\uz) := \spann(H^\varepsilon_i(\uz)) \subset \frr^1_n(\varepsilon)$.  For $ \varepsilon \ne  0$, under the isomorphism $\frr_n(\varepsilon) \cong \fs_{n+1}$, by Lemma \ref{le:Hvareps} we have 
$$ Q^\varepsilon(\uz) = Q(\varepsilon^{-1}, z_1, \dots, z_n) = Q(0, 1-\varepsilon z_1, \dots, 1 - \varepsilon z_n) $$
and it is immediate that $ Q^0(\uz) = Q(\uz) \subset \fr^1_n$.

Let $ \CQ $ be the $\CO(\CF_n) $ submodule of $ \CO(\CF_n) \otimes_{\C[\hbar]} \frr^1_n $ generated by $ \{ H^\varepsilon_i(\uz) : i =1, \dots, n \} $.  Corollary \ref{co:hepsind} shows that $ \CQ $ is locally free and its quotient is also locally free.  Thus, it defines a map $ \CF_n \rightarrow \Gr_{\BA^1}(n, \frr^1_n)$, extending our previous maps as follows.

 \begin{equation} \label{eq:newdiagram}
	\begin{tikzcd}
		\CF_n(\varepsilon) \cong M_{n+2} \ar[r] \ar[d] &  \CF_n \ar[d]  & \ar[l] \ar[d]  F_n \\
		Gr(n, \fs^1_{n+1}) \ar[r] \ar[d] & \Gr_{\BA^1}(n, \frr^1_n) \ar[d] & \ar[l] \Gr(n, \fr^1_n) \ar[d] \\
		\{\varepsilon \ne 0 \} \ar[r] & \BA^1 & \ar[l] \{0 \}
	\end{tikzcd}
\end{equation}

Now as before, we let $ \CG_n \subset \Gr_{\BA^1}(n, \frr^1_n)  $ be the image of $ \CF_n $ and we consider the closure $ \overline \CG_n$.  

We will now define a morphism $ \overline \CF_n \rightarrow \overline \CG_n $ over $ \BA^1$.  For every binary planar forest $ \tau $ and any internal vertex $ v $ of $ \tau $, we define $ h_v^\hbar \in \CO(\CW_\tau) \otimes_{\C[\hbar]} \frr^1_n $ as follows
\begin{equation*}
h_v^\hbar = \left\{ \begin{aligned} &\sum_{i \, \text{left}} \delta_{pq} u_i -  \delta_{pq} \hbar \sum_{i \ne j \, \text{left}} t_{ij} - \sum_{i \, \text{left}\,  j \, \text{other}} \delta_{pq} \nu_{ji} t_{ij} &\text{ if $ v $ is not a root }\\
 &\sum_{i \,\text{tree}} u_i - \hbar \sum_{i\ne j \, \text{tree}} t_{ij} - \sum_{i \, \text{tree} \, j \, \text{other}}  \nu_{ji} t_{ij} &\text{ if $v $ is a root}
\end{aligned} \right.
\end{equation*}
In the first equation, ``$ i \, \text{left} $'' means that $ i $ is in the left branch above $ v$.  In the second equation, ``$ i \, \text{tree} $'' means that $ i $ is in the tree above $ v $.  Finally ``$ j\, \text{other} $'' means that $ j$ lies elsewhere in the forest.

The analog of Lemma \ref{le:VtauVz} holds with the same proof; note that $ \hbar $ appears in the above formulas for $ h_v^\hbar$ since $ \nu_{ij} + \nu_{ji} = \hbar$.

By the same argument as in Lemma \ref{le:linindepVC}, these generate a rank $n$ subvector bundle of $\CO(\CW_\tau) \otimes_{\C[\hbar]} \frr^1_n $ and thus we get a map $ \CW_\tau \rightarrow \Gr_{\BA^1}(n, \frr^1_n) $.  By the same argument as in the proof of Proposition \ref{pr:FnGnsurject}, these glue to a surjective morphism $ \overline \CF_n \rightarrow \overline \CG_n $ (since $ \overline \CF_n $ is proper over $\BA^1$).

\begin{thm} \label{th:CFCG}
The morphism $ \overline \CF_n \rightarrow \overline \CG_n $ is an isomorphism.
\end{thm} 

\begin{proof}
We proceed similar to the proof of Theorem \ref{th:FnGnIso}.

First Lemmas \ref{le:pijV}, and \ref{le:deltaGn} generalize to the degeneration in a straightforward way.

For Lemma \ref{le:qJ}, we again define $w_1 = q_{ijk}(h^\hbar_{v_1}(C)), w_2 = q_{ijk}(h^\hbar_{v_2}(C)) $ as in the proof, and we find that
$$
w_1 = -(\nu_{ji}, \nu_{ki}, 0) \quad w_2 = -(\hbar, \nu_{ki}, \nu_{kj})
$$
We observe that 
$$
\nu_{ji} - \mu_{ijk} \nu_{ki} = 0 
$$
so $ w_1 \in J(\mu_{ijk}) $.  On the other hand,
$$
\hbar - \mu_{ijk}\nu_{ki} + (\mu_{ijk} - 1) \nu_{kj} = \hbar - \nu_{ji} - \mu_{jik} \nu_{kj} = \hbar - \nu_{ji} - \nu_{ij} = 0 
$$
so also $ w_2 \in J(\mu_{ijk})$.  Thus Lemma \ref{le:qJ} generalizes as well.

Hence we can follow the same argument as the proof of Theorem \ref{th:FnGnIso} and deduce the result.


\end{proof}

\section{Homogeneous Gaudin model.}
\label{homgaudin}
In this section, we recall known results concerning on homogeneous Gaudin subalgebras.

\subsection{Universal Gaudin subalgebras.} \label{se:universalGaudin}

Consider the first congruence subalgebra $t^{-1}\fg[t^{-1}]$ in the current algebra $\fg[t^{-1}]$. For any $x\in\fg,\ k\in\mathbb{Z}_{\ge1}$, let $x[-k]:=xt^{-k}\in t^{-1}\fg[t^{-1}]$. Define an algebra homomorphism 
$$i_{-1}: S(\fg) \to S(t^{-1} \fg[t^{-1}]), \quad x \mapsto x[-1] \text{ for } x \in \fg.$$ 
Let $\partial_t, t \partial_t$ be the derivations of $t^{-1}\fg[t^{-1}]$ given by
$$\partial_t (x[-k]) = -k x[-k-1], \quad t\partial_t (x [-k]) = -k x[-k]$$
which can be extended to derivations $\partial_t, t \partial_t$ of $S(t^{-1} \fg[t^{-1}])$ and $U(t^{-1} \fg[t^{-1}])$.
The derivation $ t \partial_t$ integrates to the \emph{loop rotation} action of $ \Cx $ on $ U(t^{-1} \fg[t^{-1}])$ where $ x[-k] $ has weight $ -k$.

Denote by $A_{\fg}$ the subalgebra of $S(t^{-1} \fg[t^{-1}])$ generated by all $\partial_t^k i_{-1}(\Phi_l), k \in \bz_{\geq 0}, l = 1, \dots, r$. It is well-known that this subalgebra is free and a maximal Poisson-commutative subalgebra of $S(t^{-1} \fg[t^{-1}])^{\fg}$, see \cite{r3}. Equivalently, the subalgebra $A_{\fg}\subset S(t^{-1} \fg[t^{-1}])$ can be described in the following way: consider the linear map $\delta:\fg\to t^{-1} \fg[t^{-1}][[u]]$ defined by $$ \delta(x)=\sum\limits_{k=0}^\infty x[-k-1]u^k.
$$
This extends to a $\fg$-equivariant homomorphism $\delta:S\fg\to S(t^{-1} \fg[t^{-1}])[[u]]$.

\begin{prop}\label{pr:ClassicalGenerators}
    The subalgebra $A_{\fg}\subset S(t^{-1} \fg[t^{-1}])$ is generated by all the Fourier coefficients of $\delta(\Phi_l), l = 1, \dots, r$.
\end{prop}
\begin{proof}
    Consider the automorphism $\exp(-u\partial_t)$ of $S(t^{-1} \fg[t^{-1}])[[u]]$; it maps $x[-1]$ to $\sum\limits_{k=0}^\infty x[-k-1]u^k=\delta(x)$ for $ x \in \fg$.   Since these elements generate $ S(\fg)$, we conclude that $ \delta(y) = \exp(-u\partial_t) i_{-1}(y)$ for any $ y \in S(\fg)$.
    
    So we have $$\delta(\Phi_l)=\exp(-u\partial_t)i_{-1}(\Phi_l) = \sum\limits_{k=0}^\infty \frac{(-1)^k}{k!}\partial_t^k i_{-1}(\Phi_l)u^k.$$
\end{proof}

The following result is fundamental for us.  It is proven using the centre of the affine Lie algebra at the critical level.

\begin{thm}[\cite{r3, fr}]\label{th:A-invariance-der}
\begin{enumerate}
    \item There exists a unique lifting $\ma_{\fg} \subset U(t^{-1} \fg[t^{-1}])^{\fg}$ of the subalgebra $A_{\fg} \subset S(t^{-1} \fg[t^{-1}])^{\fg}$.
    \item There exist elements $S_l \in U(t^{-1} \fg[t^{-1}])$, for $ l = 1, \ldots, r$, such that \begin{enumerate}
        \item for all $ l $, $ (t \partial_t)(S_l) = -d_l S_l $,
        \item the elements $\partial_t^k S_l, k \in \bz_{ \geq0}$ pairwise commute, 
        \item these elements generate $ \ma_\fg $.
    \end{enumerate}
    \item We have $ S_l \in U(t^{-1}\fg[t^{-1}])^{(d_l)} $ (the $d_l$ filtered part with respect to the PBW filtration) and $ S_l$ lifts $ i_{-1}(\Phi_l) \in S(t^{-1}\fg[t^{-1}])$.
    \item The subalgebra $\ma_{\fg}$ is the centralizer of $S_1$ and is a maximal commutative subalgebra of $U(t^{-1} \fg[t^{-1}])^\fg$. 
\end{enumerate}
\end{thm}

In particular, we will take $$S_1 = \sum_{a=1}^{\dim \fg}
x^a [-1] x_a[-1].$$

There is another version of the universal Gaudin algebra as a subalgebra of $U(\fg[t^{-1}])$, whose definition we review following \cite{ir}. The construction is parallel to the construction of $\ma$. Instead of $ i_{-1} $, we simply embed $ \Phi_l \in S(\fg)$ into $ S (\fg[t]) $ in the obvious way, and we use the derivation $ D : S(\fg[t]) \rightarrow S(\fg[t]) $ given by $$D (x[m]) = (m+1) x[m+1] \text{ for $ x \in \fg[t]$.}$$ 
Then we let $\tilde A_{\fg}$ be the algebra generated by all $ D^k \Phi_l, k \in \bz_{\ge 0}, l = 1, \dots, r$. It has a similar description to that of $\ma_\fg$ from Proposition~\ref{pr:ClassicalGenerators}. Consider the linear map $\tilde{\delta}:\fg\to\fg[t][[u]]$ defined by $$ \tilde{\delta}(x)=\sum\limits_{k=0}^\infty x[k]u^k.
$$
This extends to a $\fg$-equivariant homomorphism $\tilde \delta:S\fg\to S(\fg[t])[[u]]$.  The following result is parallel to Proposition \ref{pr:ClassicalGenerators} and the proof is identical.

\begin{prop}\label{pr:ClassicalGeneratorsTilde}
    The subalgebra $\tilde{A}_{\fg}\subset S(\fg[t])$ is generated by all the Fourier coefficients of $\tilde{\delta}(\Phi_l)$, for $ l = 1, \dots, r$.
\end{prop}

By \cite{ir}, the subalgebra $\tilde A_{\fg} $ is the Poisson centralizer in $S(\fg[t])^{\fg}$ of the element $ \Gamma := \sum_{a} x^{a}[0] x_{a}[1]=\sum_{a} x^{a}[1] x_{a}[0] $.

Additionally, it is well-known that there exists a unique lifting of $\tilde A_{\fg}$ to $U(\fg[t])^{\fg}$, which we denote by $\tilde \ma_{\fg}$. One can prove existence using the center of the critical level \cite{ffre}, \cite{fr}, or using Bethe subalgebras in Yangians, \cite{ir}. Below we will consider $\tilde \ma_{\fg}$ as a subalgebra of $U(\fg[t^{-1}])$ by identifying $ U(\fg[t]) $ with $ U(\fg[t^{-1}]) $ via the change of variables $t \to t^{-1}$.

\subsection{Some properties of universal Gaudin subalgebras.}




Let $\sigma$ be the antilinear automorphism of $U(t^{-1} \fg[t^{-1}])$ induced by the Cartan involution on $\fg$.
\begin{prop}\label{pr:A-invariance-sigma} 
    $\mathcal{A}_{\fg}$ is stable under $\sigma$.
\end{prop}

\begin{proof}
    Since $S_1\in U(t^{-1}\fg[t^{-1}])$ is $\sigma$-invariant, so is its centralizer.
\end{proof}

Consider the quantum Hamiltonian reduction setup from Section \ref{se:QHR} with the decomposition $ \fg[t^{-1}] = \fg \oplus t^{-1}\fg[t^{-1}]$.  Then by Proposition \ref{pr:HamReductiontoU2}, we obtain an isomorphism $$\left(U(\fg[t^{-1}])/U(\fg[t^{-1}])\fg\right)^\fg \xrightarrow{\cong} U(t^{-1}\fg[t^{-1}])^\fg $$

Now, consider the composite homomorphism $$\iota: U(\fg[t^{-1}])^\fg\to \left(U(\fg[t^{-1}])/U(\fg[t^{-1}])\fg\right)^\fg \rightarrow U(t^{-1}\fg[t^{-1}])^\fg.$$

\begin{prop}\label{pr:iota(A)} We have    $\iota(\widetilde{\ma_{\fg}})=\ma_{\fg}$.
\end{prop}

\begin{proof}
Consider $ \gr \iota : S(\fg[t^{-1}])^\fg \rightarrow S(t^{-1}\fg[t^{-1}])^\fg $.  By the last statement of Proposition \ref{pr:HamReductionToU}, this comes from the projection $ \fg[t^{-1}] \rightarrow t^{-1}\fg[t^{-1}]$.  As this projection takes $ \tilde \delta $ to $ \delta$, from Propositions~\ref{pr:ClassicalGenerators}~and~\ref{pr:ClassicalGeneratorsTilde} we conclude that $ \gr \iota(\widetilde A_\fg) = A_\fg$. On the other hand, the lifting of $A_{\fg}$ to $U(t^{-1}\fg[t^{-1}])^\fg$ is unique. 
\end{proof}

\subsection{Classical homogeneous Gaudin algebra}

Let $z \in \bc^\times$.  There is a Lie algebra map $ t^{-1} \fg[t^{-1}] \rightarrow \fg$ given by $ x[-m] \mapsto z^{-m} x$.  This extends to algebra homomorphisms $$\ev_{z}: S(t^{-1}\fg[t^{-1}]) \to S\fg, \quad \ev_z : U (t^{-1}\fg[t^{-1}]) \rightarrow U \fg. $$ 

\begin{rem} \label{re:looprotate}
    Note that for $ s \in \Cx, z \in \Cx, a \in U (t^{-1}\fg[t^{-1}])$ we have $ \ev_{sz}(a) = \ev_z(sa) $ where on the right hand side we use the loop rotation action of $ \Cx$ on $ U(t^{-1}\fg[t]) $.  Thus if $ a $ has weight $ -m $ for this action, then $ \ev_{sz}(a) = s^{-m} \ev_z(a)$.
\end{rem}

\begin{rem} \label{re:SlPhil}
    Recall that $ \widetilde \Phi_l \in Z(U \fg) $ denotes a central generator lifting $ \Phi_l \in S(\fg)^\fg$.  From the properties of $ S_l$, we can see that $ \ev_1(S_l) $ is such a central generator, so it is natural to fix this choice of $ \widetilde \Phi_l := \ev_1(S_l)$.  By the previous remark, we have $ \ev_z(S_l) = z^{-d_l} \widetilde \Phi_l$.
\end{rem}

Now, let $ z_1, \dots, z_n \in \bc^{\times}$ with $z_i \not = z_j$ for $i \not = j$.
Let $\ev_{z_1, \ldots, z_n} = \ev_{z_1} \otimes \ldots \otimes \ev_{z_n}: S (t^{-1} \fg[t^{-1}]) ^{\otimes n} \rightarrow S(\fg)$.  

We define $ A(z_1, \dots, z_n)$ to be the image of $ A $ under $$ ev_{z_1, \dots, z_n} \circ \Delta :S(t^{-1}\fg[t^{-1}]) \rightarrow S(\fg^{\oplus n}) $$

To understand this image better, we note that the linear map $ ev_{z_1, \dots, z_n} \circ \Delta : t^{-1}\fg[t^{-1}] \rightarrow \fg^{\oplus n}$ can be dualized to a linear map $ (ev_{z_1, \dots, z_n} \circ \Delta)^* : \fg^{\oplus n} \rightarrow \fg[[t]]$ where we use the bilinear form to identify $ \fg $ and $ \fg^*$ and extend in the usual way to a pairing between $t^{-1} \fg[t^{-1}]$ and $ \fg[[t]]$.  This map is given by
$$(ev_{z_1, \dots, z_n} \circ \Delta)^* = \sum_{i=1}^n  \frac{1}{z_i-t} x^{(i)}$$
Here and below, we write $ x^{(i)} : \fg^{\oplus n} \rightarrow \fg $ for the $i$th coordinate function.

Recall from section \ref{se:universalGaudin}, for each $ l = 1, \dots, r$ we defined $ \delta(\Phi_l) \in S(t^{-1}\fg[t^{-1}])[[u]] = \CO(\fg[[t]])[[u]]$.  So now, we define $ \Phi_l(u;z_1, \dots, z_n) \in \CO(\fg^{\oplus n})[[u]] $ by pulling back $ \delta(\Phi_l)$ via the map $ (ev_{z_1, \dots, z_n} \circ \Delta)^*$.  This $ \Phi_l$ is actually a rational function of $ u $ and we have 
$$
\Phi_l(u;z_1, \dots, z_n) = \Phi_l \left(\sum_{i=1}^n \frac{1}{z_i-u} x^{(i)} \right)
$$
where we regard $ \Phi_l \in \CO(\fg)^\fg$.  

Note that $ \Phi_l(u;z_1, \dots, z_n)$ has a pole of order $ d_l$ at each point $ z_1, \dots, z_n$ and we consider the Laurent expansion at these points. 
$$ \Phi_l(u;z_1, \dots, z_n) = \sum_{m=1}^{d_l}  \Phi_{l,i}^m (u - z_i)^{-m} + \CO(\fg^{\oplus n})[[u - z_i]] $$
where $  \Phi_{l,i}^m \in \CO(\fg^{\oplus n})$.  The leading coefficients of this Laurent expansion are easy to understand; we have $ \Phi_{l,i}^{d_l} = \Phi_l(x^{(i)})$.

\begin{prop}\cite[Proposition~1]{cfr} \label{pr:GenCommGaudin}
$A(z_1, \dots, z_n)$ is a polynomial ring with generators $$ \{  \Phi_{l,i}^m : m = 1\dots, d_l, i = 1, \dots, n-1, l = 1, \dots, r  \} \sqcup \{ \Phi_{l,n}^{d_l} : l = 1, \dots, r\}.$$
\end{prop}

\begin{rem}
Proposition~1 of \cite{cfr} is the quantum version of the above Proposition, namely, it states the same for the Gaudin subalgebras $\ma(z_1,\ldots,z_n)\subset U(\fg)^{\otimes n}$, but the proof is the same. 
\end{rem}

\begin{rem}
    At the beginning of this section, we assumed that $ z_i \ne 0 $ for all $i$.  This is in fact unnecessary and $ A(z_1, \dots, z_n) $ is well-defined for all $ z_1, \dots, z_n$, pairwise distinct, see below Proposition \ref{cfr2}.
\end{rem}

\subsection{Definition of the Gaudin subalgebra}

As above, let $ z_1, \dots, z_n \in \bc^{\times}$ with $z_i \not = z_j$ for $i \not = j$.

Let $\ev_{z_1, \ldots, z_n} = \ev_{z_1} \otimes \ldots \otimes \ev_{z_n}: U (t^{-1} \fg[t^{-1}]) ^{\otimes n} \rightarrow U\fg$.  Recall the coproduct  $\Delta:  U(t^{-1}\fg[t^{-1}]) \to U(t^{-1}\fg[t^{-1}])^{\otimes n}$ from Section \ref{se:Delta}.

\begin{defn}
The homogeneous Gaudin subalgebra $\ma(z_1, \ldots, z_n)$ is the image of $\ma_{\fg}$ in $U\fg^{\otimes n}$ under map $\ev_{z_1, \ldots, z_n} \circ \Delta$.
\end{defn}

In fact, $ \ma(z_1, \dots, z_n) $ lies in $ (U \fg^{\otimes n})^\fg$, the invariants with respect to the diagonal $ \fg$ action.

The following proposition shows that it is possible to use the second version of the universal Gaudin subalgebra as well.  By a slight abuse of notation, we write $ ev_z : U \fg[t] \rightarrow U\fg$ for the obvious extension of the above $ ev_z$.
\begin{prop}\label{pr:AA}
    The subalgebra $\mathcal{A}(z_1,\ldots,z_n)$ is the image of $\tilde{\mathcal{A}}_{\fg}$ under the evaluation map $ev_{z_1^{-1},\ldots,z_n^{-1}} \circ \Delta$.
\end{prop}

\begin{proof}
    This is a particular case of \cite[Proposition 4.10]{kmr} with $\chi=0$.
\end{proof}
It is clear that $\ma(z_1, \ldots, z_n)$ is commutative since it is the image of a commutative algebra.
Using the elements $S_l$, it is possible to write the subalgebra $\ma(z_1, \ldots, z_n)$ more explicitly. Let $u$ be a variable and $z_1, \ldots, z_n \in \bc$ such that $z_i \not = z_j$ for $i \not = j$. Consider the homomorphism of algebras
\begin{align*}\ev_{u-z_1, \ldots, u-z_n} \circ \Delta: U(t^{-1} \fg[t^{-1}]) &\to U\fg^{\otimes n}(u) \\
x[-m] &\mapsto \sum_{i=1}^n  (u-z_i)^{-m} x^{(i)}  \quad x \in \fg
\end{align*}
and consider the rational $U\fg^{\otimes n}$-valued function of $ u $
$$S_l(u; z_1, \ldots, z_n) = \ev_{u-z_1, \ldots, u-z_n} \circ \Delta(S_l) \in U \fg^{\otimes n}((u))$$

Note that this function is well-defined except at the points  $z_1, \ldots, z_n$ where it has poles of order $ d_l$.  We consider the Laurent expansion of $ S_l$.

$$ S_l(u;z_1, \dots, z_n) = \sum_{m=1}^{d_l} S_{l,i}^m (u - z_i)^{-m} + U \fg^{\otimes n}[[z_i - u]] $$
where $ S_{l,i}^m \in U\fg^{\otimes n}$.   As $ \ev_{z_1,\ldots, z_n} $ preserves the PBW filtration, we see  that $ S_{l,i}^m \in (U \fg^{\otimes n})^{(d_l)}$.

By Remark \ref{re:SlPhil}, we have that $ S_{l,i}^{d_l} = \widetilde \Phi_l^{(i)} $ for all $ l, i$.

\begin{prop}{\cite{cfr}}
\label{cfr2}
The homogeneous Gaudin subalgebra $\ma(z_1, \ldots, z_n)$ is the subalgebra generated by all $S_l(u; z_1, \ldots, z_n)$, for all $u \not = z_i$.
\end{prop}
We will be particularly interested in the quadratic generators

$$H_i(\uz) := S_{1,i}^1 = \sum_{j \not  =i}  \frac{\Omega^{(ij)}}{z_i - z_j}.
$$
From the theorem below we know that there are $2n-1$ algebraically independent generators of degree $2$ in $\ma(z_1, \ldots, z_n)$.  Note that $\sum H_i = 0$  so we can consider $H_1, \ldots, H_{n-1}$ and $n$ quadratic Casimirs, $ \Omega^{(1)}, \dots, \Omega^{(n)}$ as a complete list of these generators.

The following properties of homogeneous Gaudin subalgebras are well-known (for the maximality property -- see \cite{hkrw} or \cite{y}).

\begin{thm}{\cite{cfr,r2}} \label{th:homGaudin}
Let $ z_1, \dots, z_n $ be distinct.
\begin{enumerate}
\label{m1}
    \item The subalgebra $\ma(z_1, \ldots, z_n)$ is invariant under dilations and additive translations, i.e. for any $c \in \bc$ 
    $$\ma(z_1, \ldots, z_n) = \ma(z_1 +c, \ldots, z_n +c)$$ and for any $d \in \bc^{\times}$ 
    $$\ma(z_1, \ldots, z_n) = \ma(d  z_1, \ldots, d  z_n)$$
    \item $\ma(z_1, \ldots, z_n)$ is a maximal commutative subalgebra of $(U\fg^{\otimes n})^{\fg}$.

    \item The subalgebra $\ma(z_1, \ldots, z_n)$ is polynomial ring with 
     $(n-1)(p+r) + r$
   generators.  One possible choice of free generators is $$\{ S_{l,i}^m : m = 1\dots, d_l, i = 1, \dots, n-1, l = 1, \dots, r  \} \sqcup \{ S_{l,n}^{d_l} : l = 1, \dots, r\}.$$
   \item The Hilbert-Poincar\'e series of $\ma(z_1, \ldots, z_n)$ doesn't depend on $z_1, \ldots, z_n$ and equals
    $$\prod_{l=1}^r \frac{1}{(1-t^{d_l})^{(n-1)d_l+1}}.$$
    
\end{enumerate}
\end{thm}

Here the \textit{Hilbert-Poincar\'e series} of an $ \mathbb N$-filtered vector space $ V$ is defined to be the Hilbert-Poincar\'e series of its associated graded, so $ \sum_{k=0}^\infty \dim V^{(k)}/ V^{(k-1)} t^k$.  We use the filtration on $ \ma(z_1, \ldots, z_n)$ coming from intersecting with the PBW filtration on $ U \fg^{\otimes n}$. 

It turns out that the homogeneous Gaudin model has ``hidden'' invariance under the action of entire $PGL_2(\bc)$. But in order to see it is necessary to consider the image of the Gaudin subalgebra in $\left(U\fg^{\otimes(n+1)}/U\fg^{\otimes(n+1)}\Delta(\fg)\right)^\fg$. 

Consider our quantum Hamiltonian reduction setup from Section \ref{se:QHR} with $ \fg^{\oplus n+1} = \Delta(\fg) \oplus (0 \oplus \fg^{\oplus n})$ and note that $ 0 \oplus \fg^{\oplus n} $ is an ideal.  Then by Proposition \ref{pr:HamReductiontoU2}, we obtain an isomorphism $\left(U\fg^{\otimes n+1}/U\fg^{\otimes n +1} \Delta(\fg)\right)^\fg \rightarrow (U\fg^{\otimes n})^\fg$.

Let us define the composition homomorphism
$$\iota_0: (U\fg^{\otimes n+1})^\fg \to \left(U\fg^{\otimes n+1}/U\fg^{\otimes n +1} \Delta(\fg)\right)^\fg \rightarrow (U\fg^{\otimes n})^\fg.$$

\begin{thm}\label{projective-invariance}
Let $ z_0, \dots, z_n \in \C$ be distinct.
    \begin{enumerate}
        \item $\iota_0(\mathcal{A}(z_0,z_1,\ldots,z_n))=\mathcal{A}((z_1-z_0)^{-1},\ldots,(z_n-z_0)^{-1}).$
        \item The image of $\mathcal{A}(z_0,z_1,\ldots,z_n)$ in $ (U \fg^{\otimes n})^\fg $ is invariant under simultaneous projective transformations of the parameters $z_i$.
        \end{enumerate}
\end{thm}

\begin{proof}
    The first assertion follows from Proposition~\ref{pr:AA} and from the similar property for the universal algebras $\mathcal{A}$ and $\tilde{\mathcal{A}}$. The second assertion follows from the invariance properties of $\mathcal{A}(z_0,z_1,\ldots,z_n)$. Namely, $\mathcal{A}(z_0,z_1,\ldots,z_n)$ is stable under simultaneous affine transformations $z_i\mapsto az_i+b$, and, on the other hand, from the first assertion we know that its image in $U\fg^{\otimes n}$ is $\mathcal{A}((z_1-z_0)^{-1},\ldots,(z_n-z_0)^{-1})$, so it is stable under simultaneous projective transformations of $z_1,\ldots,z_n$ that preserve $z_0$. Altogether, these transformations generate $PGL_2(\bc)$.
\end{proof}

\subsection{Compactification of the family of homogeneous Gaudin subalgebras.} \label{se:compactGaudin}
By Theorem \ref{th:homGaudin}, the set of homogeneous Gaudin subalgebras is naturally parameterized by the space $M_{n+1}  = (\C^n \setminus \Delta) / B$, where $ B = \C^\times \rtimes \C$ is the group of affine linear transformations. 

In particular, by considering the quadratic part we get a map $ M_{n+1} \rightarrow \Gr(2n-1, (U \fg)^{(2)})$.  This map actually factors through the map $ M_{n+1} \rightarrow \Gr(n-1, \fs^1_n) $ defined in section \ref{se:afv}.  It is well-known (see for example \cite[Prop 1.5]{Enriques}) that there is a Lie algebra morphism
\begin{equation} \label{eq:gamma}
\gamma : \fs_n \rightarrow U \fg^{\otimes n} \quad t_{ij} \mapsto \Omega^{(ij)}
\end{equation}

The following result follows from Theorem \ref{th:homGaudin}. 
\begin{lem} \label{le:gamma}
    The restriction of $ \gamma$ to $ \fs_n^1$ is injective and for any $ \uz \in M_{n+1}$, we have $ A(\uz) \cap (U \fg)^{(2)} = \gamma(Q(\uz)) \oplus \spann( \omega^{(i)} : i = 1, \dots, n)$
\end{lem}

Thus from Theorem \ref{th:AFV}, we see that the compactified parameter space for the quadratic parts of homogeneous Gaudin algebras is $ \overline M_{n+1}$.  In fact, this is the complete parameter space.

\begin{thm}{\cite{r}}
\begin{enumerate}
\item The compactified parameter space for homogeneous Gaudin subalgebras is $\overline M_{n+1}$.
\item The homogeneous Gaudin subalgebras form a flat family of subalgebras, i.e. the Hilbert-Poincar\'e series of $\ma(C)$ is the same for each $C \in \overline M_{n+1}$.
\end{enumerate}
\end{thm}

In \cite{r}, the third author gave a recursive description of the limit subalgebras as follows.  Let $ C \in \overline M_{n+1}$ and let $ C_\infty $ be the component of $ C $ containing $ z_{n+1}$.  Identify $ C_\infty$ with $ \mathbb{CP}^1$ such that $ z_{n+1} $ is identified with $ \infty$.  Let $ w_1, \dots, w_m$ be the distinguished points of $ C_\infty$, other than $ \infty$.  For each $ k = 1, \dots, m$, let $ C_k $ be the (possibly reducible) curve attached at $ w_k$ and let $ B_k = \{ i \in [n] : z_i \in C_k\}$ be the labels of its marked points.  (If $ w_k = z_i$ is a marked point, then $C_k $ is empty and $ B_k  = \{i\}$.)  Note that $ \CB = (B_1, \dots, B_m)$ forms an ordered set partition of $ \{1, \dots, n\}$.

The limit Gaudin subalgebra $ \ma(C) $ is built out of two pieces: a diagonally embedded Gaudin subalgebra coming from $ C_\infty$ and a tensor product of (possibly limit) Gaudin subalgebras coming from the curves $ C_k$.

\begin{figure} 
	\includegraphics[trim=0 20 0 0, clip,width=\textwidth]{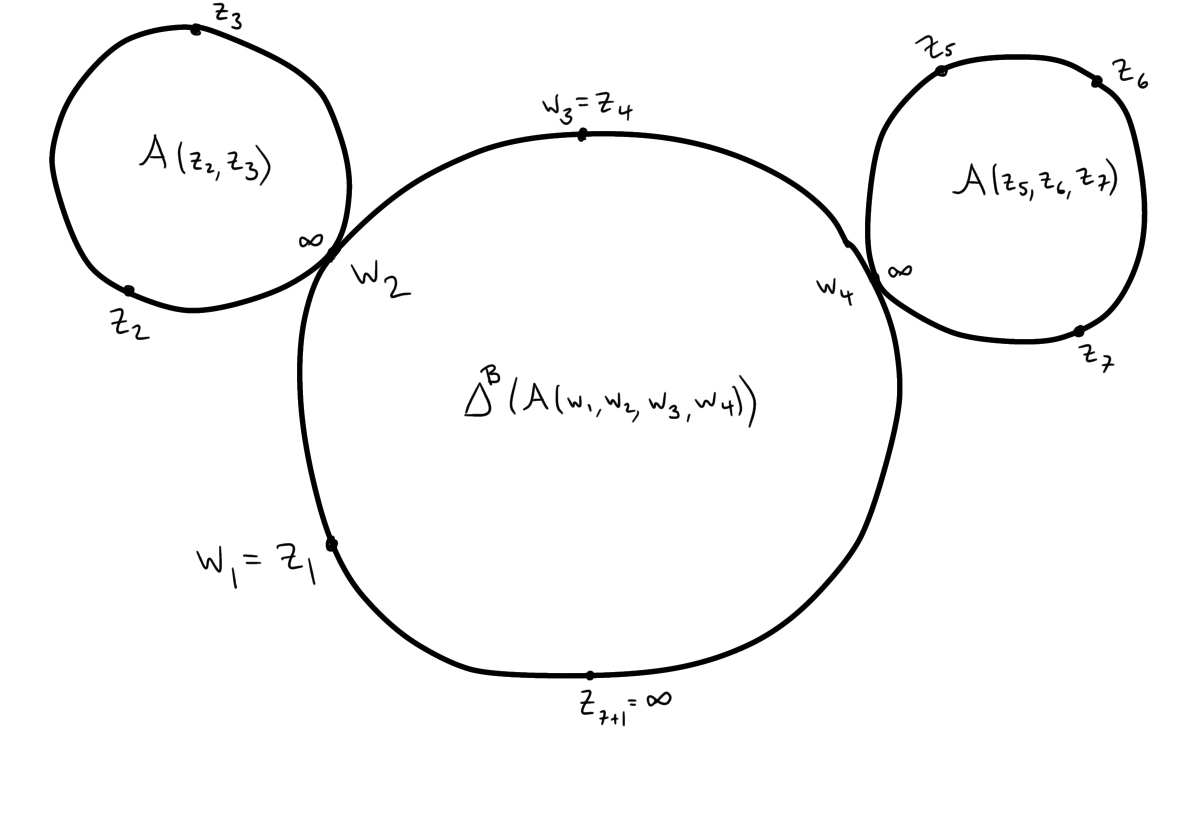}
	\caption{This is a curve $ C \in \overline M_{7+1}$.  We have $n = 7$, $m= 4 $, and $ \CB = (\{1\}, \{2,3\}, \{4\}, \{5,6,7\})$.  Each component contributes the algebra shown and the whole curve gives $ \ma(C) = \Delta^\CB(\ma(w_1,w_2,w_3,w_4))  \ma(z_2,z_3) \ma(z_5,z_6,z_7) $. }  \label{fig:Gaudin}
\end{figure}

From Section \ref{se:Delta}, we have the diagonal embedding
$$
\Delta^{\CB} : U \fg^{\otimes m} \rightarrow U \fg ^{\otimes n}$$
and we can consider $ \Delta^\CB(\ma(w_1, \dots, w_m)) \subset (U \fg^{\otimes n})^{\fg}$.

Since $ C$ is stable, $ m > 1$, and so each $ B_k $ is of size smaller than $ n$.  Thus, recursively, we have $ \ma(C_k) \subset (U \fg^{\otimes B_k})^{\fg}$, a tensor product indexed by $ B_k$.  We can collect together these index sets and identify $ \otimes_{k=1}^m U \fg^{\otimes B_k} = U \fg^{\otimes n}$ and thus we can consider $ \otimes_{k=1}^m \ma(C_k)$ as a subalgebra of $ (U \fg^{\otimes n})^\fg$.  
Let $ Z = \Delta^\CB(Z(U \fg)^{\otimes m})$.

\begin{prop} \label{pr:limitGaudin}
    With all the above notation, we have
    $$
    \ma(C) = \Delta^\CB(\ma(w_1, \dots, w_m)) \otimes_Z \bigotimes_{k=1}^m \ma(C_k)$$
\end{prop}

See Figure \ref{fig:Gaudin} for an example of this proposition.

\section{Trigonometric Gaudin model}

\subsection{Definition of trigonometric Gaudin subalgebras}
\label{se:deftrig}
Recall the triangular decomposition $ \fg = \fn_+ \oplus \fh \oplus \fn_-$.  Let $ \pi_+ : \fg \rightarrow \fn_+$ and $ \pi_- : \fg \rightarrow \fn_-$ be the corresponding projections.

Let $ \Delta(\fn_+)\subset \fg^{\oplus n+1} $ denote the diagonally embedded copy of $ \fn_+$.  We will study the quantum Hamiltonian reduction (Section \ref{se:QHR}) of $ U \fg^{\otimes n+1}$ by $ \Delta(\fn_+)$, where we choose the decomposition
\begin{equation} \label{eq:decomposeDeltafn}
    \fg^{\oplus n+1} = \Delta(\fn_+) \oplus (\fb_- \oplus \fg^{\oplus n}) 
\end{equation}

From Proposition \ref{pr:HamReductionToU}, this gives us an injective algebra homomorphism
$$\alpha:\left(U\fg^{\otimes n+1} / U\fg^{\otimes n+1} \Delta(\fn_+)\right)^{\fn_+} \to U\fb_- \otimes U\fg^{\otimes n}
$$

Note that $ \fn_- \subset \fb_- \oplus \fg^{\oplus n}$ is an ideal, and quotienting by this ideal gives a Lie algebra map $ \fb_- \oplus \fg^{\oplus n} \rightarrow \fh \oplus \fg^{\oplus n}$.  On the level of universal enveloping algebras, we get
$$\beta: U\fb_- \otimes U\fg^{\otimes n} \to U\fh \otimes U\fg^{\otimes n}.$$

There is also a natural inclusion homomorphism
$$ (U\fg^{\otimes n+1})^{\fg} \to (U\fg^{\otimes n+1})^{\fn_+},$$
which descends to a homomorphism 
$$\gamma:  (U\fg^{\otimes n+1})^{\fg} \to \left(U\fg^{\otimes n+1} / U\fg^{\otimes n+1} \Delta(\fn_+) \right)^{\fn_+}.$$

We now consider the composition of the above maps
$$\psi := \beta \circ \alpha \circ \gamma: (U\fg^{\otimes(n+1)})^{\fg} \to U\fh \otimes U\fg^{\otimes n}$$

We will now give a precise description of $ \psi $.  We begin with the following technical result.
\begin{lem}
 Let $u \in U \fn_+$.  Then $ u \otimes 1 - 1 \otimes \Delta^n(S(u)) \in U \fg^{\otimes n+1} \Delta^{n+1}(\fn_+) $ where $ S : U \fn_+ \rightarrow U \fn_+$ is the antipode map (the algebra anti-automorphism satisfying $S(x) = -x $ for $ x \in \fn_+)$.
\end{lem}

\begin{proof}
First, it suffices to consider the case $ n = 1$, since the general case follows by applying the diagonal embedding on the second factor.  

Now, we proceed by induction on the PBW filtration on $ U \fn_+$ with base case when $ u \in \C$.  Then assuming the result holds for $ u \in U \fn_+$, let $ x\in \fn_+ $, and consider the following
\begin{align*}
xu \otimes 1 - 1 \otimes S(xu) &= xu \otimes 1 + 1 \otimes x S(u) \\ &= (x \otimes 1)( u \otimes 1 - 1 \otimes S(u)) + (1 \otimes S(u))\Delta(x).
\end{align*}
As the result holds for $ u$, we conclude that $ xu \otimes 1 - 1 \otimes S(xu) $ lies in $ U \fg^{\otimes 2} \Delta(\fn_+)$ as desired.
\end{proof}

\begin{prop} \label{pr:psiab}
Let $ u \in (U \fg \otimes U \fg^{\otimes n})^\fg$.  Write $$ u = \sum_k a_{k,-} a_{k,0} a_{k,+} \otimes b_k$$ where $ a_{k,-} \in U \fn_- , a_{k,0} \in U \fh, a_{k,+} \in U \fn_+, b_k \in U \fg^{\otimes n}$.  Then 
$$
\psi(u) = \sum_k \epsilon(a_{k,-}) a_{k,0} \otimes b_k \Delta^n(S(a_{k,+}))$$ where $ \epsilon : U \fn_- \rightarrow \C $ is the usual counit.  
\end{prop}

\begin{proof}
By the previous lemma
    $$
    a_{k,-} a_{k,0} a_{k,+} \otimes b_k= a_{k,-} a_{k,0} \otimes b_k \Delta^n(S(a_{k,+})) + U \fg^{\otimes n+1} \Delta(\fn_+)
    $$
    and so the result follows.
\end{proof}

Let $\theta \in \fh \simeq \fh^{*}$. 
Composing $ \psi $ with evaluation at $ \theta$ on the first tensor factor in $ U\fh \otimes U\fg^{\otimes n}$ (using $U \fh = S \fh = \CO(\fh^*) $) gives us the map $$\psi_\theta:(U\fg^{\otimes(n+1)})^{\fg} \to (U\fg^{\otimes n})^\fh $$

The following computation follows from Proposition \ref{pr:psiab}.

\begin{lem} \label{le:thetarho}
For each $ l$, $ \psi_\theta(\tilde \Phi^{(0)}_l) = \Phi_l(\theta + \rho)$ where $ \rho $ is the half sum of the positive roots.
\end{lem}

Let $z_i \not = z_j$ for $i \not = j$ and $z_i \not = 0$ for all $i$.

\begin{defn} \label{def:trigGaudin}
The trigonometric Gaudin subalgebra $\ma_{\theta}^{trig}(z_1, \ldots, z_n)$ is the image of the homogeneous Gaudin subalgebra $\ma(0,z_1, \ldots, z_n)$ under the map $\psi_{\theta}$.
\end{defn}

Note that $\ma_{\theta}^{trig}(z_1, \ldots, z_n)$ is a commutative subalgebra (as the homomorphic image of a commutative subalgebra).
To write the quadratic elements of $\ma_{\theta}^{trig}(z_1, \ldots, z_n)$, we begin with the quadratic generators $H_i, i = 1, \ldots, n$ of the homogeneous Gaudin subalgebra $\ma(0, z_1, \ldots, z_n)$.  Then for any $ x\in \fn_+ $ such that $ x^{(0)} $  appears in $ H_i$,  we write $x^{(0)} = \Delta(x) - \sum_{j=1}^n x^{(j)}$.  After that, we bring the diagonal terms to the right and annihilate them. Finally, we evaluate $\theta \in \fh^*$ on the terms from $\fb_-^{(0)}$.

Following this procedure, we get the following quadratic Hamiltonians of the trigonometric Gaudin subalgebra $\ma_{\theta}^{trig}(z_1, \ldots, z_n)$: 
$$H_{i,\theta}^{trig} = \psi_\theta(H_i) = \sum_{j \ne i } \frac{\Omega^{(ij)}}{z_i - z_j} + \frac{\theta^{(i)}}{z_i} - \sum_j \frac{\Omega_-^{(ij)}}{z_i}$$
for $i=1, \ldots, n$.  In this formula, $ \Omega^{(ii)} $ should be interpreted as $ \omega^{(i)}$.

From the theorem below it follows that there are $2n$ algebraically independent generators of degree 2 in $\ma_{\theta}^{trig}(z_1, \ldots, z_n)$.  In fact, these quadratic trigonometric Hamiltonians together with the quadratic Casimirs $ \{H_{i,\theta}^{trig} : i = 1, \dots, n \} \cup \{ \omega^{(i)} : i =1, \dots, n \} $ form such a set of generators.

Recall the elements $ S_{l,i}^m \in \CA(0,z_1, \dots, z_n) $.  We let $ S_{l,i}^{trig, m} := \psi_\theta(S_{l,i}^m)$.  Equivalently, these are the coefficients of the principal parts of the Laurent expansions of $ \psi_\theta(S_l(u; 0, z_1, \dots, z_n))) $ at the points $ u = z_1, \dots, z_n$.  We have $ S^{trig, 1}_{1,i} =H_{i, \theta}^{trig} $.

Let $\theta \in \fh$ be arbitrary.
\begin{thm} 
\label{size} Let $ \theta \in \fh $ and let $ z_1, \dots, z_n \in \bc^\times$ be distinct.
\begin{enumerate}
\item The subalgebra $\ma_{\theta}^{trig}(z_1, \ldots, z_n)$ is invariant under dilation, i.e. for any $d \in \bc^{\times}$ 
    $$\ma_{\theta}^{trig}(z_1, \ldots, z_n) = \ma_{\theta}^{trig}(d z_1, \ldots, d  z_n)$$

    \item The subalgebra $\ma_{\theta}^{trig}(z_1, \ldots, z_n)$ is a polynomial algebra with
     $n (p+r)$ generators.
One possible choice of free generators is $$\{ S_{l,i}^{trig, m} : l = 1, \dots, r,\ i = 1, \dots, n ,\ m = 1, \dots, d_l\}.$$

\item The subalgebra $ \ma_\theta^{trig}(z_1, \dots, z_n)$ is of maximal transcendence degree for a commutative subalgebra of $(U\fg^{\otimes n})^{\fh}$.
    \item The Hilbert-Poincar\'e series of $\ma_{\theta}^{trig}(z_1, \ldots, z_n)$ is independent of $z_1, \ldots, z_n$ and $\theta $, and equals
    $$\prod_{l=1}^r \frac{1}{(1-t^{d_l})^{nd_l}}.$$
    \end{enumerate}
\end{thm}

\subsection{Classical trigonometric Gaudin algebras}

In order to prove the above theorem, it will be necessary to study the associated graded of the trigonometric Gaudin algebras.

We define $ A^{trig}_\theta(z_1, \dots, z_n) := \gr \CA^{trig}_\theta(z_1, \dots, z_n)$ to be the \textit{classical trigonometric Gaudin algebra}.  It is a Poisson commutative subalgebra of $ S(\fg)^{\otimes n} \cong \CO(\fg^{\oplus n})$.  Equivalently $$ A^{trig}_\theta(z_1, \dots, z_n) = (\gr \psi_\theta) (A(0, z_1, \dots, z_n)) $$ 
where $ \gr \psi_\theta : S(\fg^{\oplus n+1})^{\fg} \rightarrow (S \fg)^{\otimes n}$ is the associated graded of $ \psi_\theta$.  

As before, we identify $ S(\fg^{\oplus n+1})^{\fg} = \CO(\fg^{\oplus n+1})^\fg$ and $ (S \fg)^{\otimes n} = \CO(\fg^{\oplus n})$.

\begin{lem}
\label{cor1}
    The map $ \gr \psi_\theta : \CO(\fg^{\oplus n+1})^\fg \rightarrow \CO(\fg^{\oplus n}) $ is dual to the map of schemes given by
$$ \fg^{\oplus n} \rightarrow \fg^{\oplus n+1} \sslash G \quad (x_1, \dots, x_n) \mapsto (- \sum \pi_-(x_i), x_1, \dots, x_n) .$$

In particular, it is independent of $ \theta$ and hence $ A^{trig}_\theta(z_1, \dots, z_n) $ is also independent of $ \theta$.
\end{lem}

\begin{proof}
From Proposition \ref{pr:psiab}, we can see that the map  $ \gr \psi :\CO(\fg^{\oplus n+1})^\fg \rightarrow \CO(\fh \oplus \fg^{\oplus n})  $ is dual to the map of schemes
$$
\fh \oplus \fg^{\oplus n} \rightarrow \fg^{\oplus n+1} \sslash G \quad (x_0, x_1, \dots, x_n) \mapsto (x_0 - \sum \pi_-(x_i), x_1, \dots, x_n) .
$$
(This uses that $ (x, \pi_+(y)) = (\pi_-(x), y) $ for $ x, y \in \fg$.)

So it suffices to study the associated graded of the evaluation map $ U(\fh) \otimes U \fg^{\otimes n} \rightarrow U \fg^{\otimes n}$.

    If $ x \in \fh$, then the image of $ x^{(0)}$ under the evaluation map is the value of bilinear form $ (x, \theta) \in \C $.  As $ x^{(0)} $ has filtered degree 1, we see that the associated graded of the evaluation map takes $ x^{(0)}$ to 0.  Thus the resulting map of schemes is $$ \fg^{\oplus n} \rightarrow \fh \oplus \fg^{\oplus n} \quad (x_1, \dots, x_n) \mapsto (0, x_1, \dots, x_n) .$$
    By composition, the result follows.
\end{proof}

For each $ l = 1, \dots, r $, we define $$ \Phi^{trig}_l(u; z_1, \dots, z_n) := \gr \psi_\theta(\Phi_l(u; 0, z_1, \dots, z_n)) \in \CO(\fg^{\oplus n})(u).$$

By the above lemma, we see that for $ (x_1, \dots, x_n) \in \fg^{\oplus n}$, we have
    $$\Phi^{trig}_l(u; z_1, \dots, z_n) =  \Phi_l \left( \sum_{i=1}^n \frac{1}{u-z_i} x^{(i)} - \sum_{i=1}^{n} \frac{1}{u} \pi_-(x^{(i)})  \right)$$

This rational function $ \Phi^{trig}_l(u) $ has poles of order $ d_l $ at the points $ 0, z_1,\ldots, z_n$.  For $ i = 1, \dots, n$, we consider the Laurent expansions
$$
\Phi^{trig}_l(u; z_1, \dots, z_n) = \sum_{m=1}^{d_l} \Phi^{trig, m}_{l,i} (u -z_i)^{-m} + \CO(\fg^{\oplus n})[[u-z_i]]
$$
where $ \Phi^{trig, m}_{l,i} \in \CO(\fg^{\oplus n})$.  Equivalently, $ \Phi^{trig,m}_{l,i} = \gr S_{l,i}^{trig, m} = (\gr \psi_\theta)( \Phi^{m}_{l,i}$).

\begin{thm}
    The algebra $A^{trig}_\theta(z_1, \dots, z_n) $ is a polynomial algebra with generators $ \Phi^{trig, m}_{l,i}$ for $ l = 1, \dots, r$, $ i =1, \dots, n$, and $ m = 1, \dots,  d_l $.
\end{thm}

\begin{proof}
To see that these generate, we note that $ A(0, z_1, \dots, z_n)$ is generated by $  \Phi_{l,i}^m$ for $ i =1, \dots, n $ along with $ \Phi_{l,0}^{d_l} $ but this last generator is sent to 0 by $ \gr \psi_\theta$.

So it suffices to check that the generators $ \{ \Phi^{trig, m}_{l,i} \}$ for $ l = 1, \dots, r$, $ i =1, \dots, n$, and $ m = 1, \dots, d_l $ are algebraically independent.

In order to prove this, we consider the following filtration on $\CO(\fg^{\oplus n})$:
$$\deg x^{(1)} = 1, \quad \deg x^{(i)} = 0, \quad \text{ for } 2 \leq i \leq n, \, x \in \fg.$$ 
We will prove that the leading terms of $ \{ \Phi^{trig, m}_{l,i} \} $ with respect to this filtration are algebraically independent.

Let $ (e,h,f) $ be a principal $ \fsl_2$-triple.  In order to prove the algebraic independence of the leading terms $ lt(\Phi^{trig,m}_{l,i})$ of $ \{ \Phi^{trig, m}_{l,i} \} $, it suffices to show that their differentials are linearly independent at the point $ (f+h, f, \dots, f) \in \fg^{\oplus n} $.  In particular, we will show that these differentials span a subspace of dimension $n (r +p)$.

More precisely, let $$ V = \spann_{m,i,l} \bigl( d_{f+h, f, \dots, f}(lt(\Phi^{trig,m}_{l,i})) \bigr) \subset \fg^{\oplus n}$$  We will show that $ V $ contains $ \fb_- \oplus 0 $ and that the image of $ V $ in $ \fg^{\oplus n-1}$ (projecting to the last $ n-1$ summands) is $ \fb_-^{\oplus n-1}$.  Thus $ \dim V = n \dim \fb_- $ as desired.


First, we note that leading term of $\Phi_{l,1}^{trig, m} $ is the coefficient of $ (u-z_1)^{-m}$ in the Laurent expansion at $ z_1$ of 
$$\Phi_l \bigl( \tfrac{ 1}{u-z_1} x^{(1)}  - \tfrac{1}{u} \pi_-(x^{(1)}) \bigr).$$

The subalgebra generated by such coefficients coincides with the subalgebra spanned by all 
$p_{l,a} := \Phi_l(x^{(1)} + a \pi_-(x^{(1)})) \in \CO(\fg^{\oplus n})$ for  $a \in \bc$.  This follows from the fact that a rational function with one pole is uniquely determined by the principal part of its Laurent expansion at this pole.  We can regard $ p_{l,a}$ as the composition of $ \Phi_l$ with the linear map $ \fg^{\oplus n} \rightarrow \fg $ given by $ (x_1, \dots, x_n) \mapsto x_1 +a \pi_-(x_1)$.

So, by the chain rule, we have
$$
d_{f + h, f, \dots, f}(p_{l,a}) = \tau_a(d_{(1+a)f + h}(\Phi_l))
$$
where $ \tau_a : \fg \rightarrow \fg^{\oplus n}$ is the linear map given by $ x \mapsto x^{(1)} + a\pi_+(x)^{(1)}$.

Now, we will appeal to the following result.

\begin{lem}{\cite[Proof of Theorem 3.11]{fft}} \label{lem:fft}
We have
$$ \spann \bigl( d_{h + bf}(\Phi_l) : l \in 1, \dots, r, b \in \C \bigr) = \fb_-$$
\end{lem}

Combined with the above calculation, this shows that $$ \{ d_{f+h, f, \dots, f} (lt(\Phi^{trig, m}_{l,1})) : l = 1, \dots, r, m = 1, \dots, d_l \}$$ spans $ \fb_- \oplus 0 \subset \fg^{\oplus n}$.


Now fix some $i \ge 2$ and consider the leading term $ lt(\Phi_{l,i}^{trig, m})$.  This equals the coefficient of $(u - z_i)^{-m}$ in the Laurent expansion of   

$$\Phi_j\Bigl( x^{(1)} \tfrac{1}{u-z_1} + x^{(i)} \tfrac{ 1}{u-z_i} - \pi_-(x^{(1)}) \tfrac{1}{u}\Bigr).$$

As before these subalgebras are spanned by all $ p^i_{l, a,b} \in \CO(\fg^{\oplus n}) $ defined by
$$p^i_{l, a,b} := \Phi_l\bigl(x^{(1)} + a  x^{(i)} + b  \pi_-(x^{(1)})\bigr).$$
As before, we can regard $ p^i_{l,a,b}$ as the composition of $ \Phi_l$ with the linear map $ \fg^{\oplus n} \rightarrow \fg $ defined by $ (x_1, \dots, x_n) \mapsto x_1 + a  x_i + b  \pi_-(x_1)$.  Thus we can compute the differential using the chain rule and conclude
$$
d_{f+h, f, \dots, f} p^i_{l,a,b} = \tau_{a,b}(d_{(1+a +b) f + h}(\Phi_l))
$$
where $ \tau_{a,b} : \fg \rightarrow \fg^{\oplus n} $ is given by $ \tau_{a,b} = x^{(1)} + a x^{(i)} + b \pi_-(x)^{(1)}$.  
Projecting to $\fg^{\oplus n-1}$ and applying Lemma \ref{lem:fft}, we conclude that the projection of $ V $ contains $ \fb_-^{(i)}$.  Combining over all $ i$, we conclude that $ V $ has the desired properties and this concludes the proof.


\end{proof}

Now, we are in a position to prove Theorem \ref{size}.

\begin{proof}[Proof of Theorem \ref{size}]
The first statement follows from the fact that a homogeneous Gaudin subalgebra $\ma(0,z_1, \ldots, z_n)$ is invariant under dilations.

The elements $ S^{trig, m}_{l,i}$ generate the subalgebra since they are the images of the generators of $\ma(0,z_1, \dots, z_n) $, with the exception of $ S_{l,0}^{d_l} = \widetilde \Phi^{(0)}_l  $ which is sent to a scalar by Lemma \ref{le:thetarho}.  Monomials in these generators are linearly independent, since this is true after taking associated graded.  Thus, they generate a polynomial algebra.  The Hilbert-Poincar\'e series follows from the associated graded.

\end{proof}

\begin{rem}
Note that subalgebra $\ma_{\theta}^{trig}(z_1, \ldots, z_n)$ is not a maximal commutative subalgebra of $U\fg^{\otimes n}$. Indeed, we do not have any linear elements inside $\ma_{\theta}^{trig}(z_1, \ldots, z_n)$ and hence $\Delta^n(\fh)$ doesn't belong to this subalgebra. However, $\Delta^n(\fh)$ commutes with $\ma_{\theta}^{trig}(z_1, \ldots, z_n)$,  since it lies in $ (U \fg^{\otimes n})^{\fh}$
\end{rem}


\subsection{Definition from representations}
\label{4.3}
Let $ \theta \in \fh $ and $ z_1, \dots, z_n \in \mathbb C^\times $ be distinct, as above. 
 Since $ \ma^{trig}_\theta(z_1, \dots, z_n)$ is a subalgebra of $ U \fg^{\otimes n}$, it acts on any $n$-fold tensor product $$V(\ul) : = \bigotimes_{i=1}^n V(\lambda_i), $$ of irreducible finite-dimensional $\fg$-modules with highest weights $\lambda_1, \dots, \lambda_n$.  It is a natural question to describe this action.

Let $M(\theta) = U\fg \otimes_{U \fb} \C v_\theta$ be the Verma module with highest weight $ \theta $ (once again, we are identifying $ \fh$ and $ \fh^*$). We denote by $M^\vee(\theta)$ the contragredient dual of $M(\theta)$ and by $s_\theta:M(\theta)\to M^\vee(\theta)$ the (unique up to proportionality) homomorphism given by the Shapovalov pairing. $M^\vee(\theta)$ has the same character as $M(\theta)$ and is cofree with $1$ co-generator as a $U(\fn_+)$-module. Let $ \alpha^\vee_\theta : M^\vee(\theta) \rightarrow \C $ be the (unique up to proportionality) nonzero linear map sending all the weight vectors with the weights different from $\theta$ to 0. Similarly, let $ \alpha_\theta=\alpha^\vee_\theta\circ s_\theta : M(\theta) \rightarrow \C $ be the linear map sending $ v_\theta$ to 1 and all the weight vectors of other weights to 0.

The following proposition is a version of \cite[Lemma 4.1.1]{gr}.

\begin{prop}\footnote{We thank the referee for pointing out the natural generality of this statement and for suggesting this argument.}
\label{thm1}
Let $ V$ be any finite-dimensional representation of $ \fg$ and let $ \mu $ be an integral weight.  Then the linear map
$$
\pi^\vee_{\theta}: \Hom_{\fg}(M(\theta + \mu), M^\vee(\theta)\otimes V) \rightarrow V_\mu
$$
given by $ \varphi \mapsto (\alpha^\vee_{\theta} \otimes \Id)(\varphi(v_{\theta+\mu})) $ is an isomorphism.
\end{prop}

\begin{proof} Since $M(\theta)$ is a $\fg$-module induced from the $1$-dimensional $\fb$-module $\mathbb{C}v_\theta$, we have, for any $\fg$-module $L$, $\Hom_{\fg}(M(\theta), L)=\Hom_{\fb}(\mathbb{C}v_\theta, L)$. Next, $M^\vee(\theta)$ is $U(\fn_+)$-cofree with one generator of the weight $\theta$, i.e. it is the graded dual $\fb$-module of the induced from the $1$-dimensional $\fh$-module $\mathbb{C}v_\theta$. So, for any finite-dimensional $\fb$-modules $K,L$ that are semisimple with respect to $\fh$, we have $\Hom_\fb(K,M^\vee(\theta)\otimes L)=\Hom_\fb(K\otimes\mathbb{C}v_\theta^*,L)$. So we get \[\Hom_{\fg}(M(\theta + \mu), M^\vee(\theta)\otimes V)=\Hom_\fb(\mathbb{C}v_{\theta+\mu},M^\vee(\theta)\otimes V
)=\Hom_\fh(\mathbb{C}v_{\theta+\mu}\otimes \mathbb{C}v_\theta^*,V)=V_\mu.\]

This establishes that $  \Hom_{\fg}(M(\theta + \mu), M^\vee(\theta)\otimes V) $ and $ V_\mu $ have the same dimension. On the other hand, suppose that $ \varphi \in \Hom (M(\theta + \mu), M(\theta)\otimes V)$ and $(\alpha^\vee_{\theta} \otimes \Id)(\varphi(v_{\theta+\mu})) = 0$.  Then $ \varphi(v_{\theta + \mu}) \in M^\vee(\theta) \otimes V $ is a singular vector (annihilated by $ \fn_+$) and so the maximal non-zero weight component of its projection to $ M^\vee(\theta) $ is also a singular vector. As $ M^\vee(\theta) $ is $U(\fn_+)$-cofree, it has no singular vector of weight less than $ \theta$.  Thus $ \varphi = 0 $.

\end{proof}

\begin{corol}
Let $ V$ be any finite-dimensional representation of $ \fg$ and let $ \mu $ be an integral weight.  There is a linear map
$$
\pi_{\theta}: \Hom_{\fg}(M(\theta + \mu), M(\theta)\otimes V) \rightarrow V_\mu
$$
given by $ \varphi \mapsto (\alpha_{\theta} \otimes \Id)(\varphi(v_{\theta+\mu})) $.
This map is an isomorphism for generic $\theta$.
\end{corol}

\begin{proof} This follows from Proposition~\ref{thm1} since for generic $\theta$ we have $M(\theta)=M^\vee(\theta)$. 

An alternate (and a bit more elementary) proof is as follows. For generic $ \theta$ we have $ M(\theta) \otimes V \cong \oplus_\nu M(\theta + \nu) \otimes V_\nu$, where the direct sum is over the weights of $ V$.  Also for generic $ \theta$, each $ M(\theta + \nu) $ is irreducible and so $ \Hom(M(\theta + \mu), M(\theta + \nu)) $ is 0 unless $ \mu = \nu$, in which case it is one dimensional. This shows that $  \Hom_{\fg}(M(\theta + \mu), M(\theta)\otimes V) $ and $ V_\mu $ have the same dimension. The argument proving that $\pi_{\theta}$ is an isomorphism is the same as in Proposition~\ref{thm1}. 
\end{proof}

\begin{corol}
Let $ \theta $ be integral and sufficiently dominant.  Then the statement of Proposition \ref{thm1} holds true if $ M(\theta + \mu) $ and $ M(\theta)$ are replaced by  $V(\theta + \mu)$ and $ V(\theta)$.
\end{corol}

Consider the action of $\ma(0,z_1, \ldots, z_n)$ on $M(\theta) \otimes V(\ul)$.  Since the Gaudin algebra commutes with the diagonal $ \fg$ action, there is an algebra homomorphism $ \ma(0,z_1, \dots, z_n) \rightarrow \operatorname{End}_\fg(M(\theta) \otimes V(\ul))$ and thus $ \ma(0, z_1, \dots, z_n)$ acts on $ \Hom_\fg(M(\theta + \mu), M(\theta) \otimes V(\ul))$ for any $ \mu$.  Via Proposition \ref{thm1}, we have an isomorphism $ \Hom_\fg(M(\theta + \mu), M(\theta) \otimes V(\ul)) \cong V(\ul)_\mu $ and thus an action of $ \ma(0, z_1, \dots, z_n)$ on $V(\ul)_\mu$.

\begin{thm} \label{th:actiontensorprod}
This action of $\ma(0,z_1, \ldots, z_n)$ on $V(\ul)_{\mu}$ coincides (via $ \psi_\theta$) with the action of $\ma^{trig}_{\theta}(z_1, \ldots, z_n)$ on $V(\ul)_{\mu}$.
\end{thm}
\begin{proof}
By the universal property of Verma modules, we can identify $$ \Hom_\fg(M(\theta + \mu), M(\theta) \otimes V(\ul)) \cong (M(\theta) \otimes V(\ul))_{\theta + \mu}^{\fn_+} $$ via the map $ \varphi \mapsto \varphi(v_{\theta + \mu})$.

Recall the map $\psi: (U\fg^{\otimes n+1})^{\fg} \to U\fh \otimes U\fg^{\otimes n}$. Let $x \in (U\fg^{\otimes n+1})^{\fg}$. We claim that the following diagram is commutative:
\[ \begin{diagram}
\node{(M(\theta)\otimes V(\ul))^{\fn_+}_{\theta + \mu} }
\arrow{e,t}{\alpha_{\theta} \otimes id}
\arrow{s,r}{x \cdot}
\node{ V(\ul)}
\arrow{s,r}{\psi_\theta(x) \cdot} \\
\node{(M(\theta)\otimes V(\ul))^{\fn_+}_{\theta + \mu}} \arrow{e,t}{\alpha_{\theta} \otimes id}
\node{V(\ul)} 
\end{diagram}\]

Let $ v \in (M(\theta) \otimes V(\ul))^{\fn_+}$.

First, we can write $ x = x' + x'' $ where $ x' \in U \fg^{\otimes n+1} \Delta(\fn_+) $ and $ x'' \in U \fb_- \otimes U \fg^{\otimes n}$.   Since $ v$ is a singular vector, $ x' \cdot v = 0 $ and thus $ x \cdot v = x'' \cdot v$.  So following the diagram down and right, we reach $ (\alpha_\theta \otimes id)(x'' \cdot v)$.  

Now, if $ u \in M(\theta) $ and $ y \in U \fb_-$, then $ \alpha_\theta(y \cdot u) = \theta(y) \alpha_\theta(u)$, where by a slight abuse of notation, we use $ \theta $ for the composition $ U \fb_- \rightarrow U \fh \xrightarrow{\theta} \C $.  Thus  since $ x'' \in U \fb_- \otimes U \fg^{\otimes n}$, we have $(\alpha_\theta \otimes id)(x'' \cdot v) = (\theta \otimes id)(x'') \cdot (\alpha_\theta \otimes id)(v)$.

Since $ \psi_\theta(x) = (\theta \otimes id)(x'')$, the result follows. 
\end{proof}

The following proposition is a version of Ado's theorem:
\begin{prop}
Let $x \in (U\fg)^{\otimes n}$ be an element such that $x$ acts by zero on any $V(\ul)$. Then $x = 0$.
\end{prop}

By this proposition and Theorem \ref{th:actiontensorprod} one can {\it define} the trigonometric Gaudin subalgebra $\ma_{\theta}^{trig}(z_1, \ldots, z_n)$  for {\it generic}  $\theta \in \fh^*$ as the unique subalgebra in $U\fg^{\otimes n}$ that acts as above on any $V(\ul)$.
 
\subsection{Universal trigonometric Gaudin subalgebras.}
\label{4.4}

The trigonometric Gaudin subalgebras can be described universally. Namely, there are two versions of a \emph{universal trigonometric Gaudin subalgebra}, $\ma^{trig}_\theta\subset U(\hat{\fn}_+)$ and  $\widetilde{\ma}^{trig}_\theta\subset U(\hat{\fb}_-)$ where  $\hat{\fn}_+:=\fn_+ + t^{-1}\fg[t^{-1}]$ and $\hat{\fb}_-:=\fb_- + t^{-1}\fg[t^{-1}]$.

First, we write $ \fg[t^{-1}] = \fb_- \oplus \hat \fn_+$.  We consider the quantum Hamiltonian reduction of $ U(\fg[t^{-1}]) $ by $ \fb_- $ with the character $\theta:\fb_-\to\bc$.  By the general theory, we obtain an injective algebra homomorphism  $$(U(\fg[t^{-1}])/U(\fg[t^{-1}])(\{x-\theta(x)\ |\ x\in\fb_-\}))^{\fb_-} \rightarrow U \hat \fn_+.$$ 
The subalgebra $U(\fg[t^{-1}])^\fg$ of $\fg$-invariants in $U(\fg[t^{-1}])$ naturally maps to this Hamiltonian reduction.  By composition, we have a homomorphism $U(\fg[t^{-1}])^\fg\to U(\hat{\fn}_+)$ and we let  $\ma_\theta^{trig}\subset U(\hat{\fn}_+)$ be the image of $\widetilde{\ma} \subset U(\fg[t^{-1}])^\fg$ under this homomorphism.

Similarly, we can write $ \fg[t^{-1}] = \fn_+ \oplus \hat \fb_- $ and consider the quantum Hamiltonian reduction by $\fn_+$. We obtain an algebra homomorphism
$$ (U(\fg[t^{-1}])/U(\fg[t^{-1}])\fn_+)^{\fn_+} \rightarrow U \hat \fb_-$$
Again by composition, we get a homomorphism $U(\fg[t^{-1}])^\fg\to U(\hat{\fb}_-)$.  We let  $\widetilde{\ma}^{trig}\subset U(\hat{\fb})$ be the image of $\widetilde{\ma} \subset U(\fg[t^{-1}])^\fg$ under the homomorphism $i$; it is naturally a commutative subalgebra. 

Next, any $\theta\in \fh \cong \fh^* $ determines a character $ev_{\infty}^\theta:\hat{\fb}_-\to\bc$ by annihilating $\fn_-+t^{-1}\fg[t^{-1}]$ and applying $\theta$ to $\fh$. So we get a family of subalgebras $$ \widetilde{\ma}^{trig}_\theta:= (ev_{\infty}^\theta\otimes\Id)\circ\Delta (\widetilde{\ma}^{trig})\subset U(\hat{\fb}_-).
$$

For any $z\in\bc^\times$, we have the restriction of the homomorphism $ev_{z}:\fg[t^{-1}]\to\fg$ of evaluation at $z$ to $\hat{\fn}_+$ and $\hat{\fb}_-$. This gives rise to the homomorphisms $ev_{z_1,\ldots,z_n}:U(\hat{\fn}_+)\to U\fg^{\otimes n}$ and $ev_{z_1,\ldots,z_n}:U(\hat{\fb}_-)\to U\fg^{\otimes n}$ for any collection of nonzero complex numbers $z_1,\ldots,z_n$. 

\begin{prop}
    We have $ev_{z_1^{-1},\ldots,z_n^{-1}}(\widetilde{\ma}^{trig}_\theta)=ev_{z_1,\ldots,z_n}(\ma^{trig}_\theta)=\ma^{trig}_\theta(z_1,\ldots,z_n)$.
\end{prop}

\begin{proof}
    Take any $(n+2)$-tuple $w_\infty,w_0,w_1,\ldots,w_n\in\bc^\times$that is $PGL_2$-equivalent to $\infty,0,z_1,\ldots,z_n$. Then both $ev_{z_1^{-1},\ldots,z_n^{-1}}(\widetilde{\ma}^{trig}_\theta)$ and $ev_{z_1,\ldots,z_n}(\ma^{trig}_\theta)$ are the images of $\ma(w_\infty,w_0,w_1,\ldots,w_n)\subset (U\fg\otimes U\fg\otimes U\fg^{\otimes n})^\fg$ under the following composite homomorphism
    \begin{multline*}
        (U\fg\otimes U\fg\otimes U\fg^{\otimes n})^\fg\to (U\fg\otimes U\fg\otimes U\fg^{\otimes n}/\Delta^{(n+2)}(\fg))^\fg\to\\ \to U(\fn_+)\otimes U\fb_-\otimes U\fg^{\otimes n} \to U\fg^{\otimes n},
    \end{multline*}
    where the first arrow is the natural projection to the quantum Hamiltonian reduction by the diagonal $\fg$, the second one embeds it to the universal enveloping of the complement subalgebra $\fn_+\oplus\fb_-\oplus\fg^{\oplus n}$, and the last arrow takes the $0$ character on $\fn_+$ and $\theta$ on $\fb_-$.

\end{proof}

\begin{thm} \begin{enumerate}
\item The subalgebra $\widetilde{\ma}^{trig}_{\theta}\subset U(\hat{\fb}_-)$ contains the quadratic $\widetilde{\Gamma}_{\theta}^{trig}:=\theta[-1]+\sum\limits_{a}h_a[-1]h^a[0]+\sum\limits_{a}e_a[-1]f_a[0]$. Moreover the centralizer of $\widetilde{\Gamma}_{\theta}^{trig}$ in $U(\hat{\fb}_-)$ is generated by  $\widetilde{\ma}^{trig}_{\theta}$ and $\fh[0]$.
    \item The subalgebra $\ma^{trig}_{\theta}\subset U(\hat{\fn}_+)$ contains the quadratic $\Gamma_{\theta}^{trig}:=\theta[-1]+\sum\limits_{a}f_a[-1]e_a[0]$. Moreover $\ma^{trig}_{\theta}$ is the centralizer of $\Gamma_{\theta}^{trig}$ in $U(\hat{\fn}_+)$.
    \end{enumerate}
\end{thm}

\begin{proof}
    Observe that both $\widetilde{\Gamma}_{\theta}^{trig}$, $\Gamma_{\theta}^{trig}$ are the images of $\Gamma\in \widetilde{\ma}\subset U(\fg[t^{-1}])$. To describe the centralizers of $\widetilde{\Gamma}_{\theta}^{trig}$, $\Gamma_{\theta}^{trig}$ we just have to repeat the proof of \cite[Theorem~1]{r3}. Namely, the proof breaks into the following steps:
    \begin{enumerate}
        \item We have to show that the centralizer is not bigger than expected. So it is sufficient to see it in the classical limit, i.e. we have to check that the Poisson centralizers of the leading terms (with respect to the PBW filtration) of $\widetilde{\Gamma}_{\theta}^{trig}$, $\Gamma_{\theta}^{trig}$ in $S\hat{\fb}$, $S\hat{\fn}$ are $\gr \widetilde{\ma}^{trig}_{\theta}\cdot S\fh[0]$ and $\ma^{trig}_{\theta}$, respectively. Moreover, it suffices to check that the Hilbert-Poincar\'e series of the centralizers in question are not greater than those expected. Slightly abusing notations, we denote the leading terms of $\widetilde{\Gamma}_{\theta}^{trig}$, $\Gamma_{\theta}^{trig}$ in $S\hat{\fb}$, $S\hat{\fn}$ by the same letters $\widetilde{\Gamma}_{\theta}^{trig}$ and $\Gamma_{\theta}^{trig}$.
        \item To any regular $\eta\in\fh$, we assign a $1$-parametric family $\widetilde{\varphi}_s$ of Poisson automorphisms of $S\hat{\fb}$ such that $$\lim\limits_{s\to\infty}s^{-1}\widetilde{\varphi}_s(\widetilde{\Gamma}^{trig}_{\theta})=\eta[-1] \text{ and } \lim\limits_{s\to\infty}\widetilde{\varphi}_s(\gr \widetilde{\ma}^{trig}_{\theta}\cdot S\fh[0])=S\fh[t^{-1}].$$ Namely, we define $\widetilde{\varphi}_s(h_a[0]):=h_a[0]+s(\eta,h_a)$ and $\widetilde{\varphi}_s(e_a[0])=e_a[0];\ \widetilde{\varphi}_s(e_a[k])=e_a[k];\ \widetilde{\varphi}_s(f_a[k])=f_a[k];\ \widetilde{\varphi}_s(h_a[k])=h_a[k]$ for all $k<0$. 
        
        Similarly, to the principal nilpotent element $f\in\fn_-\subset\fg$, we assign the family $\varphi_s$ of Poisson automorphisms of $S\hat{\fn}$ defined as $\varphi_s(e_a[0]):=e_a[0]+s(f,e_a)$ and $\varphi_s(e_a[k])=e_a[k];\ \varphi_s(f_a[k])=f_a[k];\ \varphi_s(h_a[k])=h_a[k]$ for all $k<0$. This family of automorphisms satisfies the properties $\lim\limits_{s\to\infty}s^{-1}\varphi_s(\Gamma^{trig}_{\theta})=e[-1]$ and $\lim\limits_{s\to\infty}\varphi_s(\gr \ma^{trig}_{\theta})=S(t^{-1}\fz_\fg(e)[t^{-1}])$, where $e\in\fn_+\subset\fg$ is the opposite principal nilpotent element.
        \item It is suffices to check that the above limit subalgebras are Poisson centralizers of the limits of the quadratic elements, i.e. of $\eta[-1]$ and $e[-1]$, respectively. Indeed, this implies that $\widetilde{\varphi}_s(\gr\widetilde{\mathcal{A}}^{trig}_\theta)$ is the Poisson centralizer of $\widetilde{\varphi}_s(\widetilde{\Gamma_\theta^{trig}})$ for $s$ large enough, and hence for any $s\in\mathbb{C}$ since $\widetilde{\varphi}_s$ is a Poisson automorphism. Similarly, for $\gr\mathcal{A}_\theta^{trig}$.  
        \item By the argument entirely similar to \cite[Lemma 4]{r3}, $S\fh[t^{-1}]$ is exactly the Poisson centralizer of $\eta[-1]$ in $S\hat{\fb}$ and $S(t^{-1}\fz_\fg(e)[t^{-1}])$ is exactly the Poisson centralizer of $e[-1]$ in $S\hat{\fn}$, so the Hilbert-Poincar\'e series of the centralizers are not bigger than the expected. 
        
        Let us reproduce this argument for completeness. Choose a basis $\{x_a\}$ in $\fg$ compatible with $\mathrm{ad}_\eta$ (respectively, with $\mathrm{ad}_e$), i.e. an eigenbasis for $\mathrm{ad}_\eta$ and a Jordan basis for $\mathrm{ad}_e$, respectively. Multiplying its elements by negative powers of $t$, we get bases in $\hat \fb_-$ (respectively, in $\hat \fn_+$). Take any ordering of the basis $\{x_a\}$ such that for $[\eta,x_a]=0,\ [\eta,x_b]\ne0$ (respectively, $[e,x_a]=0,\ [e,x_b]\ne0$) we have $x_a<x_b$. In the second case, we need additionally $[e,x_b]<x_b$. Extend these orders to the bases $\{x_a[k]\}$ in such a way that $x_a[k]<x_b[m]$ if $x_a<x_b$ or if $x_a=x_b$ and $k>m$. This determines the lexicographic ordering on the monomials of $\{x_a[k]\}$ preserved by $\mathrm{ad}_{\eta[-1]}$ (respectively, by $\mathrm{ad}_{e[-1]}$). Now it remains to notice that for any element not from $S(\fh[t^{-1}])$ (respectively, not from $S(t^{-1}\fz_\fg(e)[t^{-1}])$), its leading term is mot annihilated by $\mathrm{ad}_{\eta[-1]}$ (respectively, by $\mathrm{ad}_{e[-1]}$).    
    \end{enumerate}
    
\end{proof}

Now we can compare our trigonometric Gaudin subalgebras with those introduced by Molev and Ragoucy in \cite{mr} in type $A$. Namely, in \cite{mr}, some higher Hamiltonians in $U(\hat \fb_+)$ are constructed (i.e. the roles of positive and negative roots are switched, compared to ours).

\begin{corol}
    The higher trigonometric Gaudin Hamiltonians in type A from \cite[Section 2]{mr} lie in $\tilde \ma^{trig}_0 \subset U(\hat \fb_+)$.
\end{corol} 
\begin{proof}
    The higher Hamiltonians described in \cite{mr} come from a universal commutative subalgebra in $U(\hat \fb_+)$ that contains $\Gamma_{0}^{trig}$.
\end{proof}







\subsection{Compactification of the family of trigonometric Gaudin subalgebras}
\label{comp}
We will now consider the closure of the family $\ma_{\theta}^{trig}(z_1, \ldots, z_n)$ with fixed $\theta \in \fh$. Since we have defined trigonometric Gaudin subalgebras as the images of Gaudin subalgebras, it is not surprising that the limit trigonometric Gaudin subalgebras are images of the limit Gaudin subalgebras.

Let $ C \in \overline M_{n+2}$, a genus 0 nodal curve with $ n+2$ marked points $ z_0, \dots, z_{n+1}$.  From section \ref{se:compactGaudin}, we have a (limit) Gaudin subalgebra $ \ma(C) \subset (U \fg^{\otimes n+1})^\fg$ and we define $ \ma_{\theta}^{trig}(C) = \psi_{\theta}(\ma(C))$, where $ \psi_\theta : (U \fg^{\otimes n+1})^\fg \rightarrow U \fg^{\otimes n}$ was constructed in section \ref{se:deftrig}.

We endow $ \ma_\theta^{trig}(C) $ with the filtration coming from intersecting with the PBW filtration of $ U \fg^{\otimes n} $.

\begin{thm}
\label{trigcomp}
The subalgebras $ \ma_\theta^{trig}(C)$ form a flat family of subalgebras faithfully parametrized by $ \overline M_{n+2}$.  Moreover, this is the compactification of the family of trigonometric Gaudin algebras.
\end{thm}

\begin{proof}
    By Proposition \ref{pr:HPsame}, we see the Hilbert-Poincar\'e series of $ \CA^{trig}_\theta(C)$ is the same for any $ C \in \overline M_{n+2} $.  Thus, we have a flat family and (as explained in section \ref{se:Family}), we obtain maps $ \overline M_{n+2} \rightarrow \Gr(r_k, (U \fg^{\otimes n})^{(k)}$ for each $ k \in \bn $.  By Theorem \ref{thm:trigiso}, this map is an isomorphism when $ k = 2$.  Thus the parameterization is faithful.  
\end{proof}

Using Proposition \ref{pr:limitGaudin}, we will now give a description of the subalgebra $ \ma_\theta^{trig}(C)$.  As in Proposition \ref{pr:limitGaudin}, we let $ C_\infty$ be the component of $ C $ containing $ z_{n+1}$, we identify $ C_\infty $ with $ \CP^1 $ in such a way that $ z_{n+1} $ is identified with $ \infty$, and we let $ C_0, \dots, C_m$ be the curves attached at points $ w_0, \dots, w_m$ with corresponding labels $ B_0, \dots, B_m$.  Note that the index set for the marked points of $ C $ is $ \{0, \dots, n+1 \} $ so that $ \CB = (B_0, \dots, B_m) $ is an ordered set partition of $ \{0,\ldots, n\}$.  We choose this ordering so that $ 0 \in B_0$.  We further assume (this is possible by using an automorphism of $ \CP^1$) that $ w_0 = 0$.

Applying Proposition \ref{pr:limitGaudin}, we see that
$$  \ma_\theta^{trig}(C) = \psi_\theta \Bigl( \Delta^\CB(\ma(0, w_1, \dots, w_m)) \otimes_Z \bigotimes_{k=0}^m \ma(C_k) \Bigr)$$
and thus we should calculate $ \psi_\theta(\Delta^\CB(\ma(0, w_1, \dots, w_m)) $ and $ \psi_\theta( j_{B_k}^{n+1} (\ma(C_k))) $, where $ j_{B_k}^{n+1} : U \fg^{\otimes B_k} \rightarrow U \fg^{\otimes n+1}$ denotes the inclusion of these tensor factors, as defined in Section \ref{se:Delta}.

There are two easy cases.
\begin{lem}
  \begin{enumerate}
      \item Suppose that $ k > 0$ (i.e. $ z_0 $ does not lie on $ C_k$).  Then $ \psi_\theta(j_{B_k}^{n+1}(\ma(C_k))) = j_{B_k}^n(\ma(C_k))$.
      \item Suppose that $ B_0 = \{0 \}$ (i.e. $ z_0 $ lies on $C_\infty$).  Then $$ \psi_\theta(\Delta^\CB(\ma(0, w_1, \dots, w_m))) = \Delta^{\CB'}(\ma^{trig}(w_1, \dots, w_m)).$$
      where $\CB'= (B_1, \dots, B_m)$.
  \end{enumerate}  
\end{lem}

\begin{proof}
    For (1), we note that the map $$ (U \fg^{\otimes n})^{\fg} \xrightarrow{ a\mapsto 1 \otimes a} (U \fg^{\otimes n+1})^\fg \xrightarrow{\psi_\theta} U \fg^{\otimes n}$$ is the identity.

    For (2), we note that the following diagram commutes
$$
\begin{tikzcd}
 (U \fg^{\otimes m+1})^{\fg} \ar[r,"\Delta^\CB"] \ar[d,"\psi_\theta"] & (U \fg^{\otimes n+1})^\fg \ar[d,"\psi_\theta"] \\
 U \fg^{\otimes m} \ar[r, "\Delta^{\CB'}"] & U \fg^{\otimes n}
\end{tikzcd}
$$

\end{proof}

\begin{figure} 
	\includegraphics[trim=0 60 0 40, clip,width=\textwidth]{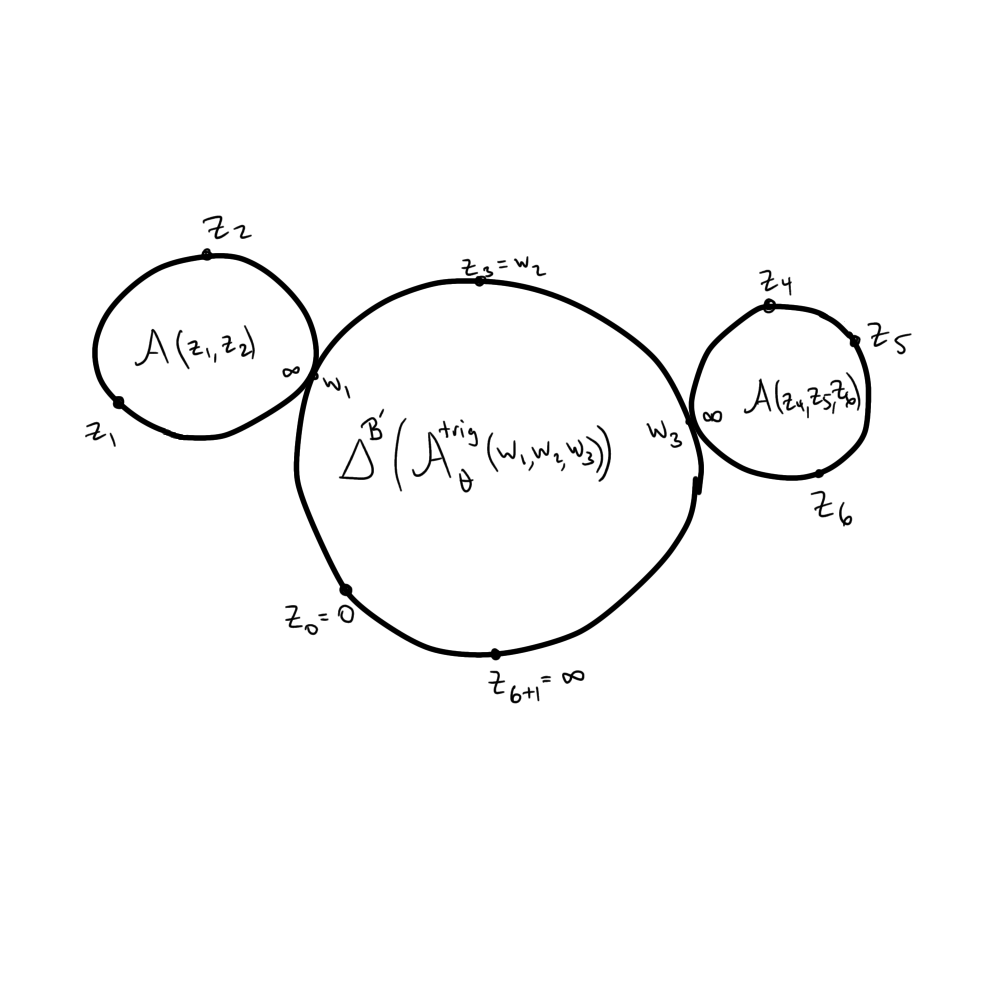}
	\caption{This is a curve $ C \in \overline M_{6+2}$.  We have $n = 6$, $m= 4 $, and $ \CB = (\{0\}, \{1,2\}, \{3\}, \{4,5,6\})$.  Each component contributes the algebra shown and the whole curve gives $ \ma^{trig}_\theta(C) = \Delta^{\CB'}(\ma^{trig}_\theta(w_1, w_2, w_3)) \ma(z_1, z_2) \ma(z_4, z_6, z_6) $. } \label{fig:TrigGaudin1}
\end{figure}

The lemma leads immediately to the following description of the limit trigonometric Gaudin subalgebra in the case where $z_0$ and $z_{n+1}$ lie on the same component, which is illustrated in Figure \ref{fig:TrigGaudin1}.
\begin{corol}
\label{co:trig-easy-limit}
If $ z_0 \in C_\infty$, then
$$ \ma^{trig}_\theta(C) = \Delta^{\CB'}(\ma^{trig}_\theta(w_1, \dots, w_m)) \otimes_{Z'} \bigotimes_{k=1}^m \ma(C_k)$$
where $ Z' = \Delta^{\CB'}(Z(U\fg^{\otimes m}))$.  Moreover, its Hilbert-Poincar\'e series is the same as that for the trigonometric Gaudin subalgebras.
\end{corol}

\begin{proof}
    Indeed, the algebra on the left is the product of the two on the right since it is so for homogeneous Gaudin algebras. Next, according to \cite{kn}  (see \cite[Lemma 3.10]{r} for more details) the product of $\Delta^{\CB'}(U\fg^{\otimes m})$ and the subalgebra of invariants $(U\fg^{\otimes n})^{\Delta^{\CB'}\fg}$ inside $U\fg^{\otimes n}$ is the tensor product over $Z'$ and both are free $Z'$-modules. Since $\Delta^{\CB'}(\ma^{trig}_\theta(w_1, \dots, w_m))\subset \Delta^{\CB'}(U\fg^{\otimes m})$, $\bigotimes_{k=1}^m \ma(C_k)\subset (U\fg^{\otimes n})^{\Delta^{\CB'}\fg}$ and both contain $Z'$, the first assertion follows. Finally, since both are free $Z'$-modules, the statement about Hilbert-Poincar\'e series follows by multiplying the Hilbert-Poincar\'e series of the tensor factors and dividing by the Hilbert-Poincar\'e series of $ Z'$.
\end{proof}

Unfortunately, when $ z_0$ and $ z_{n+1} $ do not lie on the same component (equivalently $ z_0 \notin C_\infty$) both $ \psi_\theta(\Delta^\CB(\ma(0, w_1, \dots, w_m)) $ (the contribution from $C_\infty$) and $ \psi_\theta( j_{B_0}^{n+1}( \ma(C_0))$ are not homogeneous or trigonometric Gaudin algebras (even up to diagonal embedding). To see this, consider a two-component curve $ C = C_0 \cup C_\infty$ with both identified with $ \CP^1$ in such a way that the intersection point is $ \infty \in C_0$ and $ 0 = w_0\in C_\infty$.  Also assume that $ z_0, z_1, \dots, z_{n-m}$ lie on $C_0$ with $ z_0 = 0$, and that $ z_{n-m+1} = w_1, \dots, z_n = w_m, z_{n+1}$ lie on $ C_\infty$ with $z_{n+1} = \infty$. (See Figure \ref{fig:TrigGaudin2} for an example.)

\begin{figure} 
	\includegraphics[trim=10 50 10 55, clip,width=0.8\textwidth]{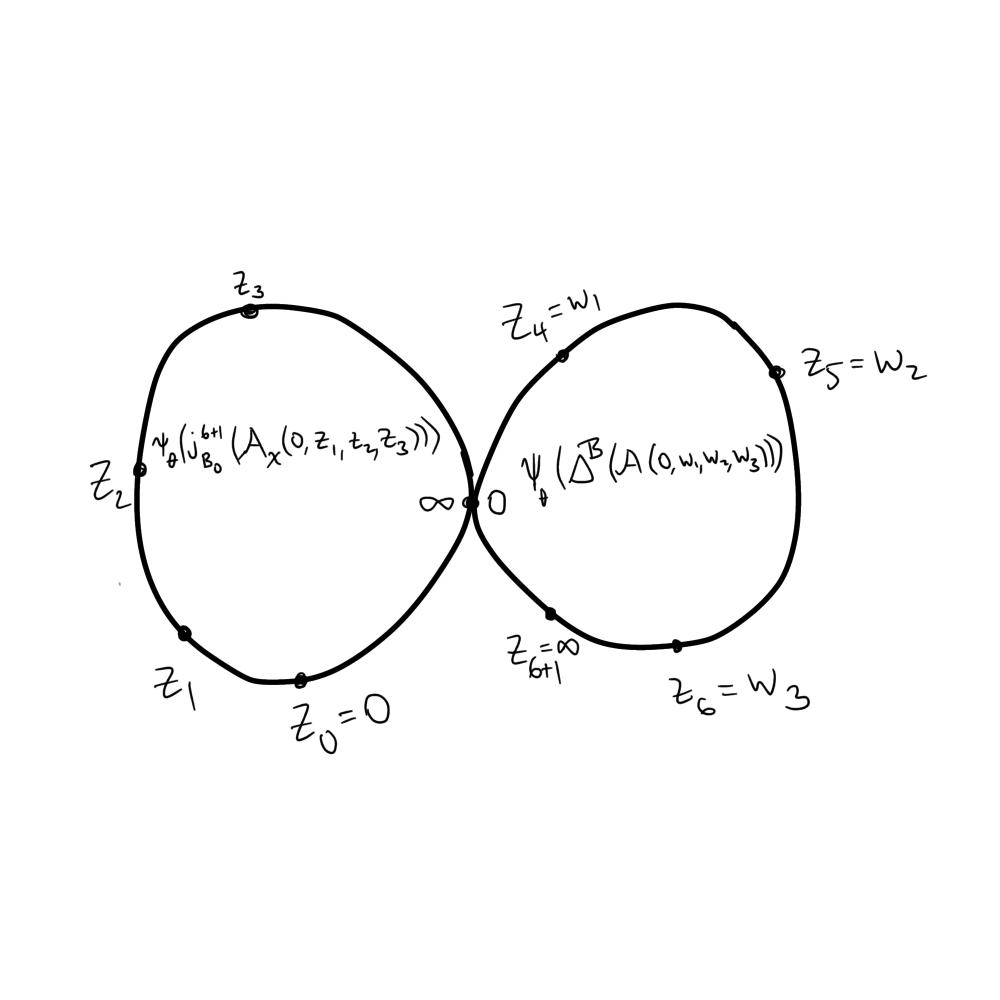}
	\caption{This is a curve $ C \in \overline M_{6+2}$ where $ z_0 $ does not lie on the same component as $ z_{6+1}$.  We have $n = 6$, $m= 4 $, and $ \CB = (\{0,1,2,3\}, \{4\} , \{5\},\{6\})$.  Each component contributes the algebra shown and the whole curve gives $ \ma^{trig}_\theta(C) = \psi_\theta(j_{\CB}^{6+1} \ma(0,z_1, z_2, z_3))) \psi_\theta( \Delta^\CB( \ma(0, z_4, z_5, z_6)) $. } \label{fig:TrigGaudin2}
\end{figure}

An easy computation, using the procedure explained after Definition \ref{def:trigGaudin}, yields the following.
\begin{lem}
    In this setting, the subalgebra $ \psi_\theta(\Delta^\CB(\ma(0, w_1, \dots, w_m)))$ contains
    $$
    \psi_\theta(\Delta^\CB(H_i)) = \sum_{\substack{j=n-m+1 \\ j \ne i}}^n \frac{\Omega^{(ij)}}{z_i - z_j} + \sum_{j=1}^{n-m} \frac{\Omega^{(ij)}}{z_i} + \frac{\theta^{(i)}}{z_i} - \sum_{j=1}^n \frac{\Omega^{(ij)}_-}{z_i} 
    $$
    for $ i = n-m+1, \dots, n$, while the subalgebra $ \psi_\theta(j_{B_0}^n(\ma(0, z_{1}, \dots, z_{n-m}))$ contains 
    $$
    \psi_\theta(H_i) =  \sum_{\substack{j=1 \\ j \ne i}}^{n-m} \frac{\Omega^{(ij)}}{z_i -z_j} + \frac{\theta^{(i)}}{z_i} - \sum_{j=1}^n \frac{\Omega^{(ij)}_-}{z_i}
    $$
    for $ i = 1, \dots, n-m$.
\end{lem}

Consider a \emph{caterpillar} curve $C\in \overline{M}_{n+2}$, i.e. a curve $ C = C_1 \cup \dots \cup C_l$, where each $ C_k $ is a projective line, and where each pair $C_k, C_{k+1} $ meet transversely at a single point (with no other intersections), such that $ z_0 \in C_1 $ and $ z_{n+1} \in C_l$. Let $C_\infty=C_l$ be the component containing $z_{n+1}$ and $C_0$ be the union of all other components. Let $m+1$ be the number of marked points on $C_\infty$. We can assume without loss of generality that $z_{n-m+1},\ldots,z_{n+1}\in C_\infty$. Then the (homogeneous) Gaudin subalgebra $\ma(C)\subset U\fg^{\otimes(n+1)}$ is the product of $\Delta^{[0,n-m],n-m+1,n-m+2,\ldots,n}(\ma(C_\infty))$ and $\ma(C_0)\otimes1^{\otimes m}$. Slightly abusing notations, we denote the first factor simply by $\ma(C_\infty)$, and the second simply by $\ma(C_0)$.  
\begin{lem}\label{sizelemma}
The Hilbert-Poincar\'e series of $\psi_\theta(\ma(C))=\psi_{\theta}(\ma(C_{\infty})) \cdot \psi_{\theta}(\ma(C_0))$ coincides with that of a generic trigonometric Gaudin subalgebra in $U\fg^{\otimes n}$.
\end{lem}
\begin{proof}
Consider the associated graded subalgebra $\gr \psi_\theta(\ma(C))\subset S\fg^{\otimes n}$. We have to show that this subalgebra has at least the same number of algebraically independent generators of the same degrees as for a generic trigonometric Gaudin subalgebra. The subalgebra $\gr \psi_{\theta}(\ma(C_{\infty}))$ contains coefficients of the principal parts at $z_i\in C_\infty$ of 
\begin{equation}\label{eq:c-infty}
    \Phi_l\left(\sum_{i=1}^{n-m} \frac{ x^{(i)}}{u} + \sum_{i=n-m+1}^n \frac{x^{(i)}}{u - z_i} - \pi_-\left(\frac{\sum_{i=1}^n x^{(i)}}{u}\right) \right).
\end{equation}
 
Now let us consider the set of components $C_1, \ldots C_k$ of $C_0$, where we count them from the one containing $z_0$.
The subalgebra $\psi_{\theta}(\ma(C_0))$ contains the coefficients of principal parts at all the points $z_i\in C_k$ of  \begin{equation}\label{eq:c-0}\Phi_l\left(\sum_{i \in B_k} \dfrac{x^{(i)}}{u - z_i} - \sum_{i \in \cup_{r=1}^{k-1} B_r}\dfrac{x^{(i)}}{u}  - \pi_-\left(\frac{\sum_{i=1}^n x^{(i)}}{u}\right) \right).
\end{equation}
Here $l=1, \ldots, k$.

We use the same idea as in the proof of Theorem \ref{size} to show that the above coefficients of (\ref{eq:c-infty}) and (\ref{eq:c-0}) are algebraically independent, and this would be sufficient. Namely, we consider the filtration on $S(\fg)^{\otimes n}$ where the $n$-th component has degree $1$ and all the others have degree $0$, take the leading terms with respect to this filtration, and show that the differentials of these leading terms at the point $y=(h, \ldots, h, \underbrace{f, \ldots, f}_{m-1}, f+h)$ are linearly independent. 

Indeed, for the generators coming from (\ref{eq:c-infty}) we have the leading terms given by the coefficients of the principal part of 
$$\Phi_j\left( \frac{ x^{(i)}}{u-z_i}+\frac{ x^{(n)}}{u-z_n}- \pi_-\left(\frac{ x^{(n)}}{u}\right)\right)$$
at the points $z_i$ with $i=n-m+1,\ldots,n-1$ and 
$$\Phi_j\left( \frac{ x^{(n)}}{u-z_n} - \pi_+\left(\frac{ x^{(n)}}{u}\right)\right)$$
at the point $z_n$. Due to the same reason as in Theorem~\ref{size}, their differentials at $y$ span $0^{\oplus(n-m)}\oplus\mathfrak{b}_-^{\oplus m}$.

For generators  coming from (\ref{eq:c-0}) we have the leading terms given by the coefficients of the principal parts of 
$$\Phi_j\left( \frac{ x^{(i)}}{u-z_i} - \pi_-\left(\frac{ x^{(n)}}{u}\right)\right)$$ at $z_i$ for $i=1,\ldots,n-m$. This means that the linear span of such leading terms is the span of $\Phi_j\left(a x^{(i)} + b \pi_-(x^{(n)})\right)$ for all $a,b\in\mathbb{C}$. Fix $i$ and consider the linear span of
$$d_y\Phi_j\left(a x^{(i)} + b  \pi_-(x^{(n)})\right).$$
It lies in $\fg^{(i)}\oplus\fg^{(n)}$ and as an element of $\fg^{(i)}\oplus\fg^{(n)}$ it is equal to $$d_{(h,f+h)} \Phi_j\left(a  x^{(i)} + b  \pi_-(x^{(n)})\right).$$ Consider the projection of the span of such differentials for all $\Phi_j$ to the first summand. This is the same as  considering $x^{(n)}$ not as a variable but as a parameter and evaluating it at $x^{(n)}=h+f$. This is the subspace of 
$\fg^{(i)}$ spanned by all
$$d_h \Phi_j(a x + b f),$$ that is $\mathfrak{b}_-^{(i)}$. This means that the sum of the above space modulo $0^{\oplus(n-m)}\oplus\mathfrak{g}^{\oplus m}$ is $\mathfrak{b}_-^{\oplus(n-m)}$. 

So the linear span of differentials at $y$ of leading terms for both types of generators span a subspace in $\fg^{\oplus n}$ whose intersection with $0^{\oplus(n-m)}\oplus\fg^{\oplus m}$ contains $0^{\oplus(n-m)}\oplus\fb_-^{\oplus m}$ and whose quotient modulo $0^{\oplus(n-m)}\oplus\fg^{\oplus m}$ is $\mathfrak{b}_-^{\oplus(n-m)}$. So the dimension of this space is not less than $m\dim\fb_-$, which is the number of the generators we consider, hence the generators are algebraically independent. This completes the proof of the Lemma.

\end{proof}

 From this, we can see that any limit trigonometric Gaudin subalgebra has the same size.

\begin{prop} \label{pr:HPsame}
    For any $\theta\in\fh^*$ and any $C\in \overline{M}_{n+2}$, the Hilbert-Poincar\'e series of $\ma_\theta^{trig}(C)$ is the same as that for generic trigonometric Gaudin subalgebra in $U\fg^{\otimes n}$.
\end{prop}

\begin{proof} Consider arbitrary $C\in\overline{M}_{n+2}$. We have to show that the Hilbert-Poincar\'e series of $\ma_\theta^{trig}(C)=\psi_\theta(\ma(C))$ is not smaller than that for generic $C$. 

Denote by $C_\infty$ the union of its components between the one containing $z_0$ and the one containing $z_{n+1}$. By construction, $C_\infty$ is a caterpillar curve. Let $m$ be the number of special points on $C_\infty$ (i.e. either marked points apart from $z_0,z_{n+1}$, or intersection points with the components outside $C_\infty$). So the original $C$ is constructed by attaching some stable rational curves $C_1,\ldots,C_m$ to $C_\infty$ at the special points (the $C_k$'s corresponding to the marked points are elements of $\overline{M}_{2}$, i.e. are just points). We can view $C_\infty\in \overline{M}_{m+2}$ as a limit of a family of non-degenerate curves $C'_\infty\in M_{m+2}$, i.e. $C'_\infty$ is just $\mathbb{P}^1$ with the marked points $0,w_1,\ldots,w_m,\infty$. Then the original curve $C$ is the limit of the family of curves $C'$ obtained by attaching $C_1,\ldots,C_m$ to $C'_\infty$ at the points $w_1,\ldots,w_m$.

From Corollary \ref{co:trig-easy-limit} we have the following description of a subalgebra corresponding to a curve $C'$:
$$ \ma^{trig}_\theta(C') = \Delta^{\CB'}(\ma^{trig}_\theta(w_1, \dots, w_m)) \otimes_{Z'} \bigotimes_{k=1}^m \ma(C_k)$$
where $ Z' = \Delta^{\CB'}(Z(U\fg^{\otimes m}))$. So, in particular, Hilbert-Poincar\'e series of $\ma^{trig}_\theta(C')$ is the same as that for generic $\ma^{trig}_\theta(z_1,\ldots,z_n)$.

Next, passing to the limit $C'\to C$, we get 
$$ \ma^{trig}_\theta(C)=\psi_\theta(\ma(C)) = \Delta^{\CB'}(\psi_\theta(\ma(C_\infty))) \otimes_{Z'} \bigotimes_{k=1}^m \ma(C_k).$$

According to Lemma \ref{sizelemma}, $\psi_\theta(\ma(C_\infty))$ has the same Hilbert-Poincar\'e series as $\ma^{trig}_\theta(w_1, \dots, w_m)$, so the Hilbert-Poincar\'e series of $\ma^{trig}_\theta(C)$ is not smaller than that of $\ma^{trig}_\theta(C')$, hence not smaller than that for generic $\ma^{trig}_\theta(z_1,\ldots,z_n)$.

\end{proof}

Since we have defined a flat filtration on the family of limit trigonometric subalgebras, we obtain maps from $ \overline M_{n+2} $ to Grassmannians $ \Gr(r_k, (U \fg^{\otimes n})^{(k)}$) for each $ k$.  In particular, we will consider the quadratic part.

\begin{thm} \label{thm:trigiso}
    The map \begin{align*} 
    \overline M_{n+2} &\rightarrow \Gr(2n, (U \fg^{\otimes n})^{(2)}) \\
    C &\mapsto \ma^{trig}_\theta(C) \cap (U \fg^{\otimes n})^{(2)}
    \end{align*} is an isomorphism onto its image.
\end{thm}

In order to prove this theorem, we will make use of Theorem \ref{th:AFV}, due to Aguirre-Felder-Veselov.   For this purpose, define a Lie algebra map 
    $ \gamma_\theta^1: \fs_{n+1} \to (U\fg^{\otimes n})^\fh$
    given on generators by
    \begin{align*} t_{ij} &\mapsto \Omega^{(ij)}\\
    t_{0i} &\mapsto - \sum_{j=1}^n \Omega^{(ij)}_- + \theta^{(i)} \end{align*} 
In fact, $ \gamma_\theta^1 = \psi_\theta \circ \gamma $ where $ \gamma : \fs_{n+1}^1 \rightarrow (U \fg^{\otimes n})^\fg$ is defined in (\ref{eq:gamma}), and so it is clear that it is a Lie algebra map.

\begin{lem} \label{le:injective}
\begin{enumerate}
    \item  The restriction of $\gamma_\theta^1 $ to $ \fs_{n+1}^1$ is injective.
    \item For $ z_1, \dots, z_n \in \C^\times$ distinct, 
    $$\gamma_\theta^1(h_i(0,z_1, \dots, z_n)) = H_{i,\theta}^{trig}(z_1, \dots, z_n)$$
    \item For any $ C \in \overline M_{n+2}$, we have
    $$
    \ma_\theta^{trig}(C) \cap (U \fg^{\otimes n})^{(2)} = \gamma_\theta(C) \oplus \spann (  \omega^{(i)} : i = 1, \dots, n)
    $$
\end{enumerate}    
\end{lem}

\begin{proof}
    In order to prove the injectivity, we will show that $ \{ \gamma_\theta(t_{ij}) : 0 \le i < j \le n \}$ is linearly independent.
    
    First, we note that $ \gamma_\theta^1(\fs_{n+1}^1) $ lies in the second filtered piece $ (U \fg^{\otimes n})^{(2)}$ and we consider the composition
    $$ \fs_{n+1}^1 \xrightarrow{\gamma_\theta^1} (U \fg^{\otimes n})^{(2)} \rightarrow S^2 (\fg^{\oplus n})  $$
Next, fix $ 1 \le i<j \le n$ consider the projection $ S^2(\fg^{\oplus n}) \rightarrow \fg^{\otimes \{i,j\}}$.  The only basis vectors of $ \fs_{n+1}^1$ which map to non-zero elements in $\fg^{\otimes \{i,j\}} $ are $ t_{ij}, t_{0i}, t_{0j} $ which are sent to 
$$
\Omega^{(ij)}, \quad \Omega^{(ij)}_+, \quad \Omega^{(ij)}_-
$$
As these are linearly independent, if there is a linear combination of the vectors $ \{\gamma_\theta^1(t_{ab}) : 0 \le a < b \le n \} $ which equals 0, then the coefficients of $ \gamma_\theta^1(t_{ij}), \gamma_\theta^1(t_{0i}), \gamma_\theta^1(t_{0j}) $ must all be 0.  Since $ i, j$ were arbitrary, this shows the desired linear independence. 

The second statement follows from the computation of quadratic elements in trigonometric Gaudin algebras (see the discussion before Theorem \ref{size}).
    For the third statement, we know that both sides of this equation are well-defined morphisms $ \overline M_{n+2} \rightarrow \Gr(2n, (U \fg^{\otimes n})^{(2)})$.  Thus, it suffices to check their equality on  $ M_{n+2}$ where it follows from the second statement.  (Or we can appeal to Lemma \ref{le:gamma} and use composition with $ \psi_\theta$.)
\end{proof}

\begin{proof}[Proof of Theorem \ref{thm:trigiso}]
    From the previous lemma, the map  $ \overline M_{n+2} \rightarrow \Gr(2n, (U \fg^{\otimes n})^{(2)}) $ factors as
    $$ \overline M_{n+2} \rightarrow \Gr(n, \fs_{n+1}^1) \rightarrow \Gr(2n, (U \fg^{\otimes n})^{(2)}) $$
    where the second arrow is given by
    $$ Q\mapsto \gamma_\theta^1(Q) \oplus \spann ( \omega^{(i)} : i = 1, \dots, n)$$
    Since $ \gamma_\theta^1$ is injective, this second arrow is a closed embedding.
    As the first arrow is an isomorphism, we conclude that the composition is an isomorphism onto its image.
\end{proof}

\subsection{Case $n=1$}
In the case $n=1$ the family of trigonometric Gaudin $\ma_{\theta}^{trig}(z) \subset U\fg$ subalgebras does not depend on $z \in \bc^{\times}$. Note that the transcendence degree of these subalgebras is the same as that for shift of argument subalgebras, i.e. inhomogeneous Gaudin subalgebras for $n=1$. However, these algebras have different Hilbert-Poincar\'e series. Namely, in the inhomogeneous case, for each degree $ d$ central generator, we have generators of $ \ma_\chi$ of degrees $1,2,\ldots,d$. On the other hand, in the trigonometric case, for each degree $d$ central generator, we have $d$ generators of $ \ma_\theta^{trig}$, all of degree $d$. For example for $\fg = \fsl_n$ there are two generators of degree 2, three generators of degree 3, $\ldots$, $n$ generators of degree $n$.
Due to Lemma \ref{cor1} we see that
$\gr \ma^{trig}_{\theta} \subset S(\fg)$ does not depend on $\theta$. Following \cite{py}, one can describe this subalgebra (plus Cartan elements) as the Poisson-commutative subalgebra arising from the decomposition $\fg = \fn_+ \oplus \fb_-$.

\section{Inhomogeneous Gaudin model}

\subsection{Inhomogeneous Gaudin subalgebras}
\label{inhom}
We now recall the inhomogeneous Gaudin model, following the work of the third author \cite{r2}.

To begin, we define a Lie algebra morphism $$ \ev_\infty : t^{-1} \fg[t^{-1}] \rightarrow t^{-1} \fg[t^{-1}] / t^{-2} \fg[t^{-1}] \cong \fg_{ab}$$ where $ \fg_{ab} $ denotes the vector space $ \fg$ with the trivial Lie bracket.  We have $ \ev_\infty(x[-1]) = x$ for $ x \in \fg $ and $ \ev_\infty(x[-m]) = 0 $ for $ m > 1$.  This extends to an algebra homomorphism $\ev_\infty: U(t^{-1} \fg[t^{-1}]) \rightarrow U \fg_{ab} = S(\fg)$. 

The following result follows from Theorem \ref{th:A-invariance-der}(3).
\begin{lem}
For $ l = 1, \dots, r$, we have $ev_\infty(S_l)=\Phi_l$.
\end{lem}

Let $ z_1, \dots, z_n \in \C^{\times}$  such that $z_i \not = z_j$ for $i \not = j$.  We have the algebra homomorphism
$$ \ev_{\infty,z_1, \ldots, z_n} \circ \Delta^{(n)}: U(t^{-1} \fg[t^{-1}]) \to S(\fg) \otimes U\fg^{\otimes n}$$
\begin{defn}
    We define the \emph{semi-universal inhomogeneous Gaudin algebra} $ \ma(\infty, z_1, \dots, z_n) \subset S(\fg) \otimes U\fg^{\otimes n} $ to be the image of $ \ma $ under $\ev_{\infty,z_1, \ldots, z_n} \circ \Delta^{(n)}$.
\end{defn}

Now, as before, let $u$ be a variable and consider the algebra homomorphism
$$\ev_{\infty,u-z_1, \ldots, u-z_n} \circ \Delta^{(n)}: U(t^{-1} \fg[t^{-1}]) \to S(\fg) \otimes U\fg^{\otimes n}(u)$$

For $ l = 1, \dots, r$, consider the rational $S(\fg) \otimes U\fg^{\otimes n}$-valued function
$$S_l(u; \infty, z_1, \ldots, z_n) = \ev_{\infty, u-z_1, \ldots, u-z_n} \circ \Delta(S_l) \in S(\fg) \otimes U \fg^{\otimes n}(u)$$
Note that this function has poles of order $d_l$ at $z_1, \ldots, z_n$.

As before, we form the Laurent expansions
$$ S_l(u; \infty, z_1, \dots, z_n) = \sum_{m=1}^{d_l} S_{l,i}^{\infty, m} (u - z_i)^{-m} + S(\fg) \otimes U \fg^{\otimes n}[[z_i - u]] $$
where $ S_{l,i}^{\infty, m} \in S(\fg) \otimes U \fg^{\otimes n}$, $ m = 1, \dots, d_l$, and $ i = 1, \dots, n$.


\begin{prop}\label{m2} \cite{r2, cfr}
\begin{enumerate}
 \item The subalgebra $\ma(\infty, z_1, \ldots, z_n)$ is invariant under additive translations of the parameters, i.e. for any $c \in \bc$ 
    $$\ma(\infty,z_1, \ldots, z_n) = \ma(\infty, z_1 +c, \ldots, z_n +c).$$ 
    \item The subalgebra $\ma(\infty, z_1, \ldots, z_n)$ is a polynomial algebra with
    $n(r + p) + r$ generators.
    One possible choice of free generators is $S_{l,i}^{\infty, m}$ for $ l = 1, \dots, r$, $ m = 1, \dots, d_l$ and $ i = 1, \dots, n$ along with $ \Phi_l^{(0)} \in S(\fg) \otimes U \fg^{\otimes n}$ which equals the value of $ S_l(u; \infty, z_1, \dots, z_n) $ at $ u = \infty$.
    \item It is a maximal transcendence degree subalgebra of $(S(\fg) \otimes U\fg^{\otimes n})^{\fg}$.
     \item The Hilbert-Poincar\'e series does not depend on $z_1, \ldots, z_n$ and is 
    $$\prod_{l=1}^{r} \frac{1}{(1-t^{d_l})^{nd_l+1}}.$$
    \end{enumerate}
\end{prop}
\begin{proof} The only statement that is not contained in \cite{r2} is that generators are algebraically independent. But the proof of Proposition 1 of \cite{cfr} shows this, as the proof is the same as for ordinary homogeneous Gaudin subalgebras.     
\end{proof}

\begin{defn}
The inhomogeneous Gaudin subalgebra $\ma_{\chi}(z_1, \ldots, z_n)$ is the image of the subalgebra  $\mathcal{A}(\infty,z_1, \ldots, z_n) $ under the map $\chi \otimes id$.
\end{defn}
Below we only consider subalgebras where $\chi \in \fh$.
We summarize known properties of inhomogeneous Gaudin subalgebras in the following Theorem. 
\begin{thm} \cite{hkrw, r2} \label{th:inhomproperties}
\begin{enumerate}
    \item The subalgebra $\ma_{\chi}(z_1, \ldots, z_n)$ is invariant under  additive translations, i.e. for any $c \in \bc$
        $$\ma_{\chi}(z_1, \ldots, z_n) = \ma_{\chi}(z_1 +c, \ldots, z_n +c).$$
        Also for any $c \in \bc^{\times}$ we have
        $$\ma_{\chi}(z_1, \ldots, z_n) = \ma_{c^{-1} \chi}(c z_1, \ldots, c z_n  )$$

    \item  For $\chi \in \fh$, the subalgebra $\ma_{\chi}(z_1, \ldots, z_n)$ is a free commutative algebra. If $\chi \in \fh^{reg}$ then one possible choice of free generators is the principal part coefficients in the Laurent expansion of $ (\chi \otimes id)(S_l(u;z_1, \ldots, z_n)) $ at the points $z_1, \ldots, z_{n}$.    
    \item For $\chi \in \fh$, the subalgebra $\ma_{\chi}(z_1, \ldots, z_n)$ is a  maximal commutative and of maximal transcendence degree subalgebra of $(U\fg^{\otimes n})^{\fz_{\fg}(\chi)}$. The transcendence degree is equal to
$$ n (r+p) - \dim \fz'_{\fg}(\chi) + \rk \fz'_{\fg}(\chi).$$
In particular, for $\chi \in \fh^{reg}$, the transcendence degree is equal to
 $n (r+p)$.
    
    \item For $\chi \in \fh^{reg}$ the Hilbert-Poincar\'e series does not depend on $z_1, \ldots, z_n$ and equals
     $$\prod_{l=1}^{r} \frac{1}{(1-t^{d_l})^{(n-1)d_l+1}} \cdot \prod_{l=1}^{r} \prod_{j=1}^{d_l-1} \frac{1}{1-t^{j}}$$
\end{enumerate}

\end{thm}

Note that residues of $S_1(u; z_1, \ldots, z_n)$ at $u = z_i$ give rise to $n$ quadratic inhomogeneous Gaudin Hamiltonians:
$$H_{i,\chi} = \sum_{j\ne i} \frac{\Omega^{(ij)}}{z_i - z_j} + \chi^{(i)} $$

There is another set of quadratic Hamiltonians called dynamical Hamiltonians. They arise in the span of
residues of $u^{d_l - 3} (\chi \otimes id)(S_l(u;z_1, \ldots, z_n)), l = 1, \ldots, r$ at $u=\infty$. For $\chi \in \fh^{reg}$ they have the form:

$$G_i = \sum_{\alpha>0} \frac{ \alpha(h_i) }{\alpha(\chi)} \Delta^n(x_{\alpha} x^{\alpha}) + \sum_j z_j h_i^{(j)}$$

Here $\alpha$ is an arbitrary positive root, $x_{\alpha}$ is the corresponding element in the Lie algebra $\fg$, and  $\{h_i\}_{i=1}^r$ is a basis of $\fh$.

\begin{rem}
In \cite{kmr} there is a construction of universal inhomogeneous Gaudin subalgebras i.e. for any $\chi \in \fg$ there is a subalgebra $\ma_{\chi} \subset U(\fg[t])$ such that $\ev_{z_1, \ldots, z_n}(\ma_{-\chi}) = \ma_{\chi}(z_1, \ldots, z_n)$.
\end{rem}

\subsection{Compactification of the family of inhomogeneous Gaudin subalgebras}
\label{secinhom}

We will now study the compactification of the family of inhomogeneous Gaudin subalgebras with fixed $ \chi \in \fg^{reg}$. This family will be parametrized by the space $ \overline F_n $ of cactus flower curves from section \ref{se:cactusflower}.  Recall that $ \overline F_n $ is a compatification of $ F_n = (\C^n \setminus \Delta) / \C$.  For $ (z_1, \dots, z_n) \in F_n $, we have the inhomogeneous Gaudin subalgebra $ \ma_\chi(z_1, \dots, z_n)$.

We begin by defining an algebra $ \ma(C) $ for each point $ C \in \overline F_n$.  

First, we begin with the special case where $ C \in \widetilde M_{n+1}$; this corresponds to cactus flower curves where the distinguished point lies on a unique component.  We identify this component with $ \CP^1$ in such a way that the distinguished point becomes $ \infty $.  We let $ w_1, \dots, w_m \in \CP^1$ be the other special points on this component.  For $ k = 1, \dots, m$, let $ C_k$ be the (possible reducible) nodal curve attached at $ w_k $, and let $ B_k \subset \{1, \dots, n\}$ be the labels of its marked points.  (As before, if $ w_k = z_i$ is a marked point, then $ C_k$ is empty, and $ B_k =\{i \}$.)  In other words, we apply Proposition 4.7 from \cite{iklpr} and write $$ C = ((C_1, \dots, C_m), (w_1, \dots, w_m)) \in \overline M_{B_1 + 1} \times \cdots \times \overline M_{B_m +1} \times F_m $$

The algebra $ \ma(C)$ is built out of two pieces: a diagonally embedded inhomogeneous Gaudin subalgebra and a tensor product of (possibly limit) Gaudin subalgebras coming from the curves $ C_k$.

As in section \ref{se:compactGaudin}, we have the subalgebras $ \ma(C_k) \subset (U\fg^{\otimes B_k})^\fg $ and we collect together these index sets and consider $ \otimes_{k=1}^m \ma(C_k)$ as a subalgebra of $ (U \fg^{\otimes n})^{\fg}$.  

Let $ \Delta^\CB : U \fg^{\otimes m} \rightarrow U \fg^{\otimes n}$ be the diagonal embedding and $ Z = \Delta^\CB(Z(U\fg)^{\otimes m})$.

We define
$$\ma_\chi(C) = \Delta^\CB(\ma_\chi(w_1, \dots, w_m)) \otimes_Z \bigotimes_{k=1}^m \ma(C_k) \subset U \fg^{\otimes n} $$

\begin{ex}
    Suppose that $ m = 1 $, so that $ C_1 \in \overline M_{n+1}$.  This corresponds to the case where $ C$ lies in the zero section $ \overline M_{n+1} \subset \widetilde M_{n+1} $ (and essentially equals $ C_1$).  In this case, we have $ \ma_\chi(C) = \Delta(\ma_\chi) \otimes_Z \ma(C_1) $, the tensor product of a diagonally embedded shift of argument subalgebra with a (possibly limit) homogeneous Gaudin subalgebra. 
\end{ex}

Now, we turn to the case where $C $ is a flower cactus curves with multiple petals.  Then we write $ C = C_1 \cup \dots \cup C_m$, where each $ C_k \in \widetilde M_{S_k+1}$ has marked points labelled $S_k$, and they are all joined at their distinguished points.  Then each $ \ma_\chi(C_k) \in U \fg^{\otimes S_k}$ is defined as above and we define
$$ \ma_\chi(C) = \bigotimes_{k=1}^m \ma_\chi(C_k) \in U \fg^{\otimes n}$$
See Figure \ref{fig:InhomGaudin} for an example.

\begin{figure} 
	\includegraphics[trim=0 60 0 40, clip,width=\textwidth]{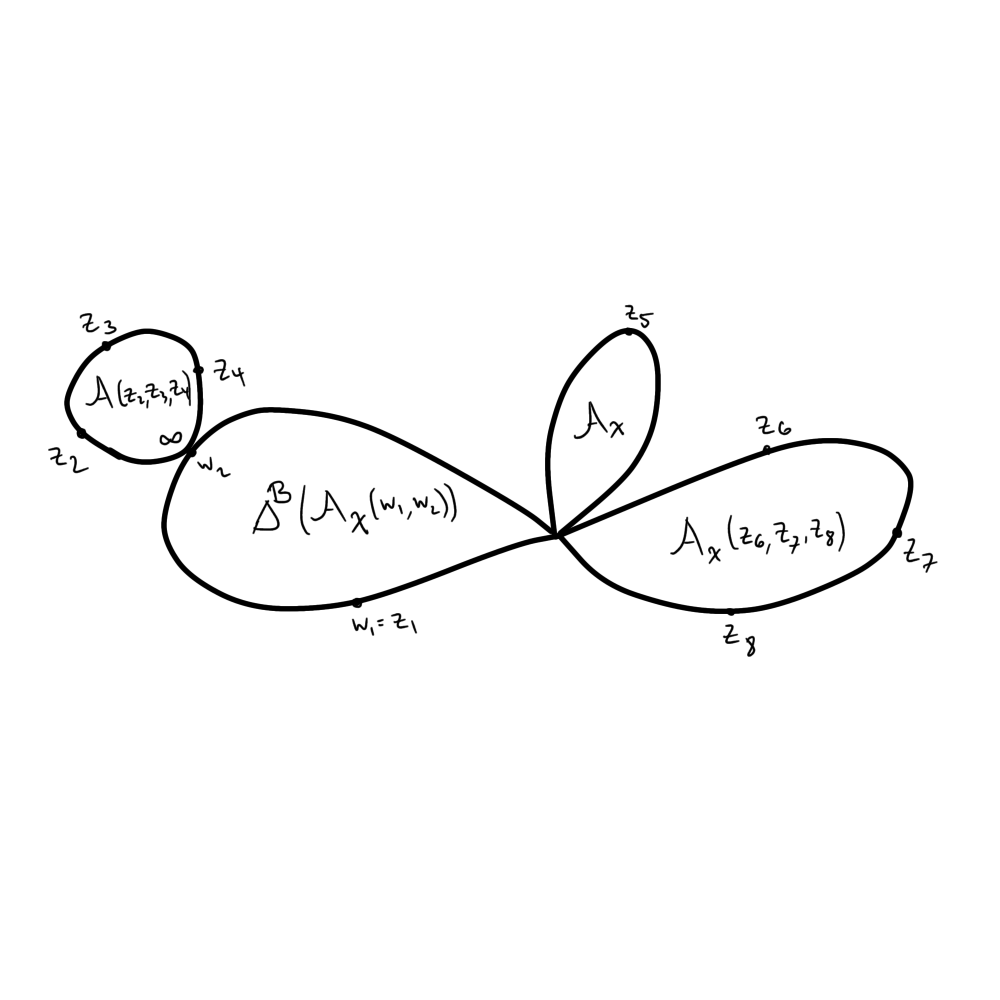}
	\caption{This is a curve $ C \in \overline F_8$.  We have $n = 8$, $m= 3 $, and $ \CS = (\{1,2,3,4\}, \{5\},  \{6,7,8\})$.  The curve $ C_1 \in F_4$ has $ \CB = (\{1\}, \{2,3,4\})$.  Each component contributes the algebra shown and the whole curve gives $ \ma_\chi(C) = \Delta^{\CB}(\ma_\chi(w_1, w_2))\ma(z_2, z_3, z_4)  \otimes \ma_\chi \otimes \ma_\chi(z_6, z_7,z_8) $. } \label{fig:InhomGaudin}
\end{figure}

\begin{ex} \label{eg:flower}
    Suppose that $ m = n $.  Then $ C $ is the maximal flower point of $ \overline F_n$.  Each $ C_k \in \widetilde M_{1+1}$, which is a single point, and thus 
    $\ma_\chi(C) = \ma_\chi^{\otimes n} $ is the tensor product of shift of argument subalgebras. 

    Note that the Hilbert-Poincar\'e series of this algebra with respect to the PBW filtration is
$$ \left(\prod_{l=1}^r \prod_{j=1}^{d_l} \frac{1}{1-t^j}\right)^n$$
which differs from the series for the non-limit inhomogeneous Gaudin subalgebras.  
\end{ex}



\begin{thm} \label{compinhom}
    This defines a family of subalgebras faithfully parametrized by $ \overline F_{n+1}$.  This family  comes with compatible flat filtrations and is the compactification of the family of inhomogeneous Gaudin subalgebras.  All the subalgebras in the family are polynomial algebras.
\end{thm}

As we saw in Example \ref{eg:flower}, the Hilbert-Poincar\'e series of the naive filtration (given by intersecting with the PBW filtration in $ U \fg^{\otimes n}$) is not constant in this family.  Thus, we will need to be more careful when we define the filtration mentioned in this theorem.


We cannot prove Theorem \ref{compinhom} directly as in the case of trigonometric Gaudin subalgebras because we do not have a compactified parameter space for the subalgebras $\ma(\infty, z_1, \ldots, z_n)$.
We will prove Theorem \ref{compinhom} below, see Corollary~\ref{co:compinhom}.

\subsection{Conjectural closure for arbitrary $\chi \in \fh$}

A natural question is to describe the closure of the parameter space for subalgebras $\ma_\chi(z_1, \ldots, z_n)$ with fixed non-regular non-zero $\chi \in \fh$. These subalgebras are smaller than those for regular $ \chi$.  They are maximal commutative subalgebras in the subalgebra of invariants $(U\fg^{\otimes n})^{\fz_\fg(\chi)}$ with respect to the centralizer of $\chi$. Unlike the case of regular $\chi$, now in the limit corresponding to the flower point of $\overline{F}_n$, i.e. the limit  $\lim\limits_{\varepsilon\to0}\ma_\chi(\varepsilon^{-1}z_1,\ldots\varepsilon^{-1}z_n)$, we get a larger subalgebra than $\ma_\chi^{\otimes n}$ because this subalgebra has smaller transcendence degree than $\ma_\chi(z_1, \ldots, z_n)$. The example of subregular $\chi$ considered in \cite[Proposition 10.17]{hkrw} shows that we get the product of $\ma_\chi^{\otimes n}$ with the homogeneous Gaudin subalgebra $\ma(z_1,\ldots,z_n)\subset U\fz_\fg(\chi)^{\otimes n}$. Note that this homogeneous Gaudin subalgebra does depend on $z_1,\ldots z_n$, which means that the actual parameter space should be a non-trivial resolution of $\overline{F}_n$, in particular having $\overline{M}_{n+1}$ as the fiber over the flower point of $\overline{F}_n$.

We expect that the corresponding closure of the $(z_1, \ldots, z_n)$ parameter space is the Mau-Woodward compactification $\mathcal{Q}_n$ from \cite{mw} (see also \cite[\S 5]{iklpr}). 
Set-theoretically, Mau-Woodward space is the set of pairs
$(X,Y), X \in \overline F_n, Y \in \overline M_{m+1}$ where $m$ is the number of petals in $X$ (so the petals of $ X $ are attached not to a single central point but to the marked points of the additional nodal genus 0 stable curve). 


Let $\mathcal{B}$ be the partition of $\{1,\ldots,n\}$ into $m$ parts according to the petals of $X$. Let $\fg_1=\fz_\fg(\chi)$ be the centralizer of $\chi$ in $\fg$. We expect the following description of the extended family of subalgebras $\ma_\chi(z_1,\ldots,z_n)$:
\begin{conj}\label{co:MW}
    There is a family of subalgebras $\ma_\chi(X,Y)$ parameterized by $\mathcal{Q}_n$ that extends $\ma_\chi(z_1,\ldots,z_n)$. The subalgebra corresponding to the point $(X,Y)\in\mathcal{Q}_n$  is the product  $\ma_\chi(X)\cdot \Delta^{\mathcal{B}}\ma^{\fg_1}(Y)$. Here $\ma_\chi(X)$ is defined as the product of $\ma_\chi(X_i)$ over all petals in $X$ and $\ma^{\fg_1}(Y)$ is the Gaudin subalgebra in $U\fg_1^{\otimes m}$.
\end{conj}

Moreover, we can use the above description of limits $\ma_\chi(z_1,\ldots,z_n)$ with not necessarily regular $\chi$ as a building block for the natural compactification of the whole parameter space of all possible $(z_1, \ldots, z_n, \chi)$.  First we note that the open part (corresponding to non-limit inhomogeneous Gaudin algebras) is $$ F_{n,\fg} := ((\C^n \setminus \Delta) / \C) \times \fh^{reg} / \Cx $$ 
where $ \Cx $ acts by scaling and inverse scaling on the two factors (see Theorem \ref{th:inhomproperties}).

Now, let $\chi(\varepsilon)=\chi_0+\varepsilon\chi_1$ with singular $\chi_0$ and consider the limit $$\lim\limits_{\varepsilon\to0}\ma_{\chi(\varepsilon)}(\varepsilon^{-1}z_1, \ldots, \varepsilon^{-1}z_n)$$ Then, again referring to the evidence from \cite[Proposition 10.17]{hkrw}, the limit subalgebra is the product of $\ma_{\chi_0}^{\otimes n}\subset (U\fg^{\fg_1})^{\otimes n}$ and $\ma_{\chi_1}\subset U\fg_1^{\otimes n}$ where $\fg_1=\fz_\fg(\chi_0)$. On the other hand, according to \cite{hkrw}, the natural closure of the parameter space of subalgebras $\ma_\chi\subset U\fg$ is the De Concini-Procesi wonderful compactification $M_{\fg}$ of $\mathbb{P}(\fh)^{reg}$.

Recall that a point of $M_\fg$ is a sequence $\chi_0,\chi_1,\ldots,\chi_l\in\fh$ such that 
\begin{enumerate}
\item $\chi_i\in\fg_i:=\mathfrak{z}_{\fg}(\chi_0,\ldots,\chi_{i-1})'$ (i.e. $\chi_i$ lies in the commutator subalgebra of the centralizer of all the previous elements);
\item $\mathfrak{z}_{\fg}(\chi_0,\ldots,\chi_{l})'=0$;
\item each of the $\chi_i$ is taken up to proportionality.
\end{enumerate}

Associated to such a point, we consider the $l$-times iterated cactus flower space.  A point in this space is a sequence
$(X_0, X_1, \ldots, X_l)$ where $X_0 \in \overline F_n, X_1 \in  \overline F_{k_1},  \dots, X_l \in \overline F_{k_l}$, where $k_i$ is the number of petals in $X_{i-1}$. It is natural to picture the elements of the iterated cactus flower space as follows: first, draw the cactus flower curve $X_l$ ; second, attach the petals  of $X_{l-1}$ to the marked points of $X_l$; third, attach the petals of $X_{l-2}$ to the marked points of $X_{l-1}$, and so on (see Figure \ref{fig:Iterated}).  

\begin{figure} \label{fig:Iterated}
	\includegraphics[trim=10 45 10 45, clip,width=\textwidth]{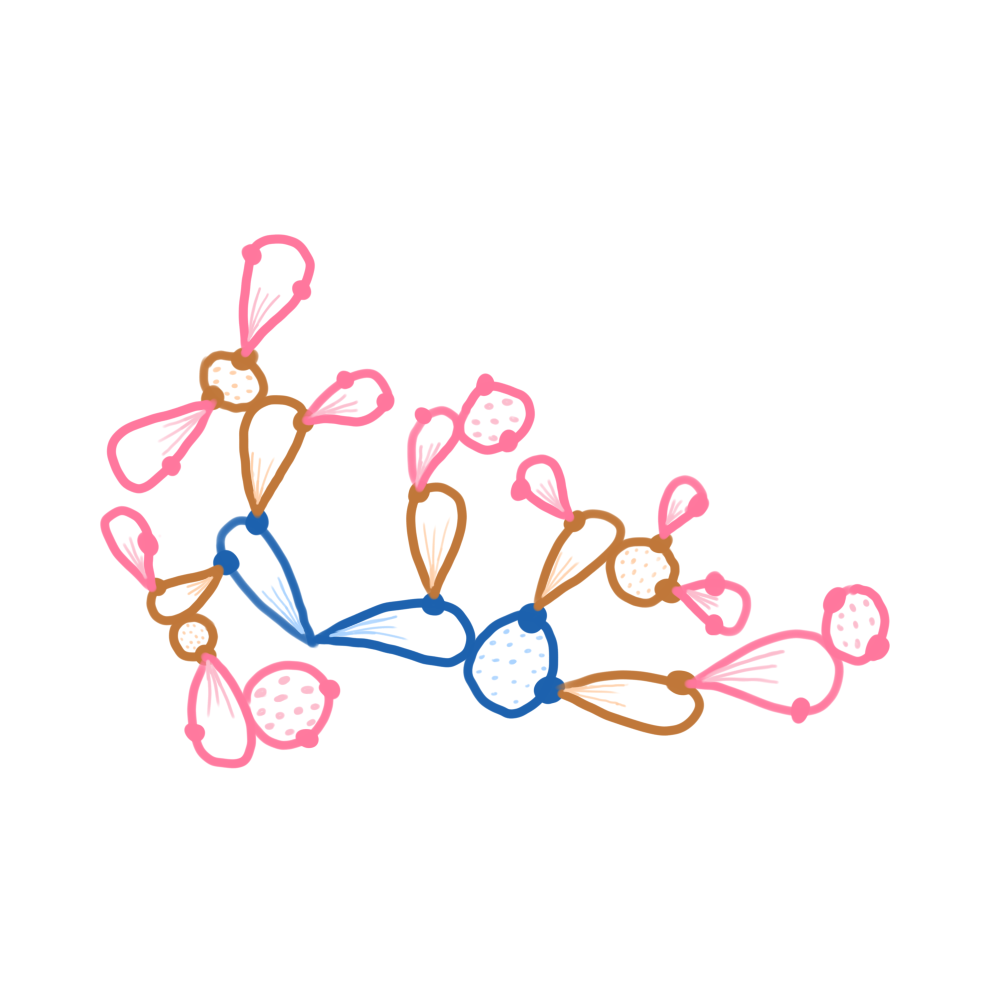}
	\caption{An iterated cactus flower.  This example represents $ (X_0, X_1, X_2) $ where $ X_2 \in \overline F_5 $ is drawn in blue, $ X_1 \in \overline F_{10} $ is drawn in brown, and $ X_0 \in \overline F_{19} $ is drawn in pink.}
\end{figure}

To a point $(\chi_0,\chi_1,\ldots,\chi_l)\in M_\fg$ and an element $(X_0, X_1, \ldots, X_l)$ in the iterated Mau-Woodward space, we assign the following data:
\begin{enumerate}
    \item for any $i=0,\ldots,l$, the subalgebra $\ma_{\chi_i}(X_i)\subset U\fg_i^{\otimes k_i}$.
    \item for any $i=1,\ldots,l$, the partition $\mathcal{B}_i$ of the set $\{1,\ldots,n\}$ according to the petals of $X_{i-1}$. 
\end{enumerate}

So, combining this with all the limits described above, we get the following conjectural description of the compactified parameter space $ \overline F_{n,\fg} $ for these subalgebras. 
\begin{conj}\begin{enumerate}
    \item There is a projective variety $ \overline F_{n,\fg}$ that maps onto $M_\fg$ with the fiber over $(\chi_0,\chi_1,\ldots,\chi_l)\in M_\fg$ being the $l$-times iterated cactus flower space.
    \item The variety $ \overline F_{n, \fg} $ is the compactified parameter space for the family $\ma_\chi(\uz)\subset U\fg^{\otimes n}$, where $ \chi, \uz $ are both allowed to vary.
    \item The subalgebra $\ma_{(\chi_0,\chi_1,\ldots,\chi_l)}((X_0, X_1, \ldots, X_l))$ corresponding to a point $$((\chi_0,\chi_1,\ldots,\chi_l),(X_0, X_1, \ldots, X_l))\in \overline F_n(\fg)$$ is the product of $\ma_{\chi_0}(X_0)\subset U\fg^{\otimes n}$ and $\Delta^{\mathcal{B}_i}\ma_{\chi_i}(X_i)\subset \Delta^{\mathcal{B}_i}U\fg_i^{\otimes k_i} \subset U\fg^{\otimes n}$ for all $i=1,\ldots,l$.
\end{enumerate}
\end{conj}

In particular, the parameter space arising in Conjecture~\ref{co:MW} is the subvariety in $\overline F_{n,\fg}$ given by fixing any regular $\chi_1$ and taking  the codimension $1$ stratum of the moduli space of $X_1$ corresponding to $\overline M_{m+1} \subset \overline F_m$. In particular, the corresponding subalgebras from Conjecture~6.11 are the ones from Conjecture~6.10 times the diagonal $\ma_{\chi_1}$ in $U\fg_1^{\otimes m}$.

For $n=1$ this Conjecture is \cite[Theorem 10.8]{hkrw}. Also, the particular case of this Conjecture with $n=2$ and $\fg_1$ being a root $\fsl_2$ is discussed in \cite[Proposition~10.17]{hkrw}. Moreover, we expect that for $\fg=\fsl_m$ all such spaces $\overline F_{n,\fg}$ are given by the construction of Bottman and Oblomkov \cite{bo} of a moduli space of marked vertical lines in $\mathbb{C}^2$ with $n-1$ vertical lines without marked points and another vertical line with $m$ marked points\footnote{We thank Alexei Oblomkov for pointing out this conjectural relation.}. 

Once the above Conjecture is established, it can be deduced from \cite{hkrw} that for the parameters in the real locus of $\overline F_{n,\fg}$, the corresponding subalgebras act with simple spectrum on any tensor product of finite-dimensional irreducible $\fg$-modules $V(\lambda_1)\otimes\ldots\otimes V(\lambda_n)$. Moreover, the spectrum of an inhomogeneous Gaudin subalgebra in such representation is naturally identified with the tensor product Kashiwara crystal $B(\lambda_1)\otimes\ldots\otimes B(\lambda_n)$. So, considering such spectra altogether, similarly to \cite{hkrw}, we get a covering of the real locus $\overline{\mathcal{Q}}_{n,\fg}(\mathbb{R})$ with the fiber $B(\lambda_1)\otimes\ldots\otimes B(\lambda_n)$, hence an action of the fundamental group of this real locus on the tensor product crystal. We expect that this fundamental group is generated by the inner Weyl-cactus groups and outer virtual cactus groups for all the Levi subalgebras in $\fg$ modulo all universal relations these groups obey altogether.

\section{Compactification of the families of trigonometric and inhomogeneous Gaudin subalgebras.}
\label{sec6}

\subsection{Degeneration of trigonometric to inhomogeneous} \label{se:triginhom}
Recall $ \CF_n := (\BG^n \setminus \Delta) / \BG $, as defined in section \ref{se:deg}.  Our goal in this section is to define a family of subalgebras of $ U \fg^{\otimes n}$ parametrized by $ \mathcal F_n$ and to prove the following result.

\begin{thm} \label{th:familyCFn} 
Assume that $ \chi \in \fh^{reg}$.
There exists a family of subalgebras of $ U \fg^{\otimes n} $ parametrized by $ \mathcal F_n $ such that for $(z_1, \dots, z_n; \varepsilon) \in \mathcal F_n$, we have
$$
\ma_\chi(z_1, \dots, z_n; \varepsilon) := \begin{cases} \ma_{\varepsilon^{-1} \chi}^{trig}(1 - \varepsilon z_1, \dots, 1 - \varepsilon z_n), \ \text{ if $ \varepsilon \ne 0 $} \\
\ma_\chi(z_1, \dots, z_n), \ \text{ if $\varepsilon = 0 $}
\end{cases}
$$
In particular, we have $ \lim\limits_{\varepsilon \to 0} \ma_{\varepsilon^{-1} \chi}^{trig}(1 - \varepsilon z_1, \dots, 1 - \varepsilon z_n) = \ma_\chi(z_1, \dots, z_n)$
\end{thm}

The remainder of this subsection will be devoted to constructing this family.

Consider $ U \fg $ with its usual PBW filtration and let $ U_\hbar \fg = \Rees(U \fg) $.  Equivalently, $U_\hbar \fg$ is the algebra generated by $ x \in \fg$ and $ \hbar$ with the relations $ xy - yx = \hbar[x,y]$.  We write $ U_\varepsilon \fg $ for the specialization of $ U_\hbar \fg $ at $ \hbar = \varepsilon$.  For $ \varepsilon \ne 0 $, we have $ U_\varepsilon \fg \cong U \fg$ via the map $ x \mapsto \varepsilon x$, and we have $ U_0 \fg \cong S (\fg) $.

For $z_1, \dots, z_n \in \C $ non-zero, define a evaluation map by
    \begin{align*}
\ev^\hbar_{z_1, \dots, z_n} : t^{-1} \fg[t] &\rightarrow U_\hbar \fg \otimes U \fg^{\otimes n} \\
 x[-m] &\mapsto \hbar^{m-1} x^{(0)} + \sum_i z_i^{-m} x^{(i)}
\end{align*}

The following result is immediate.
\begin{prop}
    \begin{enumerate}
        \item This is a Lie algebra map (with respect to the commutator on the right hand side) and thus extends to an algebra map $ \ev^\hbar_\uz : U(t^{-1} \fg[t^{-1}]) \rightarrow  U_\hbar \fg \otimes U\fg^{\otimes n}  $.
        \item Let $\varepsilon \in \C$ such that $ 1 - \varepsilon z_i \ne 0 $ for all $ i$.  We can specialize at $ \hbar = \varepsilon$
        $$ U( t^{-1} \fg[t^{-1}]) \xrightarrow{\ev_\uz^\hbar}  U_\hbar \fg \otimes U\fg^{\otimes n} \rightarrow  U_\varepsilon \fg \otimes U\fg^{\otimes n}  $$
        After identifying $ U_0 \fg = S(\fg) $ and $U_\varepsilon \fg = U \fg$, for $\varepsilon \ne 0 $, this map equals $ \ev_{\infty, z_1, \dots, z_n} $ if $ \varepsilon = 0$ and $ \ev_{\varepsilon^{-1}, z_1, \dots,  z_n}$ otherwise.
    \end{enumerate}
\end{prop}

From this proposition, we can define $\mathbfcal A^\hbar(\uz) := \ev^\hbar_\uz(\ma) \subset U_\hbar \fg \otimes U\fg^{\otimes n} $.  

Let $ (z_1, \dots, z_n; \varepsilon) \in \BG^n \setminus \Delta$. We can specialize $ \mathbfcal A^\hbar(\uz) $ to obtain $ \mathbfcal A^\varepsilon(\uz) \in U_\varepsilon \fg \otimes U\fg^{\otimes n} $ which we identify with $ U \fg^{\otimes n+1}$ for $ \varepsilon \ne 0 $ and $ S(\fg) \otimes U \fg^{\otimes n} $ for $ \varepsilon = 0 $.  Each specialized subalgebra is either a homogeneous Gaudin algebra or an inhomogeneous one.  From the previous Proposition, we immediately deduce the following result.

\begin{prop} \label{pr:mathbfcalA} We have
$$
\mathbfcal A^\varepsilon(z_1, \dots, z_n) = \begin{cases} \ma(0, 1 - \varepsilon z_1, \dots, 1 - \varepsilon z_n) = \ma(\varepsilon^{-1},  z_1, \dots, z_n) \text{ if $ \varepsilon \ne 0 $} \\
\ma(\infty; z_1, \dots, z_n) \text{ if $ \varepsilon = 0 $}
\end{cases}$$
In particular, applying what we already know about invariance of trigonometric Gaudin algebras under scaling (Theorem \ref{size}) and inhomogeneous Gaudin algebras under translation (Theorem \ref{th:inhomproperties}), we see that $ \mathbfcal A^\varepsilon(z_1, \dots, z_n) $ depends on $ (z_1, \dots, z_n; \varepsilon) $ only as an element of $ \CF_n = (\BG^n \setminus \Delta) / \BG$.
\end{prop}

So $ \mathbfcal A^\varepsilon(\uz)$ is a family of algebras of $ (U_\hbar \fg \otimes U \fg^{\otimes n})^\fg$ parameterized by $ \CF_n$ (over the base ring $ \C[\hbar]$).

Recall that when constructing the trigonometric Gaudin algebras in section \ref{se:deftrig}, we defined an algebra homomorphism using Hamiltonian reduction by $\Delta(\fn_+)$.
\begin{equation} \label{eq:psichi}
\psi : (U \fg \otimes U \fg^{\otimes n})^\fg \rightarrow U(\fh) \otimes U \fg^{\otimes n}
\end{equation}

Define a filtration on $ U \fg \otimes U \fg^{\otimes n}$ by considering the PBW filtration on the first factor, i.e. $ (U \fg \otimes U \fg^{\otimes n})^m = U \fg^m \otimes U \fg^{\otimes n} $.  The Rees algebra of this filtration is $ U_\hbar \fg \otimes U \fg^{\otimes n}$.  On the RHS of (\ref{eq:psichi}) we also use the PBW filtration on the first factor.

\begin{prop} \label{pr:grpsi2}
    The map $ \psi $ preserves these filtrations.  The resulting map $ \gr \psi :  (S(\fg) \otimes U \fg^{\otimes n})^\fg \rightarrow  S(\fh) \otimes U \fg^{\otimes n} $ comes from the projection map $ S(\fg) \rightarrow S(\fh)$. 
\end{prop}

\begin{proof}
This follows immediately from the description of $ \psi $ given in Proposition \ref{pr:psiab}.
\end{proof}

Since $ \psi $ preserves the filtrations, we apply the Rees functor to (\ref{eq:psichi}) and obtain
$$
\Rees \psi : (U_\hbar \fg \otimes U \fg^{\otimes n})^\fg \rightarrow  U_\hbar \fh \otimes U \fg^{\otimes n} 
$$

As before, we may specialize the map at any $ \varepsilon$.  From the general principles, when we specialize at $ \varepsilon \ne 0$, this recovers $ \psi$, and when we specialize at $ \varepsilon = 0$, we get $ \gr \psi$ as described in the previous proposition. 

Let $ \chi \in \fh $.  We can define a $\C[\hbar]$-algebra morphism $ \chi:  U_\hbar \fh \rightarrow \C[\hbar] $ given by $ x \mapsto \chi(x)$ for $ x \in \fh$.   We will now study the composition 
$$\chi \circ \Rees \psi : (U_\hbar \fg \otimes U \fg^{\otimes n})^\fg \rightarrow  \C[\hbar] \otimes U \fg^{\otimes n} $$

\begin{prop} \label{pr:psichi}
    For any $ \varepsilon \ne 0 $ we have a commutative diagram
$$
\begin{tikzcd}
    (U_\varepsilon \fg \otimes U \fg^{\otimes n})^\fg \ar[r,"\Rees \psi"] \ar[d] & U_\varepsilon \fh \otimes U \fg^{\otimes n} \ar[r,"\chi"] \ar[d] &  \C[\hbar]/ (\hbar - \varepsilon) \otimes U \fg^{\otimes n} \ar[d] \\
    (U \fg \otimes U \fg^{\otimes n})^\fg \ar[r,"\psi"] & U\fh \otimes U \fg^{\otimes n} \ar[r,"\varepsilon^{-1} \chi"] &  U \fg^{\otimes n}
\end{tikzcd}
$$
    where the first two vertical arrows are the canonical isomorphisms from the Rees construction and the third is the obvious isomorphism.
\end{prop}

\begin{proof}
  The first square commutes by the generalities of the Rees construction.  The second commutes because if $ x \in \fh$, then $ x^{(0)}$ is mapped to $ \chi(x) $ by going right and then down and $ \varepsilon^{-1} \chi(\varepsilon x)$ by going down and then right.  
\end{proof}

Let $ Z_\hbar(\chi) \subset U_\hbar \fg$ be the central ideal corresponding to $\chi$ under the Harish-Chandra isomorphism.  More precisely, it is the ideal generated by $ \widetilde \Phi_l - \Phi_l(\chi + \hbar \rho) $ for $ l = 1, \dots, r $.

\begin{lem} \label{le:psiinjective}
    The map $ \chi \circ \Rees \psi$ factors as
    $$(U_\hbar \fg \otimes U \fg^{\otimes n})^\fg \rightarrow (U_\hbar \fg/ Z_\hbar(\chi) \otimes U \fg^{\otimes n})^\fg \rightarrow \C[\hbar] \otimes U \fg^{\otimes n}
    $$ 
    When $ \chi $ is regular, then the second arrow is injective.

\end{lem}

We begin with the following simple result.
\begin{lem}
    Let $ A, B$ be two $ \mathbb N$-graded $ \C[\hbar]$-modules, where $ \hbar$ has degree 1.  Let $ \phi : A \rightarrow B $ be a graded $ \C[\hbar]$-module morphism.  Assume that $ \phi : A /\hbar A \rightarrow B / \hbar B $ is injective and that $ B $ is torsion-free as a $ \C[\hbar]$-module.  Then $ \phi $ is injective.

        Similarly, if $ A, B $ are two $\mathbb N $ filtered vector spaces and $ \phi : A \rightarrow B $ is a filtered linear map such that $ \gr \phi $ is injective, then $ \phi $ is injective.
\end{lem}

\begin{proof}
Suppose $ \phi $ is not injective.  Then there exists $ k \in \mathbb N$ and $ a \in A_k $ such that $ a \ne 0 $ and $ \phi(a) = 0 $; choose such $ a, k $ such that $ k$ is minimal.  Then $ \phi(a) = 0 $ in $ B / \hbar $, so by hypothesis, $ a \in \hbar A$. Thus we can find $ a' \in A_{k-1} $ such that $ a = \hbar a' $.  Since $ \phi(a) = \hbar \phi(a') $ and $ B $ is torsion-free, this implies that $ \phi(a') = 0 $, which contradicts the minimality of $ k$.  

The proof of the second statement is very similar.
\end{proof}

\begin{proof}[Proof of Lemma \ref{le:psiinjective}]
By Lemma \ref{le:thetarho}, we see that $$ \chi \circ (\Rees \psi)(\widetilde \Phi^{(0)}_l) = \Phi_l(\chi + \hbar \rho)$$
Thus, we deduce the desired factoring.

For the injectivity, by the previous lemma, it suffices to show that
\begin{equation} \label{eq:SchiU}
(S \fg/ Z_0(\chi) \otimes U \fg^{\otimes n})^\fg \rightarrow U \fg^{\otimes n}
\end{equation}
is injective, where $ Z_0(\chi) \subset S \fg $ is the ideal of $ p^{-1}(\chi) \subset \fg $, where $p : \fg \rightarrow \fg \sslash G \cong \fh \sslash W $ is the adjoint quotient map. We can easily compute (\ref{eq:SchiU}) using Proposition \ref{pr:psiab}.  In the notation of that proposition, we see that (\ref{eq:SchiU}) is given by
$$
u \mapsto \sum_k \epsilon(a_{k,-}) \epsilon(a_{k,+}) \chi(a_{k,0}) b_k 
$$
since it is the composition of $ \chi $ with the associated graded of $ \psi$ with respect to the PBW filtration on the first tensor factor.

Now, to prove the injectivity of (\ref{eq:SchiU}), it suffices to prove that its associated graded map is injective, with respect to the PBW filtration on $ U \fg^{\otimes n}$.  When we take this associated graded, we reach the algebra homomorphism
$$ 
(S \fg/ Z_0(\chi) \otimes S \fg^{\otimes n})^\fg \rightarrow S \fg^{\otimes n}
$$
dual to the map of schemes given by 
$$ \fg^n \rightarrow p^{-1}(\chi) \times \fg^n \sslash G  \quad   (y_1, \dots, y_n) \mapsto (\chi, y_1, \dots, y_n).$$  

When $ \chi $ is regular this map has dense image, hence the previous algebra map was injective and this completes the proof.
\end{proof}

For the remainder of this section, fix $ \chi \in \fh^{reg}$.  Using $\chi \circ \Rees \psi $, we can transfer the family of subalgebras $ \mathbfcal A$ to $ \C[\hbar] \otimes U \fg^{\otimes n}$.  We define $$ \mathbfcal A_\chi := (\chi \circ \Rees \psi)(\mathbfcal A) \subset \C[\CF_n] \otimes_{\C[\hbar]} \C[\hbar] \otimes U \fg^{\otimes n} = \C[\CF_n] \otimes U \fg^{\otimes n}$$
So we can regard $ \mathbfcal A_\chi$ in two ways:
\begin{enumerate}
    \item as a family of subalgebras of $ \C[\hbar] \otimes U\fg^{\otimes n}$ parametrized by $ \CF_n $, regarded as a family over the base ring $ \C[\hbar]$,
    \item as a family of subalgebras of $ U \fg^{\otimes n}$ parametrized by $ \CF_n$ over the base ring $ \C$.
\end{enumerate}
Using the second interpretation, we are now ready to prove Theorem \ref{th:familyCFn}.

\begin{proof}
Let $ (z_1, \dots, z_n; \varepsilon) \in \mathcal F_n$.  Suppose that $ \varepsilon \ne 0$.  Then Propositions \ref{pr:mathbfcalA} and \ref{pr:psichi}  show that $$ (\chi \circ \Rees \psi)\left(\mathbfcal A^\varepsilon(z_1, \dots, z_n)\right) = \ma_{\varepsilon^{-1} \chi}^{trig}(1 - \varepsilon z_1, \dots, 1 - \varepsilon z_n)$$
as desired.  

On the other hand, if $ \varepsilon = 0 $, then Propositions \ref{pr:mathbfcalA} and \ref{pr:grpsi2} combine to show that
$$
(\chi \circ \Rees \psi)\left(\mathbfcal A^0(z_1, \dots, z_n)\right) = \ma_{\chi}( z_1, \dots, z_n) $$

\end{proof}

\subsection{Quadratic parts of inhomogeneous/trigonometric Gaudin subalgebras}

From now on, we will regard $ \mathbfcal A_\chi $ as a family of subalgebras of $ \C[\hbar] \otimes U\fg^{\otimes n}$ parametrized by $ \CF_n $, regarded as a family over the base ring $ \C[\hbar]$.  Our goal is to study the compactification of this family.  For this purpose, we begin by defining filtrations.

For any $ (\uz, \varepsilon) \in \CF_n$, define a filtration on $ \mathbfcal A^\varepsilon(\uz) $ by intersecting with the PBW filtration, 
$$ \mathbfcal A^\varepsilon(\uz)^m := \mathbfcal A^\varepsilon(\uz) \cap (U_\varepsilon \fg \otimes U \fg^{\otimes n})^m $$ 
and define $ \mathbfcal A_\chi^\varepsilon(\uz)^m := (\chi \circ \Rees \psi)(\mathbfcal A^\varepsilon(\uz)^m) $.  When $ \varepsilon \ne 0 $, this coincides with intersection of $ \mathbfcal A^\varepsilon(\uz) $ with the PBW filtration on $ U \fg^{\otimes n} $ but when $ \varepsilon = 0 $ it does not.

We can study the quadratic parts $ \mathbfcal A^\varepsilon(\uz)^{(2)} $ using the Lie algebra $ \frr_n $ from section \ref{se:CFnfrrn}.  We define a $ \C[\hbar]$ linear map
\begin{gather*}
\gamma^\hbar : \frr_n \rightarrow (U_\hbar \fg \otimes (U \fg)^{\otimes n})^\fg \\
t_{ij} \mapsto \Omega^{(ij)} \quad u_i \mapsto \Omega^{(0i)}
\end{gather*}
and we define 
$$ \gamma_\chi^\hbar := \chi \circ \Rees \psi \circ \gamma^\hbar : \frr_n \rightarrow \C[\hbar] \otimes U \fg^{\otimes n} $$
We have $ \gamma^\hbar_\chi(t_{ij}) = \Omega^{(ij)}$ and $ \gamma^\hbar_\chi(u_i) = \chi^{(i)} + \hbar \sum_j \Omega_-^{(ij)}$.

\begin{lem}
  The maps $ \gamma^\hbar, \gamma^\hbar_\chi $ are both Lie algebra morphisms which restrict to injective maps $ \frr_n^1 \rightarrow (U_\hbar \fg \otimes U \fg^{\otimes n})^\fg $ and $ \frr_n^1 \rightarrow \C[\hbar] \otimes U \fg^{\otimes n}$. 
  \end{lem}

\begin{proof}
    A standard computation shows that $ \gamma^\hbar$ is a Lie algebra morphism, which implies that $ \gamma^\hbar_\chi$ is one as well.

    To see that they restrict to injective maps on the first graded piece, we fix some $ 1 \le i \ne j \le n$, and consider $$ \frr^1_n \xrightarrow{\gamma^\hbar_\chi} \C[\hbar] \otimes U \fg^{\otimes n} \rightarrow \C[\hbar] \otimes U \fg^{\otimes\{i,j\}}
    $$
    This image lies in the second filtered piece with respect to the PBW filtration on $ U \fg^{\otimes \{i,j\}} $, so we can consider the further map to $S^2(\fg^{\oplus \{i,j\}}) $.
    The resulting map is given by
    $$ u_i \mapsto \chi^{(i)} + \hbar \Omega_-^{(ij)} \quad t_{ij} \mapsto \Omega^{(ij)} \quad u_j \mapsto \chi^{(j)} + \hbar \Omega^{(ij)}_+ $$
    As these elements are linearly independent (over $\C[\hbar]$), injectivity follows by the same argument as in Lemma \ref{le:injective}.
\end{proof}

 Recall that for any point $ (\uz;\varepsilon) \in \CF_n $, we have $Q^\varepsilon(\uz) \subset \frr_n^1(\varepsilon)$, from section \ref{se:CFnfrrn}.

\begin{lem}
    For any $ (\uz, \varepsilon) \in \CF_n$, we have \begin{gather*} \mathbfcal A^\varepsilon(\uz)^{(2)} = \gamma^\varepsilon( Q^\varepsilon(\uz)) \oplus \spann(\omega^{(i)} : i = 0, \dots, n) \\ 
    \mathbfcal A^\varepsilon_\chi(\uz)^{(2)} = \gamma^\varepsilon_\chi( Q^\varepsilon(\uz)) \oplus \spann(\omega^{(i)} : i = 1, \dots, n) \end{gather*}
\end{lem}

\begin{proof}
    The first statement follows immediately from Lemma \ref{le:gamma} and the second follows from the first using $ (\chi \circ \Rees \psi)(\widetilde \Phi_1^{(0)}) = \Phi_1(\chi + \varepsilon\rho)$.
\end{proof}

Now, the quadratic parts of these algebras define maps 
\begin{gather*} \CF_n \rightarrow \Gr_{\mathbb A^1}(2n+1, U_\hbar \fg \otimes U \fg^{\otimes n} ) \quad (\uz,\varepsilon) \mapsto \mathbfcal A^\varepsilon(\uz)^{(2)} \\
\CF_n \rightarrow \Gr_{\mathbb A^1}(2n, \C[\hbar] \otimes U \fg^{\otimes n} )
\quad (\uz,\varepsilon) \mapsto \mathbfcal A^\varepsilon_\chi(\uz)^{(2)} 
\end{gather*}

Combining Theorem \ref{th:CFCG} with the previous two Lemmas, we deduce the following corollary.
\begin{corol} \label{co:quad}
    These extend to closed embeddings $$ \overline \CF_n \rightarrow \Gr_{\mathbb A^1}(2n+1, U_\hbar \fg \otimes U \fg^{\otimes n} ) \quad \overline \CF_n \rightarrow \Gr_{\mathbb A^1}(2n, \C[\hbar] \otimes U \fg^{\otimes n} ) $$
\end{corol}

\subsection{The Main Theorem and plan of the proof.}
In section \ref{se:triginhom}, we constructed $ \mathbfcal A $, a family of subalgebras of $ U_\hbar \fg \otimes U \fg^{\otimes n}$ parametrized by $ \CF_n$, and $ \mathbfcal A_\chi$, a family of subalgebras of $ \C[\hbar] \otimes U \fg^{\otimes n}$ also parametrized by $ \CF_n$ (in both cases regarded over the base ring $ \C[\hbar]$).  Recall also that in \cite{iklpr}, we constructed a (relative) compactification $ \overline \CF_n $, see section \ref{se:deg}.

Here is the main theorem of this paper.

\begin{thm} \label{th:algCFn}
    These extend to flat families of subalgebras faithfully parametrized by $ \overline \CF_n$.  Thus $ \overline \CF_n $ is the compactified parameter space for these inhomogeneous/trigonometric subalgebras.
\end{thm}

As in the proof of Theorem \ref{trigcomp}, it suffices to show that the family extends to a flat family parametrized by $ \overline \CF_n$.  Subsequently, the faithfulness of the parametrization follows from Corollary \ref{co:quad}.

Since $ \mathbfcal A_\chi$ is the homomorphic image of $ \mathbfcal A$, it will suffice to construct the extension of $ \mathbfcal A$ and check that both $ \mathbfcal A, \mathbfcal A_\chi$ carry compatible flat filtrations.  More precisely, we will carry out the following plan which is similar to the proof of \cite[Theorem~3.13]{r}. 

\begin{itemize}
    \item[Step~1.]
    For each open affine $ \CW_\tau \subset \overline \CF_n$ (see section \ref{se:Wtau}), we find a collection of algebraically independent generators of $ \mathbfcal A^\varepsilon(z_1, \dots, z_n)$, which extend as regular functions over $ \CW_\tau $. 

    \item[Step~2.]
    Each $ \CW_\tau$ is a union of strata of $ \overline \CF_n$.  These strata are constructed in an ``operadic manner'' from smaller $ \overline \CF_r $ and $ \overline M_{k+1} $.  We prove that the restrictions of our generators to these strata are compatible with the generators of the Gaudin subalgebras over these strata.

    \item[Step~3.]
    By induction on the codimension of a stratum, we prove that our generators are algebraically independent at any boundary point as well, so the family $\mathbfcal{A}^\varepsilon(C)$ is flat.  
    \item[Step~4.]
    We show that  $\chi \circ \Rees \psi $ takes $r$ of the generators to constants and leaves the others algebraically independent. This implies that the family of subalgebras $\mathbfcal{A}^\varepsilon_\chi(C)$ is flat as well.
\end{itemize}

\subsection{Step~1. The generators}

We will need to consider residues of rational functions extensively in this section.  Let $\mathbb{C}(w,w_1,\ldots,w_N)$ be the field of rational functions of $N+1$ variables. Consider residue maps $\Res_{w_i}:\mathbb{C}(w,w_1,\ldots,w_N)\to \mathbb{C}(w_1,\ldots,w_N)$ taking a function $f\in\mathbb{C}(w,w_1,\ldots,w_N) $ (regarded as a one-variable function $f(w)$ over $\mathbb{C}(w_1,\ldots,w_N)$) to the residue of $f(w)$ at $w=w_i$. The following lemma will be quite useful for us.
    
\begin{lem}\label{lem:resudue-sum}
    Let $f\in\mathbb{C}[w,(w-w_1)^{-1},\ldots, (w-w_N)^{-1}]$ and $I\subset \{1,\ldots,N\}$. Then the sum of residues $\sum\limits_{i\in I} \Res_{w_i} f$ belongs to $\mathbb{C}[w_i,(w_i-w_j)^{-1}]_{i\in I,\ j\not\in I}$.
\end{lem}

\begin{proof}
    Let $x=w-w_i$ be the local coordinate at $w=w_i$. Then the Laurent expansion of $f(w)$ at $w_i$ is a polynomial expression of $w=x+w_i$, $\frac{1}{w-w_i}=x^{-1}$ and $$\frac{1}{w-w_j}=\frac{1}{x-(w_j-w_i)}=\frac{(w_i-w_j)^{-1}}{1-(w_i-w_j)^{-1}x}$$ so the residue of $f(w)$ at $w=w_i$ lies in $\mathbb{C}[w_i,(w_i-w_j)^{-1}]_{j\ne i}$. This means that 
    \begin{equation} \label{eq:res1}
    \sum\limits_{i\in I} \Res_{w_i} f(w) \in \mathbb{C}[w_i,(w_i-w_j)^{-1}]_{i\in I,\ j\ne i}
    \end{equation}
    
   On the other hand, the sum of residues of $f(w)$ at all $w_i$ and $\infty$ is $0$, so we have \begin{equation} \label{eq:res2}
   \sum\limits_{i\in I} \Res_{w_i} f(w) =-\Res_{\infty} f(w) -\sum\limits_{j\not\in I} \Res_{w_j} f(w)
   \end{equation}
Similar to (\ref{eq:res1}), we have
$$
   \sum\limits_{j \notin I} \Res_{w_j} f(w) \in \mathbb{C}[w_j,(w_i-w_j)^{-1}]_{j\notin I,\ i\ne j}
$$
Also, the Laurent expansion of $\frac{1}{w-w_k}=\frac{w^{-1}}{1-w_kw^{-1}}$ at $\infty$ lies in $ \C[w_k]((w^{-1})))$.  Thus, the right hand side of (\ref{eq:res2}) lies in $ \mathbb{C}[w_k,(w_i-w_j)^{-1}]_{j\notin I,\ i\ne j}$, so we deduce that $ \sum\limits_{i\in I} \Res_{w_i} f(w)$ lies in $ \C[w_i, (w_i - w_j)^{-1}]_{i \in I, j \notin I} $ as desired.
\end{proof}

Recall the rational functions $ S_l(u; \varepsilon^{-1}, \uz) \in U \fg \otimes U\fg^{\otimes n}(u)
$ for $l = 1, \ldots, r$.  The coefficients in the Laurent expansions of these functions generate $\mathbfcal{A}^{\varepsilon}(\uz)$.  We consider a universal version of these rational functions.  We define
$$ S_l(u; \hbar^{-1},\uz) :=   \ev_{u - \hbar^{-1}, u-z_1, \dots, u-z_n}(S_l) \in \mathbb{C}(u,\hbar,z_1,\ldots,z_n) \otimes_{\mathbb{C}[\hbar]} U_\hbar \fg \otimes U\fg^{\otimes n}$$ 
Note that $ U_\hbar \fg \otimes U\fg^{\otimes n} $ is free as a $ \C[\hbar] $ module.  Fixing an isomorphism $ U_\hbar \fg \otimes U\fg^{\otimes n} \cong \C[\hbar] \otimes U$ (such a $ U $ is non-canonically isomorphic to $ U \fg^{\otimes n+1}$), we can regard $ S_l(u; \hbar^{-1}, \uz) $ as a rational function of $ u, \hbar, z_1, \dots, z_n $ with  values in $U$.

Let $\tau$ be a binary rooted forest.  Let $v$ be an internal vertex of $\tau$.  Let $p$ be the rightmost leaf of the left branch growing from $ v$ and let $q$ be the leftmost leaf on the right branch growing from $v$. We make an affine change of the variable $u$ taking $z_p$ to $0$ and $z_q$ to $1$; the new variable is 
$$w = \frac{u-z_p}{z_q-z_p}.$$

\begin{lem} 
    Under the above coordinate change, 
    $$(z_q-z_p)^{d_l}S_l(u;\hbar^{-1},\uz) =S_l(w; \hbar^{-1}\nu_{qp}, \mu_{pq1}, \ldots, \mu_{pqn})$$
    In particular, it lies in $ \C(\CF_n) \otimes \C(w) \otimes U  $.
\end{lem}

\begin{proof}
Since $ S_l $ has degree $ d_l $ with respect to the loop rotation, by Remark \ref{re:looprotate}, we have
$$
(z_q - z_p)^{d_l} ev_{u -\hbar^{-1}, u-z_1, \dots, u - z_n}(S_l) = ev_{w_0, w_1, \dots, w_n}(S_l)
$$
where
$$
w_0 = (z_q - z_p)^{-1}(u - \hbar^{-1}) = w - \hbar^{-1}\frac{1 - \hbar z_p}{z_q - z_p}$$
and 
$$
w_i = (z_q - z_p)^{-1}(u-z_i) = w - \frac{z_p - z_i}{z_p - z_q}
$$
for $ i = 1, \dots, n$.
\end{proof}

For $ m = 1, \dots, d_l $, define $ s_{l,v,\tau}^m \in \C(\CF_n) \otimes U$ 
by $$ s_{l,v,\tau}^m = \sum_{i \in I} \Res_{\mu_{pqi}} w^{m-1} S_l(w; \hbar^{-1}\nu_{qp}, \mu_{pq1}, \dots, \mu_{pqn})$$
where $ I \subset [n]$ is the set of leaves on the left branch growing from the vertex $ v$.

Now, suppose that $ v $ is a root of $ \tau$. Let $ p $ be the rightmost leaf on the tree growing from $ v $. We make an affine change of the variable $u$ taking $z_p$ to $0$ and preserving $ \hbar^{-1}$; the new variable is 
$$w = \frac{u-z_p}{1 - \hbar z_p}$$

\begin{lem}
      Under the above coordinate change, 
    $$(1 - \hbar z_p)^{d_l}S_l(u;\hbar^{-1},z_1,\ldots,z_n) = S_l(w; \hbar^{-1}, \delta_{1p}, \ldots, \delta_{np})$$
    In particular, it lies in $ U \otimes \C(\CF_n) \otimes \C(w)$.
\end{lem}
\begin{proof}
    Same as the previous lemma.
\end{proof}

For $ m = 1, \dots, d_l $, define $ s_{l,v,\tau}^m \in U \otimes \C(\CF_n)  $ by 
$$ s_{l,v,\tau}^m := \sum_{i \in I} \Res_{\delta_{ip}} w^{m-1} S_l(w; \hbar^{-1}, \delta_{1p}, \dots, \delta_{np})$$
where  $I \subset [n] $ is the set of leaves on the tree growing from the vertex $v $.

In this way, we have defined a collection $ \{ s_{l,v,\tau}^m \}$ of rational functions on $ \CF_n $ with values in $ U $, with $v $ ranging over the non-leaf vertices of $ \tau$, $ l = 1 \dots, r $, and $ m = 1, \dots, d_l$, along with the $ r $ ``constant functions'' $ \widetilde \Phi_l^{(0)}$.
 This is a total of $ n(p+r) + r $ functions, which is the correct number to give generators for the Gaudin subalgebras in $ U \fg^{\otimes n+1}$.

\begin{prop}
    Every function $s_{l,v,\tau}^m $ defined above extends to a regular function on $ \CW_\tau$.
\end{prop}
\begin{proof}
Suppose that $ v $ is an internal vertex.  From Lemma \ref{lem:resudue-sum} we see that $$s_{l,v,\tau}^m \in U \otimes \C[\mu_{pqi},  (\mu_{pqi} - \mu_{pqj})^{-1},  (\mu_{pqi} - \hbar^{-1} \nu_{qp})^{-1}]_{i \in I, j \notin I}$$ 
where $ I $ is the set of leaves on the left branch above $ v $.  So it remains to check that $\mu_{pqi},  (\mu_{pqi} - \mu_{pqj})^{-1},  (\mu_{pqi} - \hbar^{-1} \nu_{qp})^{-1}$ are all regular functions on $ \CW_\tau$ (for $ i \in I$ and $j \notin I$).

First, note $\mu_{pqi} $ is regular on $ \CW_\tau$, since the meet of $ p $ and $i$ is above the meet of $ p $ and $ q$.  On the other hand 
$$ (\mu_{pqi} - \mu_{pqj})^{-1} = \frac{z_p - z_q}{z_i - z_j}  $$ 
is also regular on $ \CW_\tau$, since the meet of $ i $ and $ j$ is below vertex $ v $ (or $ j $ lies in a different tree).  Finally, 
$$
(\mu_{pqi} - \hbar^{-1} \nu_{qp})^{-1} = \frac{z_q - z_p}{\hbar^{-1} - z_i} = \hbar(\delta_{pi} - \delta_{qi})
$$
which is also regular on $ \CW_\tau$, since $ i, p, q $ all lie in the same tree.

Now suppose that $ v $ is a root. Again from Lemma \ref{lem:resudue-sum}, we see that 
$$
s_{l,v,\tau}^m \in  \C[\delta_{ip}, (\delta_{ip} - \delta_{jp})^{-1}, (\delta_{ip} - \hbar^{-1})^{-1}]_{i \in I, j \notin I} \otimes U
$$
where $ I $ is the set of leaves of the tree rooted at $ v$.  First $ \delta_{ip}$ are regular on $ \CW_\tau$ since $ i,p $ are on the same tree.  Second, we have
$$
(\delta_{ip} - \delta_{jp})^{-1} = \frac{ 1- \hbar z_p}{z_i - z_j} = \frac{1 - \hbar z_p}{z_j - z_p} \frac{ z_j - z_p}{z_i - z_j} = \nu_{pj}(\mu_{ijp} -1)
$$
which is also regular on $ \CW_\tau$ (since $ j $ lies in a different tree than $i,p$).  Finally 
$$
(\delta_{ip} - \hbar^{-1})^{-1} = - \hbar( 1- \hbar \delta_{pi})
$$
is also regular on $ \CW_\tau$ (here we use that $ (1 - \hbar \delta_{ip})(1- \hbar \delta_{pi}) = 1$).

Thus in both cases we conclude that $ s_{l,v,\tau}^m \in U \otimes \C[\CW_\tau] $ as desired.
\end{proof}

\begin{thm} \label{th:genCFn}
    For any point $ (\uz; \varepsilon) \in \CF_n$, the elements $ \{ s_{l,v,\tau}^m \} \cup \{ \widetilde \Phi_l^{(0)} \} $ freely generate $ \mathbfcal A^\varepsilon(\uz)$.
\end{thm}

\begin{proof}
    Suppose that $ \varepsilon \ne 0$.  Then by Prop \ref{pr:mathbfcalA}, $ \mathbfcal A^\varepsilon(\uz) = \ma(\varepsilon^{-1}, \uz) $ under the identification $ U_\varepsilon \fg \otimes U \fg^{\otimes n} = U \fg^{n+1}$.  Thus, by Theorem~\ref{m1}, $\mathbfcal{A}^\varepsilon(\uz)$ is freely generated by $S_{l,i}^m$ for $i=1,\ldots,n$, $l=1,\ldots, r$ and $m=1,\ldots,d_l$, along with $ S_{l,0}^{d_l} = \widetilde \Phi_l^{(0)}$.
     
     Let $ a \in \Cx, b \in \C$.  Since $ S_{l,i}^m$ is a coefficient in the Laurent expansion of $ S_l(u;z_1, \dots, z_n)$ at $ z_i$, we see that (up to a non-zero constant) $ S_{l,i}^m$ equals the residue 
     $$
     \Res_{w_i} w^{m-1} S_l(w; w_0,w_1, \ldots, w_n)
     $$
     where $ w = au + b$,  $w_0 = a \varepsilon^{-1} + b$, and $ w_i = a z_i + b$.

Now, recall the bijection between non-leaf vertices and leaves given by Lemma \ref{le:biject}.  By the definition of $ s_{l,v,\tau}^m $ and the above observation, we conclude that our collection $ \{ s_{l,v,\tau}^m  \}$ (indexed by non-leaf vertices) is related to the original generators $ \{ S_{l,i}^m\} $ (indexed by leaves) by an upper triangular linear transformation with non-zero diagonal entries.  Thus, our new collection $ \{s_{l,v,\tau}^m  \}$ is also a set of free generators.

The case $ \varepsilon = 0 $ is similar using Prop \ref{m2} instead of Theorem \ref{m1}.
\end{proof}

\subsection{Step 2. Operadic nature of the generators}
In \cite[\S 6.3]{iklpr}, we considered a stratification of $ \overline \CF_n$ indexed by pairs of set partitions; on each stratum certain of the $ \nu_{ij} $ coordinates are equal to 0 and others are equal to $ \infty$ (as $\delta_{ij} = \nu_{ij}^{-1}$, this is equivalent to setting some $ \delta_{ij} = 0$).

Let $ A $ be a codimension~1 stratum of $ \overline \CF_n$ such that $ A \cap \CW_\tau \ne \emptyset$.  From \cite[Remark 6.13]{iklpr},  we have the following 2 possibilities: \begin{enumerate}
    \item $A= \overline \CF_I \times_{\mathbb{A}^1} \overline \CF_J $, determined by a partition $[n]=I\sqcup J$ compatible the forest $\tau$, by putting $\nu_{ji}=0$ for all $i\in I,\ j\in J$;
    \item $A= \overline M_{I+1} \times \overline \CF_{J \cup \{I\}} $ determined by a subset $I\subset [n]$ with $ J = [n] \setminus I$, compatible with the forest $\tau$, by putting $\delta_{ij}= \infty $ for all $i, j \in I$.
\end{enumerate}

Let $A$ be a codimension~$1$ stratum of the first type. It corresponds to a partition of  $[n]$ into two subsets $I$ and $J$ such that the leaves of each tree from the forest $\tau$ lie entirely  either in $I$ or in $J$; this splits the forest $\tau$ into the union of two subforests $\tau_I$ and $\tau_J$. 
 Let $ n_I = |I|, n_J = |J|$.  The stratum $A$ is determined by the equations $\nu_{j i} = 0$ (and hence $ \nu_{ij} = \hbar $) for any $i\in I,\ j\in J$.   We have an isomorphism $ A \cong \overline \CF_I \times_{\BA^1} \overline \CF_J$ (by \cite[Prop. 6.12]{iklpr}) which we will write as $ C \mapsto (C_I, C_J) $.  A generic point of this stratum can be viewed a pair $ (\uz^I; \varepsilon), (\uz^J; \varepsilon)$.  Equivalently, when $ \varepsilon \ne 0 $, we can think of this point as a two component curve $C_I \cup C_J$ with $C_I$ containing marked points $ z_0 = \varepsilon^{-1} $ and $ z_i $ for $ i \in I$ and $ C_J $ containing $ z_j $ for $ j \in J $ along with $ z_{n+1} = \infty$.  The two curves are glued along the point $ \infty \in C_I $ and $ \varepsilon^{-1} \in C_J $.  (When $ \varepsilon = 0 $, we have a flower curve with two petals.)
 
 Moreover, $ A \cap \CW_\tau \cong \CW_{\tau_I} \times_{\BA^1} \CW_{\tau_J}$.  Thus, the inclusion $ A \cap \CW_\tau \subset \CW_\tau$ gives a restriction morphism
$$ \C[\CW_\tau] \rightarrow  \C[\CW_{\tau_I}] \otimes_{\C[\hbar]} \C[\CW_{\tau_J}]  $$
and thus a morphism
$$
\C[\CW_\tau] \otimes_{\C[\hbar]} U_\hbar \fg \otimes U \fg^{\otimes n} \rightarrow \C[\CW_{\tau_I}] \otimes_{\C[\hbar]} \C[\CW_{\tau_J}] \otimes_{\C[\hbar]} U_\hbar \fg \otimes U \fg^{\otimes n}
$$
For any vertex $ v \in \tau $, our generators $ s_{l,v,\tau}^m $ live in the left hand side, and by restricting them to $ A $, we obtain elements in the right hand side.

On the other hand, for any vertex $ v \in \tau_I$, we have $$ s_{l, v,\tau_I}^m \in \C[\CW_{\tau_I}] \otimes_{\C[\hbar]} U_\hbar \fg \otimes U \fg^{\otimes I}$$ and similarly for $ v \in \tau_J$.


Let $\mathcal{B}$ be the partition of $\{0,1,\ldots,n\}$ into the parts $\{0\}\cup I, \{j_1\},\ldots, \{j_{n_J}\}$.

\begin{lem}\label{lem:restriction-codim1-strata-1sttype}
Restricted to the stratum $ A \cap \CW_\tau$, we have
$$
s_{l,v,\tau}^m = \begin{cases}
    j_{\{0\} \cup I}^{n+1}(s_{l,v,\tau_I}^m) \quad \text{ if $v$ is in the subforest $\tau_I$} \\ 
    \Delta^{\mathcal{B}}(s_{l,v,\tau_J}^m) \quad \text{ otherwise.} 
\end{cases}
$$
Moreover the generators $ \{ s_{l,v,\tau}^m \} \cup \{ \widetilde \Phi^{(0)}_l \}$ freely generate the subalgebra $$j_{ \{0\} \cup I}^{n+1}(\mathbfcal{A}^{\varepsilon}(C_I))\otimes_{\Delta^{\mathcal{B}}(ZU\fg\otimes1^{\otimes q})}\Delta^{\mathcal{B}}(\mathbfcal{A}^{\varepsilon}(C_J))$$
\end{lem}

\begin{proof}
Let $ v $ be an internal vertex of $ \tau $. As before, let $ p,q $ be the consecutive vertices whose meet is $ v$. Then $ s_{l,v,\tau}^m$ is defined using the rational function
$$
S_l(w; \hbar^{-1} \nu_{qp}, \mu_{pq1}, \dots, \mu_{pqn}):= \ev_{w - \hbar^{-1} \nu_{qp}, w - \mu_{pq1}, \dots, w 
- \mu_{pqn}}(S_l)
$$

    Assume that $ v \in \tau_I$ be an internal vertex.   
Now, on $ A$, for $j \in J$, we have $ \mu_{pqj} = \infty$, since $ \mu_{pqj} = \nu_{jp}^{-1} \delta_{qp}^{-1}$ and $ \nu_{jp} = 0 $ on $ A$.  Thus
$$
\ev_{w - \hbar^{-1}\nu_{qp}, w - \mu_{pq1}, \dots, w 
- \mu_{pqn}} = j_{\{0\} \cup I}^{n+1} \circ \ev_{w - \hbar^{-1} \nu_{qp}, w - \mu_{p q i_1}, \dots, \mu_{p q i_{n_I} } }
$$
where $ \{i_1, \dots, i_{n_I} \} = I $.  

Thus $$S_l(w; \hbar^{-1} \nu_{qp}, \mu_{pq1}, \dots, \mu_{pqn}) = j_{\{0\} \cup I}^{n+1} S_l(w; \hbar^{-1} \nu_{qp}, \mu_{p  q i_1}, \dots, \mu_{p q i_{n_I}}) $$ and hence on $ A$, $s_{l,v,\tau}^m = j_{\{0\} \cup I}^{n+1}(s_{l,v,\tau_I}^m)$.

The case where $ v$ is a root of $ \tau_I $ is similar.

Now assume that $ v \in \tau_J $.  Let $ i \in I$.  Then on $ A$, $ \nu_{ip} = \hbar $ and so $\mu_{pqi} = \hbar^{-1} \nu_{qp}  $ is independent of $ i $.  Thus, 
$$
\ev_{w - \hbar^{-1} \nu_{qp}, w - \mu_{pq 1}, \dots, w 
- \mu_{pqn}} = \Delta^{\CB} \circ \ev_{w - \hbar^{-1} \nu_{qp}, w - \mu_{p q j_1 }, \dots, w 
- \mu_{p q j_{n_J} }}
$$
Thus, 
$$S_l(w; \hbar^{-1} \nu_{qp}, \mu_{pq1}, \dots, \mu_{pqn}) =  \Delta^{\CB} S_l(w; \hbar^{-1} \nu_{qp}, \mu_{pq j_1 }, \dots, \mu_{p q j_{n_J} }) $$
and hence on $ A $, $ s_{l,v,\tau}^m = \Delta^{\CB}(s_{l,v,\tau_J}^m)$ as desired.  Again the case where $ v $ is a root of $ \tau_J $ is similar.

Now it remains to show that the above generators stay algebraically independent. This follows from the particular case of Knop's theorem \cite{kn} (see \cite[Lemma 3.10]{r} for more details), namely $[U\fg^{\otimes I}]^\fg$ is a free $\Delta (ZU\fg)$-module and that the homomorphism $\Delta \cdot\Id:U\fg\otimes [U\fg^{\otimes I}]^\fg\to U\fg^{\otimes I}$ factors as $$U\fg\otimes [U\fg^{\otimes I}]^\fg\to U\fg\otimes_{ZU\fg} [U\fg^{\otimes p}]^\fg \hookrightarrow U\fg^{\otimes p}$$ 
This implies that the product of $j_{\{0\} \cup I}^{n+1}(\mathbfcal{A}^{\varepsilon}(C_I))$ and $\Delta^{\mathcal{B}}(\mathbfcal{A}^{\varepsilon}(C_J))$ is in fact the tensor product $$j_{\{0\} \cup I}^{n+1}(\mathbfcal{A}^{\varepsilon}(C_I))\otimes_{\Delta^{\mathcal{B}}(ZU\fg\otimes1^{\otimes J})}\Delta^{\mathcal{B}}(\mathbfcal{A}^{\varepsilon}(C_J))$$ 
On the other hand, the elements $j_{\{0\} \cup I}^{n+1}(s_{l,v,\tau_I}^m)$ for $v $ in the subforest $\tau_I$, along with $ \widetilde \Phi^{(0)}_l $, freely generate the first factor while  $\Delta^{\mathcal{B}}(s_{l,v,\tau_J}^m)$ for $v$ in the subforest $\tau_J$, along with $ \Delta^{\mathcal{B}}(\widetilde \Phi^{(0)}_l) \in \Delta^{\mathcal{B}}(ZU\fg\otimes1^{\otimes J}) $ freely generate the second factor.  In collecting the generators, we do not use $ \Delta^{\mathcal{B}}(\widetilde \Phi^{(0)}_l)$ which corresponds to the fact that they live in the algebra over which we are forming the tensor product.

\end{proof}

Finally, let $A$ be a codimension~$1$ stratum of the second type. It corresponds to a subset $I \subset [n]$ which are the leaves of a subtree $\tau_I$ of a tree from the forest $\tau$. Let $J = [n]\setminus I$. Let $\tau_{J \cup \{I\}}$ be the forest obtained from $\tau$ by contracting $\tau_I$.  Note that $ \tau_{J \cup \{I\}} $ has one leaf corresponding to the subtree $\tau_I$, so that the leaves of $ \tau_{J \cup \{I\}} $ are in bijection with $ J \cup \{ I \} $, where $ J = [n] \setminus I $.

The stratum $A$ is determined by the equations $ \delta_{ij} = 0 $ for any $ i, j \in I $.  This implies that $\mu_{ikj} = 0$ for any $i,j\in I,\ k\not\in I$. 

We have an isomorphism $ A \cong \overline M_{I+1} \times \overline \CF_{J \cup \{I\}}$ (by \cite[Prop. 6.12]{iklpr}), which we will write as $ C \mapsto (C_I, C_{J \cup \{I\}}) $. Generically, a point of this stratum can be viewed a pair $ \uz^I, (\uz^{J \cup \{I\}}, \varepsilon)$.  Equivalently, (when $ \varepsilon \ne 0 $) we can think of this point as a two component curve $C_I \cup C_{J \cup \{I\}}$ with $C_I$ containing marked points $ z_i $ for $ i \in I$ and $ C_{J \cup \{I\}} $ containing $ z_j $ for $ j \in J $ along with $ z_0 = \varepsilon^{-1}, z_{n+1} = \infty $.  The two curves are glued along the point $ \infty \in C_I $ and $ z_I \in C_J $.  

Moreover, $ A \cap \CW_\tau \cong U_{\tau_I} \times \CW_{\tau_{J \cup \{I\}}}$, where $ U_{\tau_I} \subset \overline M_{I+1} $ is the stratum constructed in section 2.3 of \cite{r}.  In the proof of Theorem 3.13 of \cite{r}, the third author constructed generators $ s_{l,v,\tau_I}^m$ for the Gaudin subalgebra $ \ma(C_I)$.

Thus, the inclusion $ A \cap \CW_\tau \subset \CW_\tau$ gives a restriction morphism
$$ \C[\CW_\tau] \rightarrow  \C[U_{\tau_I}] \otimes \C[\CW_{\tau_{J \cup \{I\}}}]  $$
and thus a morphism
$$
\C[\CW_\tau] \otimes_{\C[\hbar]} U_\hbar \fg \otimes U \fg^{\otimes n} \rightarrow \C[U_{\tau_I}] \otimes \C[\CW_{\tau_{J \cup \{I\}}}] \otimes_{\C[\hbar]} U_\hbar \fg \otimes U \fg^{\otimes n}
$$
For any vertex $ v \in \tau $, our generators $ s_{l,v,\tau}^m $ live in the left hand side and once again, we wish to determine their images in the right hand side.

Let $\mathcal{B}$ be the partition of $\{0,1,\ldots,n\}$ into the parts $\{0\},\{j_1\},\ldots,\{j_{n_J} \}, I$.

\begin{lem}\label{lem:restriction-codim1-strata-2ndtype}
Restricted to the stratum $ A \cap \CW_\tau$, we have
$$
s_{l,v,\tau}^m = \begin{cases}
    j_{\{0\} \cup I}^{n+1}(s_{l,v,\tau_I}^m) \quad \text{ if $v$ is in the subforest $\tau_I$} \\
    \Delta^{\mathcal{B}}(s_{l,v, \tau_{J \cup \{I\}}}^m) \quad \text{ otherwise.} 
\end{cases}
$$
 They freely generate the subalgebra $$j_I^{n+1}(\mathcal{A}(C_I))\otimes_{\Delta^{\mathcal{B}}(ZU\fg)}\Delta^{\mathcal{B}}(\mathbfcal{A}^{\varepsilon}(C_{J\cup \{I\}}))$$ 
     Here, the first factor in this tensor product is a homogeneous Gaudin subalgebra while the second factor is a trigonometric one for $\varepsilon\ne0$ and a universal inhomogeneous one for $\varepsilon=0$.
\end{lem}

\begin{proof} 
The proof is entirely similar to that of the previous Lemma.

\end{proof}

\subsection{Step~3. The family $\mathbfcal{A}^\varepsilon(C)$ is flat over $\overline{\mathcal{F}}_n$} 

Applying Lemmas \ref{lem:restriction-codim1-strata-1sttype} and \ref{lem:restriction-codim1-strata-2ndtype} inductively over the stratification immediately yields the following result.
\begin{prop}
    For any $ (C; \varepsilon) \in  \CW_\tau$, the elements $ \{ s_{l,v,\tau}^m \} \cup \{\widetilde \Phi^{(0)}_l\} $ are algebraically independent in $ U_\varepsilon \fg \otimes U \fg^{\otimes n}$.
\end{prop}

For any $ (\varepsilon, C) \in \overline \CW_\tau$, we let $ \mathbfcal A^\varepsilon(C) \subset U_\varepsilon \fg \otimes U \fg^{\otimes n}$ be the subalgebra generated by these elements.  Again applying Lemmas \ref{lem:restriction-codim1-strata-1sttype} and \ref{lem:restriction-codim1-strata-2ndtype} inductively implies that the algebra $ \mathbfcal A^\varepsilon(C)$ is independent of the choice of the forest $ \tau $.  

Equivalently, we begin with the subalgebras
$$
\C[\{ s_{l,v,\tau}^m \} \cup \{\widetilde \Phi^{(0)}_l\}  ] \subset \C[\CW_\tau] \otimes_{\C[\hbar]}  U_\hbar \fg \otimes U \fg^{\otimes n}
$$
and the above Lemmas show that these glue together to form the subalgebra
$$
\mathbfcal A \subset \C[\overline \CF_n] \otimes_{\C[\hbar]}  U_\hbar \fg \otimes U \fg^{\otimes n}
$$

For any $ (C; \varepsilon) \in  \overline \CF_n $, we define a filtration on $ \mathbfcal A^\varepsilon(C)$ as follows. 

\begin{prop} \label{pr:twofiltrations}
    The following two procedures define the same filtration on $\mathbfcal A^\varepsilon(C)$:
\begin{enumerate}
    \item Regarding $\mathbfcal A^\varepsilon(C) $ as a polynomial ring in the generators $ \{ s_{l,v,\tau}^m \} \cup \{\widetilde \Phi^{(0)}_l\} $ (for some $ \tau $ such that $(C;\varepsilon) \in \CW_\tau$), and then setting $ \deg s_{l,v,\tau}^m = d_l $ and $ \deg \widetilde \Phi^{(0)}_l = d_l$.
    \item Intersecting $\mathbfcal A^\varepsilon(C)$ with the PBW filtration on $ U_\varepsilon \fg \otimes U \fg^{\otimes n}$.
\end{enumerate}
\end{prop}

\begin{proof}
    It suffices to show that the leading terms of the generators $ \{ s_{l,v,\tau}^m \} \cup \{\widetilde \Phi^{(0)}_l\} $ with respect to the PBW filtration are algebraically independent elements of $S(\fg)^{\otimes n+1}$. Note that for $(C;\varepsilon) \in \CF_n$ and $ \varepsilon\ne0$, this follows from Proposition \ref{pr:GenCommGaudin}, similar to the proof of Theorem \ref{th:genCFn}.  When $ \varepsilon = 0 $, the argument is similar.

    Next, for degenerate $C$, we proceed by induction on the codimension of a stratum. According to Lemmas~\ref{lem:restriction-codim1-strata-1sttype}~and~\ref{lem:restriction-codim1-strata-2ndtype}, the restrictions of $\{ s_{l,v,\tau}^m \}$ are the generators of the same form for each of the components, embedded into $ U_\hbar \fg \otimes U \fg^{\otimes n}$ according to the subtree or subforest embedding. So the algebraic independence of the leading terms of $ \{ s_{l,v,\tau}^m \} \cup \{\widetilde \Phi^{(0)}_l\} $ follows from the argument entirely similar to that in the proof of Lemma~\ref{lem:restriction-codim1-strata-1sttype}. Namely, by \cite{kn}, $[S(\fg)^{\otimes p}]^\fg$ is a free $\Delta^{[p]}(S(\fg)^\fg)$-module and that the homomorphism $\Delta^{[p]}\cdot\Id:S(\fg)\otimes [S(\fg)^{\otimes p}]^\fg\to S(\fg)^{\otimes p}$ factors as $S(\fg)\otimes [S(\fg)^{\otimes p}]^\fg\to S(\fg)\otimes_{S(\fg)^\fg} [S(\fg)^{\otimes p}]^\fg \hookrightarrow S(\fg)^{\otimes p}$. This implies that the associated graded of the subalgebra generated by $ \{ s_{l,v,\tau}^m \} \cup \{\widetilde \Phi^{(0)}_l\} $ is the tensor product of those for smaller $n$ over their common subalgebra freely generated by some of these generators. Iterating this argument, we see that the leading terms of $ \{ s_{l,v,\tau}^m \} \cup \{\widetilde \Phi^{(0)}_l\} $ are algebraically independent for any $(C;\varepsilon)\in\mathcal{W}_\tau$.   
\end{proof}

Note that the above equivalence of two definitions of the filtration implies that it is independent of the choice of $ \tau$. From the algebraic independence of the generators, we immediately conclude the following.

\begin{corol}
    $\mathbfcal{A}^\varepsilon(C)$ has constant Hilbert-Poincar\'e series with respect to this filtration.  Thus this forms a flat family of subalgebras parametrized by $ \overline \CF_n$.
\end{corol}

\subsection{Step 4. The family $\mathbfcal{A}^\varepsilon_\chi(C)$ is flat}
Let $ \chi \in \fh^{reg}$.

In the previous section we construct the family of subalgebras 
$$\mathbfcal{A} \subset \C[\overline \CF_n] \otimes_{\C[\hbar]} (U_\hbar \fg \otimes U \fg^{\otimes n})^{\fg}$$  
Recall from section \ref{se:triginhom}, we have the homomorphism $$ \chi \circ \Rees \psi : (U_\hbar \fg \otimes U \fg^{\otimes n})^\fg \rightarrow \C[\hbar] \otimes U \fg^{\otimes n}$$  
Thus, we define 
$$
\mathbfcal{A}_\chi  := (\chi \circ \Rees \psi)( \mathbfcal{A}) \subset \C[\overline \CF_n] \otimes_{\C[\hbar]} \C[\hbar] \otimes  U \fg^{\otimes n}
$$
In particular for any $ (C; \varepsilon) \in \overline \CF_n$, we obtain a subalgebra $ \mathbfcal{A}_\chi^\varepsilon(C) \subset U \fg^{\otimes n}$.  

By construction, if $ \varepsilon \ne 0$, then $ \mathbfcal{A}_\chi^\varepsilon(C) $ is the (limit) trigonometric Gaudin algebra $ \ma_{\varepsilon^{-1} \chi}(C)$.

\begin{lem}
Fix a planar binary forest $ \tau$.  The set
$ \{ (\chi \circ \Rees \psi)(s_{l, \tau, v}^m) \} $ is algebraically independent in $ \C[\hbar] \otimes U \fg^{\otimes n}$.  They generate the subalgebra $ \mathbfcal{A}_\chi^\varepsilon(C)$ for any $ (C; \varepsilon) \in \CW_\tau$.
\end{lem}
\begin{proof}
By Lemma \ref{le:psiinjective}, the homomorphism $\chi \circ \Rees \psi : (U_\hbar \fg \otimes U\fg^{\otimes n})^\fg\to \C[\hbar] \otimes U\fg^{\otimes n}$ factors as follows: 
$$
(U_\hbar \fg \otimes U \fg^{\otimes n})^\fg \rightarrow (U_\hbar \fg/ Z_\hbar(\chi) \otimes U \fg^{\otimes n})^\fg \rightarrow \C[\hbar] \otimes U \fg^{\otimes n}
    $$ 
Since $ \{ s_{l, \tau, v}^m \} \cup \{ \widetilde \Phi_l^{(0)} \} $ are algebraically independent in $ (U_\hbar \fg \otimes U\fg^{\otimes n})^\fg $ and the latter set freely generates $ (U_\hbar \fg)^\fg $, we see that $ \{  s_{l, \tau, v}^m \}$ remain algebraically independent in $ (U_\hbar \fg/ Z_\hbar(\chi) \otimes U \fg^{\otimes n})^\fg $ (here we use that $ U_\hbar \fg / Z_\hbar(\chi) = U_\hbar \fg \otimes_{Z(U_\hbar \fg)} \C[\hbar]$). 

On the other hand, by Lemma \ref{le:psiinjective}, the homomorphism $(U_\hbar \fg/ Z_\hbar(\chi) \otimes U \fg^{\otimes n})^\fg \rightarrow \C[\hbar] \otimes U \fg^{\otimes n}$ is injective, so thus we conclude the algebraic independence in $ \C[\hbar] \otimes U \fg^{\otimes n} $.

The last assertion follows immediately.

\end{proof}

Define the filtration on $\mathbfcal{A}^\varepsilon_\chi(C)$ by either of the following procedures (whose equivalence follows from Proposition \ref{pr:twofiltrations}).
\begin{enumerate}
    \item Regarding $\mathbfcal A^\varepsilon_\chi(C) $ as a polynomial ring in the generators $ \{ (\chi \circ \Rees \psi)(s_{l,v,\tau}^m) \} $ (for some $ \tau $ such that $(C; \varepsilon) \in \CW_\tau$), and then setting $$ \deg (\chi \circ \Rees \psi)(s_{l,v,\tau}^m) = d_l $$
    \item Defining $\mathbfcal A_\chi^\varepsilon(C)^{(p)} = (\chi \circ \Rees \psi)(\mathbfcal A^\varepsilon(C)^{(p)})$
\end{enumerate}

From the algebraic independence of the generators, we immediately conclude the following which completes the proof of Theorem \ref{th:algCFn}. 
\begin{corol}
 $\mathbfcal{A}^\varepsilon_\chi(C)$ has constant Hilbert-Poincar\'e series with respect to this filtration.  Thus this forms a flat family of subalgebras parametrized by $ \overline \CF_n$.
\end{corol}

Moreover, we also complete the proof of Theorem \ref{compinhom}.
\begin{corol} \label{co:compinhom}
For any $ C \in \overline F_n$, the subalgebra $ \mathbfcal A^0_\chi(C) $ coincides with the subalgebra $\ma_\chi(C) $ constructed in section \ref{secinhom}.  Thus $ \ma_\chi(C) $ forms a flat family of subalgebras faithfully parametrized by $ \overline F_n$.
\end{corol}

\begin{proof}
We simply set $ \varepsilon = 0 $ and compare the iterative construction from Lemmas \ref{lem:restriction-codim1-strata-1sttype} and \ref{lem:restriction-codim1-strata-2ndtype}  to the construction in section \ref{secinhom}.
\end{proof}

\begin{rem}
    
    Note that the above filtration on $\mathbfcal{A}^0_\chi(C)$ is \emph{not} the one induced from the PBW filtration on $U\fg^{\otimes n}$.  For example, the $PBW$-quadratic part consists of two types of generators, namely, the KZ and dynamical Gaudin Hamiltonians; however with respect to the image of the PBW filtration on the universal inhomogeneous Gaudin subalgebra, only KZ Gaudin Hamiltonians are quadratic, and the dynamical ones have bigger degrees.  If we use the PBW filtration on the algebras $ \ma_\chi(z_1, \dots, z_n) $, we will get a different compactified parameter space. 
\end{rem}

\section{Spectra of Gaudin models.}
\subsection{Simple spectrum property.}
In this section, we will recall and prove certain theorems about the simplicity of the spectrum for Gaudin models. To prove the simplicity of the spectrum we use the following method:
\begin{itemize}
    \item Show the existence of a cyclic vector for the Gaudin algebra;
    \item Show that the Gaudin algebra acts by normal operators with respect to a positive-definite Hermitian form.
\end{itemize}
The second property implies semisimplicity of action, and combined with the first property, we deduce the simplicity of the spectrum.
Let $V(\ul) = V(\lambda_1) \otimes \ldots \otimes V(\lambda_n)$ be the tensor product of irreducible finite-dimensional representations of $\fg$ with highest weights $\lambda_1, \ldots, \lambda_n$. 
Let $V(\ul)^{\fn_+} \subset V(\ul)$ be the subspace of singular vectors. For any weight $\mu $, let $V(\ul)_{\mu} \subset V(\ul)$ be the subspace of vectors of weight $\mu$.
Concerning the cyclic vector, we have the following results:
\begin{itemize}
\item For the homogeneous Gaudin model, $\ma(\uz)$ acts with a cyclic vector on $V(\ul)^{\fn_+}$.
\item For the inhomogeneous Gaudin model, if $\chi \in \fg^{reg}$, then $\ma_{\chi}(\uz)$ acts with a cyclic vector on $V(\ul)$.
\item For the trigonometric Gaudin model, if $\theta \in \fh^{reg}$ is generic, then  $\ma_{\theta}^{trig}(\uz)$ acts with a cyclic vector on $V(\ul)$.
\end{itemize}
Moreover, it is possible to generalize the statements above to the limit subalgebras, see below.
Concerning the semisimplicity of the action, we have the following results (here, $\fh^{split}$ is the Cartan subalgebra of the split real form of $\fg$, and $\fh^{comp}=i\fh^{split}$ is the compact real form):
\begin{itemize}
\item For the homogeneous Gaudin model, if $z_i \in \br$ then $\ma(z_1, \ldots, z_n)$ acts by normal operators on $V(\ul)^{\fn_+}$.
\item For the inhomogeneous Gaudin model, if $\chi \in \fh^{split}$ and all $z_i \in \br$, then $\ma_{\chi}(z_1, \ldots, z_n)$ acts by normal operators on $V(\ul)$.
\item For the trigonometric Gaudin model we have two cases where we know semisimplicity.  \begin{enumerate} 
\item If  $\theta \in (\fh^{reg})^{split}$ is generic and $z_i \in \br$, then $\ma_{\theta}^{trig}(\uz)$ acts by normal operators on $V(\ul)$ (with respect to non-standard form, see below). 
\item If  $\theta \in (\fh^{reg})^{comp}$ is generic, $\theta + \frac{1}{2} \mu \in \fh^{comp}$ and $\overline{z_i} = z_i^{-1}$ for all $i$, then   $\ma_{\theta}^{trig}(\uz)$ acts by normal operators on $V(\ul)_{\mu}$ with respect to standard Hermitian form on $V(\ul)$.
\end{enumerate}

\end{itemize}

Again it is possible to generalize these statements to the appropriate limit subalgebras (corresponding to a certain real form).
\subsection{Spectrum of homogeneous Gaudin model.}
Consider the standard Hermitian form on $V(\lambda)$ uniquely determined (up to a constant factor) by the property $f_i^+=e_i,\ h_i^+=h_i$, where $+$ denotes the Hermitian conjugation. 
Let $V(\ul) = V(\lambda_1) \otimes \ldots \otimes V(\lambda_n)$ be the tensor product of irreducible finite-dimensional representations of $\fg$ with highest weights $\lambda_1, \ldots, \lambda_n$. We consider $V(\ul)$ with the standard Hermitian form which is the product of standard Hermitian forms on each component. Let $V(\ul)^{\fn_+} \subset V(\ul)$ be the subspace of singular vectors.
\begin{thm}\label{th:old-cyclic-homogeneous} (\cite{r4})
For any $C \in \overline M_{n+1}$, the subalgebra $\ma(C)$ acts with a cyclic vector on $V(\ul)^{\fn_+}$. In particular, it acts with a cyclic vector on any Hom-space $\Hom_\fg(V(\nu),V(\ul))$.
\end{thm}

\begin{thm}[\cite{ffry}]
\label{l1}
Suppose that $C \in \overline M_{n+1}^{split}$. Then the subalgebra $\ma(C)$ acts by normal operators on $V(\ul)$.
\end{thm}

From the previous two theorems, it follows that the family $\ma(C), C \in \overline M_{n+1}^{split}$ has a simple spectrum on any $V(\ul)^{\fn_+}$.

We will need a generalization of Theorem \ref{l1}.
Let $\theta \in \fh^{split}$.
Consider the Hermitian version of the Shapovalov form on $M(\theta)$, i.e. the Hermitian form such that for the highest vector $v_{\theta} \in M(\theta)$ we have
$$(v_{\theta}, v_{\theta}) = 1$$
and, for any $v,w\in M(\theta)$, we have
$$ (f_i \cdot v, w) = (v, e_i \cdot w),\ \ (h_i \cdot v, v) = (v, h_i \cdot v).$$
Note that the Shapovalov form for finite-dimensional representation coincides with the standard Hermitian form.
Define the Hermitian form on $M(\theta) \otimes V(\lambda_1) \otimes \ldots \otimes V(\lambda_{n-1})$ as the product of the above Shapovalov form on $M(\theta)$ and standard Hermitian forms on $V(\lambda_i), i = 1, \ldots, n-1$. 
\begin{prop}
The restriction of the above Hermitian form to $(M(\theta) \otimes V(\lambda_1) \otimes \ldots \otimes V(\lambda_{n-1}))^{\fn_+}$ for big enough $\theta \in \fh^{split}$ is positive definite.
\end{prop}
\begin{proof}
For any $\theta$, the tensor product 
$M(\theta) \otimes V(\lambda_1) \otimes \ldots \otimes V(\lambda_{n-1})$ is is an object from the category $\mathcal{O}$ which has a standard filtration with the subquotients being the Verma modules with the highest weights $\theta+\mu$, for all weights $\mu$ of $V(\lambda_1) \otimes \ldots V(\lambda_{n-1})$. It follows that any singular vector belongs to a finite sum of weight spaces. On each weight space, the above form degenerates for $\theta$ from a finite union of hyperplanes parallel to the root hyperplanes \cite{shap}. On the other hand, for integral dominant $\theta$ such that for any positive root $\alpha$, the value of $(\theta,\alpha)$ is big enough, the above Hermitian form coincides with the standard one on the weight space in the tensor product of finite dimensional $\fg$-modules, so is positive definite. This means that for any $\theta\in\fh^{split}$ such that for any positive root $\alpha$, the value of $(\theta,\alpha)$ is big enough, this form is positive definite as well.  
\end{proof}
Using the proposition above we can now state the following

\begin{prop}
\label{t2}
For any $C \in \overline M_{n+1}^{split}$ and big dominant $\theta \in \fh^{split}$, the
subalgebra $\ma(C)$ acts by normal operators on $(M(\theta) \otimes V(\lambda_1) \otimes \ldots \otimes V(\lambda_{n-1}))^{\fn_+}$.
\end{prop}
\begin{proof}
Note that the condition on operators to be normal is a closed condition hence it is enough to check the statement for $\ma(0,z_1, \ldots, z_n)$ with $z_1, \ldots, z_n \in \br$. From the proposition above we have a positive defined Hermitian form on $(M(\theta) \otimes V(\lambda_1) \otimes \ldots \otimes V(\lambda_{n-1}))^{\fn_+}$. 
From Proposition~\ref{pr:A-invariance-sigma}, we know that the universal Gaudin subalgebra $\ma_\fg$ is invariant under the anti-involution $f_i[k] \mapsto e_i[k]$ and $h_i[k] \to h_i[k]$, so for any $x \in \ma(C)$ we have $x^{+} \in \ma(C)$, i.e. all the operators coming from $\ma(C)$ are normal with respect to the Hermitian form.
\end{proof}

\subsection{Spectra of inhomogeneous Gaudin model.}
For the inhomogeneous Gaudin model, the existence of cyclic vector and conditions for simplicity of spectra are also known.

\begin{thm}[\cite{ffry}] \label{th:old-cyclic}
Suppose that $\chi \in \fg^{reg}$. Then the subalgebra $\ma_{\chi}(\uz)$ acts with cyclic vector on $V(\ul)$ for any $\uz=(z_1,\ldots,z_n)$ with $z_i\ne z_j$.
\end{thm}

\begin{thm}[\cite{ffry}] \label{th:old-simple-spec}
Suppose that $\chi \in \fh^{split}$ is regular and $z_i \in \br$. Then subalgebras $\ma_{\chi}(\uz)$ act by normal operators and hence with simple spectrum on $V(\ul)$.
\end{thm}

One can extend the above theorems to the compactified parameter space $\overline F_n$. Consider the natural real form of the space $\overline F_n$ (where all functions $ \nu_{ij}, \mu_{ijk} $ take on real values; equivalently, its points are real cactus flower curves). We denote this space by $\overline F_{n}^{split}$.

\begin{thm}\label{th:cyclic-ss-inhomogeneous} For any $\chi\in\fh^{reg}$ and for any $C\in \overline{F}_n$, the subalgebra $\ma_\chi(C)$ acts on $V(\ul)$ with a cyclic vector.
If moreover $\chi\in \fh^{split}$, then $\ma_{\chi}(C)$ acts with simple spectrum on $V(\ul)$.
\end{thm}
\begin{proof} 
The existence of a cyclic vector follows from the explicit description of the subalgebras $\ma_\chi(C)$ corresponding to the boundary points of $\overline{F}_n$, as given in section \ref{secinhom}. Suppose first that the underlying flower curve has only one petal. Then the subalgebra $\ma_\chi(C)$ is the product of $\Delta^{\mathcal{B}}\ma_\chi(\underline{u})$ (for some partition $\mathcal{B}$ of $\{1,\ldots,n\}$ into $m$ parts $B_1,\ldots, B_m$ and $\underline{u}=(u_1,\ldots,u_m)$ being a collection of pairwise distinct points on the line) and $j_{B_i}^n\ma(C_i)$ (for $i=1,\ldots,m$). So it preserves the decomposition $V(\ul)=\bigoplus W_\ul^{\underline{\nu}}\otimes V(\underline{\nu})$ with respect to $\Delta^{\mathcal{B}} U\fg^{\otimes m}$. Then according to Theorem~\ref{th:old-cyclic}, $\Delta^{\mathcal{B}}\ma_\chi(\underline{u})$ acts on each of the $V(\underline{\nu})$ with some cyclic vector $\xi_{\underline{\nu}}$. Next, the multiplicity space $W_\ul^{\underline{\nu}}$ is the tensor product of $\Hom_\fg(V({\nu_i}),\bigotimes\limits_{j\in B_i}V(\lambda_j))$ that is acted on by $j_{B_i}\ma(C_i)$ with some cyclic vector $\eta_{\nu_i}$ according to Theorem~\ref{th:old-cyclic-homogeneous}. So the sum of all the vectors $\eta_{\nu_1}\otimes\ldots\otimes\eta_{\nu_m}\otimes\xi_{\underline{\nu}}$ over all $\underline{\nu}$ is a cyclic vector for $\ma_\chi(C)$.

Now suppose $C$ has several petals and $C_1,\ldots,C_k$ are the cactus curves growing from each of them. Let $B_i$ be the set of indices of all the marked points on $C_i$. Then $\ma_\chi(C)$ is the product of $j_{B_i}^n \ma_\chi(C_i)$ acting each on its own $\bigotimes\limits_{j\in B_i} V(\lambda_j)$. So the cyclic vector is just the tensor product of cyclic vectors for these actions of inhomogeneous Gaudin algebras corresponding to one-petal curves, and the cyclicity is proved.

Finally, the property of $\ma_\chi(C)$ acting by normal operators on $V(\ul)$ is a closed condition on $C$ satisfied on $F_n^{split}$, so it holds on $\overline{F}_n^{split}$ as well. 
\end{proof}

\subsection{Cyclicity property for trigonometric Gaudin model.} Fix a regular element $\chi\in\fh^{reg}$ and the finite-dimensional tensor product representation $V(\ul)$.
\begin{prop}
\label{cyc}
For all except finitely many $\theta \in \mathbb{C}^\times \cdot \chi $, all the subalgebras $\ma_{\theta}^{trig}(C), C \in \overline M_{n+2}$ act with a cyclic vector on $V(\ul)$. 
\end{prop}
\begin{proof}
Note that if $\theta=\varepsilon^{-1}\chi$ then $\ma_{\theta}^{trig}(C)=\mathbfcal{A}^\varepsilon_\chi(C)$, so it suffices to prove the statement for $\mathbfcal{A}^\varepsilon_\chi(C)$. Let $Y$ be the subset of $\overline{\mathcal{F}}_n$ consisting of such $(C; \varepsilon)$ that there is no cyclic vector in $V(\ul)$ for the corresponding subalgebra $\mathbfcal{A}^\varepsilon_\chi(C)$. This set is Zariski closed. The map
$$\varphi:\overline{\mathcal{F}}_n \to \ba^1,\ (C; \varepsilon)\mapsto \varepsilon$$
is proper, hence the image $\varphi(Y)$ is Zariski closed as well. According to Theorem~\ref{th:cyclic-ss-inhomogeneous}, $\mathbfcal{A}^0_\chi(C)=\ma_{\chi}(C)$ has a cyclic vector on any tensor product for all $C\in \overline{F}_n$, so it follows that $\varphi(Y)$ is not the whole $\mathbb{A}^1$. As a result, $\varphi(Y)$ is a finite set and the statement follows.
\end{proof}

\begin{corol}
    Once $\ma_\theta^{trig}(C)$ acts on $V(\ul)$ semisimply, it has a simple spectrum on $V(\ul)$.
\end{corol}

\begin{rem}
Our proof gives us some description of what kind of generic set we have:
for example, on each line inside $\fh^{reg}$ there are only finitely many possible values of $\theta$ for which there may be no cyclic vector. We expect that a cyclic vector in fact exists for any $\theta \in \fh, C \in \overline M_{n+2}$.
\end{rem}

\subsection{Split real form and simplicity of spectrum for the trigonometric Gaudin model.}

Fix a regular real element $\chi\in (\fh^{split})^{reg}$ and the finite-dimensional tensor product representation $V(\ul)$.

\begin{thm}\label{th:split-ss}
The family of trigonometric Gaudin subalgebras $\ma_{\theta}^{trig}(C)$ with $C \in \overline M_{ n+2}^{split}$ and big enough $(\theta >> 0)$ generic dominant $\theta \in \br\chi\subset (\fh^{split})^{reg}$ acts with simple spectrum on $V(\ul)$.
\end{thm}
\begin{proof}
Indeed, by Proposition~\ref{cyc}, there is a cyclic vector in $V(\ul)$ for $\ma_\theta^{trig}(C)$ for all $C\in\overline{M}_{n+2}$ and generic $\theta$ in a linear span of any regular element in $\fh^{split}$, so it holds for all $C\in\overline{M}_{n+2}$ and sufficiently large $\theta \in \br\chi\subset (\fh^{split})^{reg}$. So we have to prove that these subalgebras act semisimply
.
Using Proposition~\ref{t2} we see that for big enough $\theta \in \fh^{split}$ the Gaudin subalgebra $\mathcal{A}(0,z_1,\ldots,z_n)$ with $z_i\in\br$ acts by normal operators with respect to the restriction of the standard Hermitian form on $M(\theta)\otimes V(\ul)$ to the space $(M(\theta)\otimes V(\ul))_{\theta+\mu}^{\fn_+}$. The identification $(M(\theta)\otimes V(\ul))^{\fn_+}_{\theta+\mu}\simeq V(\ul)_{\mu}$ gives a Hermitian form on $V(\ul)_{\mu}$ (note that this form is not the standard form on $V(\ul)$!) such that all the trigonometric Gaudin subalgebras $\ma_{\theta}^{trig}(C)$ with $C \in \overline M_{n+2}^{split}$ act by normal operators, hence semisimply. 

\end{proof}

\subsection{Compact real form and simplicity of spectrum for trigonometric Gaudin model.}

Recall the Cartan antilinear anti-involution $\sigma:\fg\to\fg, e_i\mapsto -f_i, f_i\mapsto -e_i, h_i\mapsto -h_i$. The compact real form $\fg^\sigma=\fg^{comp}$ of the Lie algebra $\fg$ is the fixed point set for this involution so that $\fh^{comp}=\fg^{comp}\cap \fh=i\fh^{split}$.

\begin{prop}
    Suppose that $z_i^{-1}=\overline{z_i}$ for any $i = 1, \ldots, n$ and $\theta \in \fh$ such that $\theta-\frac{1}{2}\mu\in \fh^{comp}$, the algebra $\ma_\theta^{trig}(z_1, \ldots, z_n)$ acts on $V(\ul)_{\mu}$ by normal operators with respect to the standard Hermitian form on $V(\ul)$.
\end{prop}

\begin{proof} The condition that $\ma_\theta^{trig}(z_1, \ldots, z_n)$ acts on $V(\ul)_{\mu}$ by normal operators with respect to a fixed Hermitian scalar product is a closed condition on the parameters $\theta$ and $z_i$, so we can always assume that these parameters are generic.
Denote by $\pi_\ul^\mu$ the homomorphism from $(U\fg^{\otimes n})^\fh$ to ${\rm End} \, V(\ul)_\mu$ and denote by $M(\theta)^*$ the restricted dual of the Verma module $M(\theta)$. From the projective invariance property of homogeneous Gaudin subalgebras, we have
\begin{lem}
    $\pi_\ul^\mu(\ma^{trig}_{\theta}(z_1,\ldots,z_n))=\pi_\ul^\mu(\ma^{trig}_{-\theta-\mu}(z_1^{-1},\ldots,z_n^{-1}))$.
\end{lem}
\begin{proof}
    The weight space $V(\ul)_\mu=\text{Hom}_{\fg}(M(\theta+\mu),M(\theta)\otimes V(\ul))$ is naturally a module over $(U\fg^{\otimes(n+2)}/U\fg^{\otimes(n+2)}\Delta(\fg))^\fg$, with $\pi_\ul^\mu(\ma^{trig}_{\theta}(z_1,\ldots,z_n))$ being identified with the image of $\ma(u_\infty,u_0,u_1,\ldots,u_n)$, where $u_\infty,u_0,u_i$ are the images of $\infty,0,z_i$, respectively, under a projective transformation of $\bp^1$. The latter vector space can be regarded as a  $(U\fg^{\otimes(n+1)}/U\fg^{\otimes(n+1)}\Delta(\fn_+))^{\fn_+}$-module in two equivalent ways.
    \begin{enumerate}
        \item Identifying the Hom-space $\text{Hom}_{\fg}(M(\theta+\mu),M(\theta)\otimes V(\ul))$ with $(M(\theta)\otimes V(\ul))^{\fn_+}_{\theta}$. The image of $\ma(u_\infty,u_0,u_1,\ldots,u_n)$ is then the image of the Gaudin algebra $\ma(0,z_1,\ldots,z_n)$ which is, by definition, $\pi_\ul^\mu(\ma^{trig}_{\theta}(\uz))$.
        
        \item Identifying the Hom-space $\text{Hom}_{\fg}((M(\theta+\mu),M(\theta)\otimes V(\ul))$ with the \emph{coinvariants space} $(M(\theta+\mu)^*\otimes V(\ul))_{\fn_+}^{-\theta}$. The image of $\ma(u_\infty,u_0,u_1,\ldots,u_n)$ is then the image of the Gaudin algebra $\ma(0,z_1^{-1},\ldots,z_n^{-1})$ which factors through $\ma^{trig}_{-\theta-\mu}(z_1^{-1},\ldots,z_n^{-1})$, since $-\theta-\mu$ is the lowest weight of $M(\theta+\mu)^*$, so it is $\pi_\ul^\mu(\ma^{trig}_{-\theta-\mu}(z_1^{-1},\ldots,z_n^{-1}))$.
    \end{enumerate}
\end{proof}

\begin{lem} We have
    $\pi^\mu_{\ul}(\ma^{trig}_{\theta}(z_1,\ldots,z_n))^+=\pi^\mu_{\ul}(\ma^{trig}_{-\sigma(\theta)}(\overline{z}_1,\ldots,\overline{z}_n))$
\end{lem}

\begin{proof}
    The Hermitian dual space to $V(\ul)_\mu=\text{Hom}_{\fg}(M(\theta+\mu),M(\theta)\otimes V(\ul))$ is the coinvariant space $V(\ul)_\mu^*=(M(\theta+\mu)\otimes M(\theta)^*\otimes V(\ul)^*)_\fg$ with the $(U\fg^{\otimes(n+2)}/U\fg^{\otimes(n+2)}\Delta(\fg))^\fg$-action twisted by the antiautomorphism $s\circ\sigma^{\otimes(n+2)}$ with $s$ being the antipode. The latter space then is isomorphic to $(M(-\sigma(\theta+\mu))^*\otimes M(-\sigma(\theta))\otimes V(\ul))_\fg$ as a $(U\fg^{\otimes(n+2)}/U\fg^{\otimes(n+2)}\diag(\fg))^\fg$-module. Since $$s\circ\sigma^{\otimes(n+2)}(\ma(z_\infty,z_0,z_1,\ldots,z_n))=\ma(\overline{z_\infty},\overline{z_0},\overline{z_1},\ldots,\overline{z_n}),$$ we have $\ma^{trig}_{\theta}(z_1,\ldots,z_n)^+=\ma^{trig}_{-\sigma(\theta)}(\overline{z_1},\ldots,\overline{z_n})$.
\end{proof}
Combining the assertions of the above Lemmas, we get 
\begin{gather*}
    \pi^\mu_{\ul}(\ma^{trig}_{\theta}(z_1,\ldots,z_n))^+=\pi^\mu_{\ul}(\ma^{trig}_{-\sigma(\theta)}(\overline{z_1},\ldots,\overline{z_n}))= \\ 
    =\pi^\mu_{\ul}(\ma^{trig}_{-\sigma(\theta)}(z_1^{-1},\ldots,z_n^{-1}))=\pi^\mu_{\ul}(\ma^{trig}_{\sigma(\theta)-\mu}(z_1,\ldots,z_n))=\pi^\mu_{\ul}(\ma^{trig}_{\theta}(z_1,\ldots,z_n)).
\end{gather*}
\end{proof}

\begin{thm}\label{th:trig-ss}
The family of trigonometric Gaudin subalgebras with $C \in \overline M_{n+2}^{comp}$ and generic $\theta \in \fh$ such that $\theta - \frac{1}{2} \mu \in \fh^{comp}$ acts with simple spectrum on $V(\ul)_{\mu}$ for any weight $\mu$ of $V(\ul)$.
\end{thm}
\begin{proof}
From the above discussion, we have 
\begin{itemize}
\item Cyclic vector for generic $\theta$.
\item Semisimplicity of action for $\theta-\frac{1}{2} \mu \in \fh^{comp}$.
\end{itemize}
The combination of these two facts implies the statement of the theorem.
\end{proof}

\bigskip

\noindent\footnotesize{
{\bf Aleksei Ilin} \\
Higher School of Modern Mathematics, MIPT, Russia \\
HSE University, Moscow, Russia \\
{\tt alex.omsk2@gmail.com}}\\

\noindent\footnotesize{
{\bf Joel Kamnitzer} \\
Department of Mathematics and Statistics, McGill University, Montreal QC, Canada \\
{\tt joel.kamnitzer@mcgill.ca}} \\

\noindent\footnotesize{
{\bf Leonid Rybnikov} \\
Department of Mathematics and Statistics,
University of Montreal, Montreal QC, Canada\\
{\tt leo.rybnikov@gmail.com}}
\end{document}